\documentclass[final,11pt]{article}

\usepackage{graphicx,caption,xcolor}
\usepackage{hyperref}
\usepackage{mathrsfs, mathenv, amsmath, amsthm, amssymb, amsfonts, amscd}
\usepackage[margin=1in]{geometry}
\usepackage[square,numbers]{natbib}
\usepackage{setspace}
\usepackage{xcolor}
\usepackage{subcaption}
\usepackage[linesnumbered,ruled,vlined]{algorithm2e}
\usepackage[capitalise]{cleveref}
\usepackage{soul}
\usepackage{xparse}

\usepackage{tikz,graphicx}
\usetikzlibrary{patterns}
\usetikzlibrary{calc}
\usetikzlibrary{angles}
\usetikzlibrary{quotes}

\theoremstyle{plain}
\newtheorem{theorem}{Theorem}[section]

\newtheorem{lemma}[theorem]{Lemma}
\newtheorem{property}[theorem]{Property}

\theoremstyle{remark}

\definecolor{darkred}{rgb}{0.5,0,0}
\definecolor{darkgreen}{rgb}{0,0.5,0}

\newcommand{\vect}[1]{\boldsymbol{\mathbf #1}}
\newcommand{\grad}{\vect \nabla}
\newcommand{\dup}[2]{\left\langle#1,#2\right\rangle}

\newcommand{\ip}[2]{\left(#1,#2\right)}

\newcommand{\laplacian}{\Delta}
\newcommand{\dsurf}{\mathrm d \sigma}
\newcommand{\lp}[2]{L^{#1}({#2})}
\newcommand{\norm}[1]{\left\|#1\right\|}

\newcommand{\Span}[1]{\text{span}\{#1\}}

\newcommand{\abs}[1]{\left|#1\right|}

\newcommand{\seq}[4]{{\left\{#1_{#2}\right\}}_{#2=#3}^{#4}}
\newcommand{\real}{\mathbb R}
\newcommand{\nat}{\mathbb N}

\newcommand{\energyA}[0]{\frac{1}{\varepsilon}}
\newcommand{\energyB}[0]{\varepsilon}

\renewcommand{\d}[0]{\mathrm{d}}
\renewcommand{\land}[0]{\textrm{ and }}
\renewcommand{\lor}[0]{\textrm{ or }}

\makeatletter
\DeclareDocumentCommand\derivative{s m O{} m}{%
    \IfBooleanTF{#1}{\def\@der{\mathrm{d}}}{\def\@der{\partial}}
    \mathchoice{%
        \frac{
            \@der\ifnum\pdfstrcmp{#2}{1}=0\else^{#2}\fi {#3}
        }{%
            \@for\@token:={#4}\do{\@der \@token}
        }
    } {%
        \partial_{#4} #3
    } {} {}
}
\makeatother

\begin{document}

\title{A linear, second-order, energy stable, fully adaptive finite-element method for phase-field modeling of wetting phenomena}
\author{B. Aymard $^1$, U. Vaes $^{1,2}$, M. Pradas $^3$ and S. Kalliadasis $^1$}

\maketitle

{%
    \noindent \scriptsize \centering
    $^1$ Department of Chemical Engineering, Imperial College London, London SW7 2AZ, United Kingdom \\
    $^2$ Department of Mathematics, Imperial College London, London SW7 2AZ, United Kingdom \\
    $^3$ School of Mathematics and Statistics, The Open University, Milton Keynes MK7 6AA, United Kingdom \\
}

\begin{abstract}
\setlength{\parskip}{6pt}
We propose a new numerical method to solve the Cahn-Hilliard equation coupled with non-linear wetting boundary conditions.
We show that the method is mass-conservative and that the discrete solution satisfies a discrete energy law similar to the one satisfied by the exact solution.
We perform several tests inspired by realistic situations to verify the accuracy and performance of the method:
wetting of a chemically heterogeneous substrate in three dimensions, wetting-driven
nucleation in a complex two-dimensional domain and three-dimensional
diffusion through a porous medium.

\end{abstract}

\section{Introduction}

Capillarity and wetting phenomena, driven primarily by interfacial forces,
are ubiquitous in a wide spectrum of natural phenomena and technological applications.
Examples range from the wetting of plant leaves by rainwater and insects walking on water to coating processes,
inkjet printing, oil recovery and microfluidic devices;
for reviews, see e.g.~\cite{de1985wetting,WettingSpreading}.
From an historical point of view,
two of the concepts essential to the understanding of capillarity and wetting were introduced and studied already in 1805:
these are the Laplace pressure~\cite{Laplace1805} and the Young--Dupr\'e contact angle~\cite{Young1805}.
Later, following the work of Plateau on soap films~\cite{Plateau1873},
Poincar\'e~\cite{hp1895ca} linked interfacial phenomena with the theory of minimal surfaces.

Wetting phenomena typically involve
a fluid-fluid interface advancing or receding on a solid substrate and
a contact line formed at the intersection between the interface and the substrate.
The wetting properties of the substrate determine to a large extent the behavior of the fluids in the contact-line region,
and in particular the contact angle at the three-phase conjunction,
defined as the angle between the fluid-fluid interface and the tangent plane at the substrate.
At equilibrium, this is precisely the Young--Dupr\'e angle.
When one of the two fluids moves against the other,
the contact angle becomes a dynamic quantity, and
when the problem is formulated in the framework of conventional hydrodynamics,
the contact line motion relative to the solid boundary results in the notorious stress singularity there,
as first noted in the pioneering studies by Moffat~\cite{Moffat1963} and Huh and Scriven~\cite{HuhScriven}.
Since then there have been numerous analyses and discussions of the singularity over the years,
see e.g.~\cite{DussanDavis,Hocking1983,Cox1986} and also recent studies in~\cite{Davidfifteena,David2015b}
(with the latter one revisiting the classical Cox--Hocking matched asymptotic analysis
and providing a correction to it).
Recently, it was shown~\cite{andreasThesis} that mesoscopic approaches such as dynamic density functional theory (DDFT)
can fix this singularity behavior.

A popular model for interface dynamics is the Cahn--Hilliard (CH) equation~\cite{cahn1961spinodal,cahn1958free},
which belongs to the class of phase-field and diffuse interface models.
Originally proposed to model spinodal decomposition,
the mechanism by which a binary mixture can separate to form two coexisting phases due to,
e.g., a change of temperature~\cite{cahn1958free},
it has been used in a wide spectrum of different contexts since,
such as solidification phenomena~\cite{PFMsolidifaction} and  Saffman--Taylor instabilities in Hele--Shaw flows~\cite{InterfaceRougheningInHeleShawFlows}.
To account for wetting phenomena and contact lines on solid boundaries,
the CH equation can be coupled to a wall boundary condition~\cite{CahnSeventySeven}.
Such CH model has been employed successfully in various situations,
including microfluidic devices~\cite{demenech2006modeling,demenechtransition,Queralt_et_al,Wylock2012},
flow in porous media~\cite{bogdanov2010pore},
rheological systems~\cite{boyer2004numerical},
and patterning of thin polymer films~\cite{kim2006three}.
Other potential applications include micro-separators~\cite{roydhouse2014operating},
fuel cells~\cite{benzigerwater} and CPU chip cooling based on electro-wetting \cite{chipCooling}.
Many of these applications are characterized by the presence of chemically heterogeneous substrates and/or complex geometries,
which make their numerical simulation challenging.

The form of the wetting boundary condition is dictated by the form of the wall free energy.
For liquid-gas problems linear forms have been adopted,
e.g. in the pioneering study by Seppecher~\cite{Seppecher} and in~\cite{briantlattice, XuQian2010}.
But a cubic is the lowest-order polynomial required such that
the wall free energy can be minimized for the bulk densities,
and it prevents the formation of boundary layers on the wall.
Cubic forms have been adopted for binary fluid problems, e.g. in~\cite{jacqmin2000contact,YueZhouFeng},
but also for liquid-gas ones, see~\cite{davidthirteena,davidthirteenb}.
The latter studies, in particular,
showed asymptotically that
a CH model can alleviate the contact line discontinuity without any additional physics
(and at the same time completing but also correcting Seppecher's work).
The detailed asymptotic analysis of the unification of binary-fluid CH models can be found in~\cite{davidthirteenc}.

Various approaches have been proposed in the literature for the numerical solution  of the CH equation.
Because of the high order of the equation and its multiscale features
(scale separation between interface size and the characteristic length),
most existing time-stepping schemes are implicit or semi-implicit.
Several of these schemes aim to satisfy discrete mass and energy laws
in agreement with the underlying continuum model.
Discretization in space can be achieved using finite-difference methods~\cite{furihata2001stable,kimJunseok2011},
finite element methods~\cite{barrett1999finite,elliott1992error,tucker2013},
spectral methods~\cite{weng2017fourier}, or pseudospectral methods~\cite{MR3606246}.
In addition, the computation time can be reduced by applying adaptive mesh
refinement~\cite{yue2006phase,dtAdapt} and time-step adaptation~\cite{guillen2014second}.

Among the several linear schemes for the CH equation with homogeneous Neumann boundary conditions introduced in~\cite{guillen2013linear},
the authors showed by means of numerical experiments that
their second-order optimal dissipation scheme, referred to as OD2,
is the most accurate and the one introducing the least numerical dissipation.
In this work, we outline a numerical scheme that extends and appropriately generalizes OD2 as follows:
(a) it includes a non-linear wetting boundary condition;
(b) it adopts an efficient energy-based time-step adaptation strategy.
In contrast with the time-adaptation scheme introduced in~\cite{guillen2014second},
where the time step is adapted to limit numerical dissipation,
we base the time-step adaptation directly on the variation of free energy.
With this method we are able to solve the CH system efficiently and systematically
to capture wetting phenomena in both two- and three-dimensional (2D and 3D, respectively) settings,
and in a wide range of situations, including confinement with complex geometry,
chemical and topographical heterogeneities, or both.

Like the OD2 scheme on which it was based,
the time-stepping scheme we propose is semi-implicit and linear.
We show that it is mass-conservative and satisfies a discrete free-energy law with a numerical dissipation term of order 2 in time.
Space discretization is achieved using a finite element method,
leading to an unsymmetrical sparse linear system to solve at each iteration.
In addition to adapting the time step as mentioned above,
with the aim of increasing the resolution in time during fast phenomena,
we use a mesh refinement strategy to capture interfaces precisely.

To test the efficiency of the proposed numerical scheme we consider several wetting problems as test cases.
We first study relaxation towards equilibrium in two situations:
the spreading of a sessile droplet and the coalescence of two sessile droplets on a flat,
chemically homogeneous substrate.
We then consider two-component systems in complex geometries delimited by chemically heterogeneous substrates in both 2D and 3D.

In \cref{sec:mathematical_model},
we introduce the CH system and the non-linear wetting boundary condition.
In \cref{sec:numerical_method_cahn-hilliard},
we outline our numerical scheme and prove the associated conservation properties.
In \cref{sec:simulations},
we present the results of several numerical experiments.
Conclusions and perspectives for future work are offered in \cref{sec:phase_field:conclusions}.

\section{Phase-field model for wetting phenomena}
\label{sec:mathematical_model}

Throughout this study, $\Omega \subset \real^d$ denotes a $d$-dimensional domain,
$\partial \Omega$ denotes its boundary with outward unit normal vector $\textbf{n}$,
$\Gamma_S$ is the solid substrate and $\Gamma_G = \partial \Omega \setminus \Gamma_S$.
The CH system we use to describe the  dynamics of two immiscible fluids in contact with a solid substrate,
is a free-energy-based model.
The starting point is the introduction of a locally conserved field,
denoted by $\phi: \Omega \rightarrow \real$,
that plays the role of an order-parameter:
two equilibrium values, say $+1$ and $-1$, represent the pure phases,
and the interface is conventionally located at the points where $\phi = 0$~\cite{cahn1961spinodal, cahn1958free}.
We consider systems with a free energy given by
\begin{align*}
    E(\phi) &:= E_m(\phi) + E_w(\phi) \\
            &:= \int_{\Omega} \left( \frac{1}{\varepsilon}F_m(\phi) + \varepsilon \frac{| \grad \phi |^2}{2} \right) \,\d \Omega + \int_{\partial\Omega} F_w(\phi)\,\dsurf,
\end{align*}
where the two terms, $E_m$ and $E_w$,
represent the mixing and wall components of the free energy, respectively.
Here $F_m(\phi) = \frac{1}{4}(\phi^2 - 1)^2$ and $F_w$ is taken to be a cubic polynomial,
following e.g.~\cite{davidthirteena,davidthirteenb}:
\begin{equation}
        \label{eq:cahn-hilliard:wallEnergy}
        F_w(\phi) = \frac{\sqrt{2}}{2}\cos \theta(\vect x) \left(\frac{\phi^3}{3} - \phi \right),
\end{equation}
where $\theta = \theta(\vect x)$ is the equilibrium contact angle,
which can depend on the spatial position $\vect x$.
From the expression of the free energy,
we calculate that, for a sufficiently smooth function $\psi: \Omega \rightarrow \real$:
\begin{align*}
    \derivative*{1}{\alpha} E(\phi + \alpha \, \psi) \big|_{\alpha = 0}
    &= \int_{\Omega} \left( \frac{1}{\varepsilon} f_m(\phi) - \varepsilon \laplacian \phi\right) \psi \, \d \Omega
    + \int_{\partial \Omega} (f_w(\phi) + \varepsilon \grad{\phi} \cdot \vect n) \, \psi \, \dsurf,
\end{align*}
with $f_m = F_m'$ and $f_w = F_w'$,
so the chemical potential is equal to
\begin{equation*}
         \mu := \frac{\delta E}{\delta \phi} = \frac{1}{\varepsilon} f_m(\phi) - \varepsilon \,\laplacian \phi,
\end{equation*}
and the natural boundary condition associated with the surface energy is
\begin{equation}
    \label{eq:wettingBC}
    \varepsilon \grad \phi \cdot \vect{n} = - f_w(\phi) = \frac{\sqrt{2}}{2}\cos\theta(\vect{x})(1-\phi^2).
\end{equation}
We assume that the dynamics of the system is governed by the CH equation,
\begin{equation}
        \label{eq:cahn-hilliard:CH}
        \derivative{1}[\phi]{t} = \grad \cdot (b(\vect x) \, \grad \mu),
\end{equation}
where $b(\vect x)$ is a mobility parameter, assumed to be uniform hereafter.
This leads to the following mass-conservation property:
\begin{equation}
    \label{Mlaw}
    \derivative*{1}{t} M(\phi) := \derivative*{1}{t} \int_{\Omega} \phi \, \d \Omega = \int_{\partial \Omega} b \, \grad{\mu} \cdot \vect n \, \dsurf,
\end{equation}
so the mass flux at the boundary can be specified using the condition $b \grad{\mu} \cdot \vect n = \dot m(\vect x)$,
where $\dot m(\vect x)$ is the desired mass flux.
In particular, we will set $\dot m(\vect x) = 0$ at the solid boundary, $\Gamma_S$.
In summary, the equations we are solving in this study are the following:
\begin{subequations}
    \begin{align}
        \label{eq:summary_system_phi} \derivative{1}[\phi]{t} &= \grad \cdot (b \, \grad \mu), \\
        \label{eq:summary_system_mu}  \mu &= \frac{1}{\varepsilon} f_m(\phi) - \varepsilon \, \laplacian \phi \quad &\text{for $x \in \Omega$, $t \in (0,T]$},\\
        \label{eq:summary_system_boundary_condition_phi} \varepsilon \grad \phi \cdot \vect n &= -f_w(\phi), \\
        \label{eq:summary_system_boundary_condition_mu}  b \, \grad \mu \cdot \vect n &= \dot m(\vect x) \quad &\text{for $x \in \partial\Omega$, $t \in (0,T]$}.
    \end{align}
\end{subequations}
In addition to the conservation of mass, \cref{eq:summary_system_phi,eq:summary_system_mu,eq:summary_system_boundary_condition_phi,eq:summary_system_boundary_condition_mu}
imply the following energy-conservation law, involving the phase field and the chemical potential:
\begin{equation}
  \label{eq:energy_conservation_cahn_hilliard}
  \derivative*{1}{t}E(\phi(t)) =
  - \|\sqrt{b}\,\grad \mu \|_{L^2(\Omega)}^2
  + \int_{\partial \Omega} \dot m \, \mu \,\dsurf.
\end{equation}
An advantage of the cubic surface energy~\eqref{eq:cahn-hilliard:wallEnergy} over other surface energy formulations
(see~\cite{huang2015wetting} for a review of wetting boundary conditions for binary fluids)
is that the well-known hyperbolic tangent profile is an equilibrium solution in more than 1 dimension.
Specifically, the function
\begin{equation}
    \label{eq:solution_easy_case}
    \phi(\vect x) = \tanh \left( \frac{\vect x \cdot \vect u}{\sqrt 2 \, \varepsilon}\right), \quad \text{where } \vect{u} = (\pm\sin\theta, \cos\theta)^T
\end{equation}
is solution to the CH equation posed in the half plane $\{y \geq 0\}$ with
the boundary condition \eqref{eq:cahn-hilliard:wallEnergy} at $\{ y = 0\}$
and constant $\theta(\vect x) = \theta$.
A schematic representation of this solution
and the corresponding fluid-fluid interface is given
in~\cref{fig:contact_angle_illustration}.
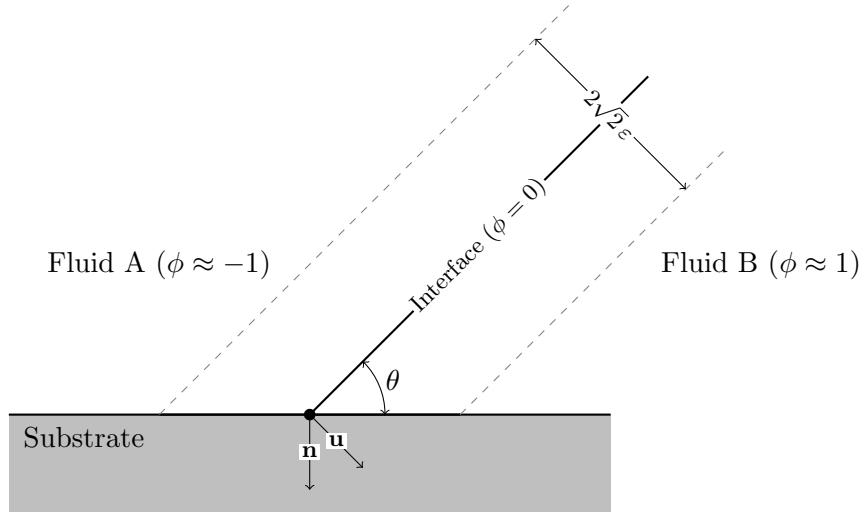
\begin{figure}[htb]
    \centering{%
        \begin{tikzpicture}
            \tikzstyle{ann} = [fill=white,font=\footnotesize,inner sep=1pt]
            \coordinate (origin) at (0,0);
            \fill[lightgray] ($ (origin) + (-4,-1.3) $) rectangle ($ (origin) + (4,0.0) $);
            \draw[thick] ($ (origin) + (-2,0) $) -- ($ (origin) + (2,0) $);
            \draw[black, thick] (origin) -- ++(4.5,4.5) node [midway, ann, sloped] {Interface ($\phi = 0$)};;
            \draw[black, thick] (-4,0) -- (4,0);
            \draw[gray, dashed] (-2,0) -- ++(5.5,5.5);
            \draw[gray, dashed] (2,0) -- ++(3.5,3.5);

            \draw[arrows=<->] ($ (4,4) + (-1, 1) $) -- ($ (4,4) + (1, -1) $) node [midway, ann, sloped] {$2\sqrt{2} \, \varepsilon$};

            \draw[arrows=->] (0,0) -- ++(0,-1) node [midway, ann] {$\vect n$};
            \draw[arrows=->] (0,0) -- ( $1/sqrt(2)*(1,-1)$ ) node[midway, ann] {$\vect u$};

            \filldraw[black] (0,0) circle (2pt);
            \node (substrate) at (-3, -0.3) {Substrate};
            \node (fluid1) at (-2,2) {Fluid A ($\phi \approx -1$)};
            \node (fluid2) at (6,2) {Fluid B ($\phi \approx 1$)};
            \draw
                (1,0) coordinate (a) -- (0,0) coordinate (b)
                -- (1,1) coordinate (c)
                pic["$\theta$", draw=black, <->, angle eccentricity=1.2, angle radius=1cm]
                {angle=a--b--c};
        \end{tikzpicture}
    }
    \caption{Schematic of the profile geometry of a fluid-fluid interface intersecting a solid boundary and
    illustration of the stationary solution~\eqref{eq:solution_easy_case}.}
    \label{fig:contact_angle_illustration}
\end{figure}

A drawback of the cubic wall energy~\eqref{eq:cahn-hilliard:wallEnergy} is that the conservation of energy
no longer seems to imply stability bounds for the solution,
making it impossible to use the tools traditionally employed (see e.g. \cite{elliott1996cahn}) to
prove the existence of a solution.
Indeed, an application of the trace inequality gives only that,
under appropriate regularity assumptions on $\phi$:
\begin{subequations}
    \begin{align*}
        \norm{\phi^3}_{\lp{1}{\partial \Omega}} &\leq C \, (\norm{\phi^3}_{\lp{1}{\Omega}} + \norm{\grad (\phi^3)}_{\lp{1}{\Omega}}) \\
                                                &= C \, (\norm{\phi^3}_{\lp{1}{\Omega}} + 3\norm{\phi^2 \, \grad \phi}_{\lp{1}{\Omega}}) \\
                                                &\leq C \, (\norm{\phi^3}_{\lp{1}{\Omega}} + \frac{3}{2 \alpha}\norm{\phi^4}_{\lp{1}{\Omega}} + \frac{3\alpha}{2}\norm{\grad \phi}_{\lp{2}{\Omega}}^2) \quad \forall \alpha > 0,
    \end{align*}
\end{subequations}
where we used H\"older's inequality and Young's inequality with a parameter.
Therefore, the wall energy cannot be controlled by the mixing energy for arbitrary domains.
This issue can be remedied by a simple modification of the wall energy outside of the physical range $[-1; 1]$;
instead of~\eqref{eq:cahn-hilliard:wallEnergy}, we consider the following wall energy:
\begin{equation}
    \label{eq:modified_wall_energy}
    F_w^*(\phi) = \frac{\sqrt{2}}{2} \cos \theta(\vect x) \times
    \begin{cases}
        \left(\frac{2}{3} - (\phi + 1)^2 \right) \quad &\text{ if }\phi < - 1; \\
        \left(\frac{\phi^3 -3 \phi}{3} \right) \quad &\text{ if }\phi \in [-1, 1]; \\
        \left(- \frac{2}{3} + (\phi - 1)^2\right) \quad &\text{ if } \phi > 1, \\
    \end{cases}
\end{equation}
This function is such that $F_w^*(\phi) = F_w (\phi)$ for $\phi \in [-1, 1]$, $F_w^* \in C^2(\real)$, and $(F_w^*)''$ is absolutely continuous,
which makes it possible to prove the second order convergence of our time-stepping scheme, see \cref{sec:numerical_method_cahn-hilliard}.
Another possibility would have been to choose constant values for $F_w^*$ outside of the interval $[-1, 1]$,
but this would have lead to $F_w^*$ being only $C^1(\real)$,
making it more difficult to show second order convergence theoretically.
The weak formulation of
\cref{eq:summary_system_phi,eq:summary_system_mu,eq:summary_system_boundary_condition_phi,eq:summary_system_boundary_condition_mu}
with the modified wall energy~\eqref{eq:modified_wall_energy}
is as follows: find $(\phi, \mu)$ such that
\begin{align*}
    \phi \in \lp{\infty}{0,T; H^1(\Omega)}, \quad \derivative{1}[\phi]{t} \in \lp{2}{0,T;(H^1(\Omega))'}, \quad \mu \in \lp{2}{0,T;H^{1}(\Omega)},
\end{align*}
and the following variational formulation is satisfied:
\begin{subequations}
    \begin{align}
        \label{eq:variational_equations_for_weak_formulation_phi}
        & \dup{\textstyle \derivative{1}[\phi]{t}}{\psi} + \ip{b \, \grad \mu}{\grad \psi} = \ip{\dot m}{\psi}_{\partial \Omega} \, \quad & \forall \psi \in H^1(\Omega) \text{ and a.e. } t,\\
        \label{eq:variational_equations_for_weak_formulation_mu}
        & \ip{\mu}{\nu} = \energyB\, \ip{\grad \phi}{\grad \nu} + \energyA \ip{f_m(\phi)}{\nu}  + \ip{f_w^*(\phi)}{\nu}_{\partial \Omega}  \quad & \forall \nu \in H^1(\Omega) \text{ and a.e. } t,
    \end{align}
\end{subequations}
with $f_w^* := (F_w^*)'$ and where $\langle\cdot,\cdot\rangle$, $(\cdot,\cdot)$ and $(\cdot, \cdot)_{\partial \Omega}$ denote respectively
the duality pairing between $(H^1(\Omega))'$ and $H^1(\Omega)$,
the standard inner product in $\lp{2}{\Omega}$,
and the standard inner product in $\lp{2}{\partial \Omega}$.
For simplicity of notations,
the symbols $F_w$, $f_w$ and $E$ will refer in the rest of this paper to
$F_w^*$, $f_w^*$, and $E_m + \int_{\partial \Omega} F_w^* \, \dsurf$, respectively.

\subsection{Existence of a solution}
\label{sub:well_posedness_of_the_model}
We can show the following existence result
for the weak formulation of the Cahn--Hilliard system with the modified boundary condition presented above,
under appropriate regularity assumptions for the initial condition and the mass flux $\dot m$.
\begin{theorem}
    \label{theorem:well-posedness_modified_boundary_condition}
    Assume that $\phi_0 \in H^1(\Omega)$ and $\dot m \in C([0,T];\lp{2}{\partial\Omega})$.
    Then there exists a pair of functions $(\phi,\mu)$, with
    \begin{enumerate}
        \item $\phi \in L^{\infty}(0,T; H^1(\Omega)) \bigcap C([0, T]; \lp{2}{\Omega})$,
        \item $\derivative{1}[\phi]{t} \in \lp{2}{0,T;(H^1(\Omega))'}$,
        \item $\phi(0) = \phi_0$,
        \item $\mu \in \lp{2}{0,T;H^1(\Omega)}$,
    \end{enumerate}
    that solve the variational formulation~\cref{eq:variational_equations_for_weak_formulation_phi,eq:variational_equations_for_weak_formulation_mu}.
\end{theorem}
\begin{proof}
    See \cref{sec:proof_of_well_posedness_theorem}.
\end{proof}

\section{Numerical method}%
\label{sec:numerical_method_cahn-hilliard}

In this section we introduce a new time-stepping scheme to solve the CH
equation (\ref{eq:cahn-hilliard:CH}) with the non-linear wetting boundary
condition~\eqref{eq:wettingBC}, which is  a generalization of the optimal
dissipation scheme of order 2, OD2, developed in~\cite{guillen2013linear}. We
decided to extend this particular scheme because the authors
of~\cite{guillen2013linear} showed that, among all the linear schemes they
proposed, it is the most accurate and the least dissipative. And in selected
test cases, they showed that for a large enough time step, it is the only
scheme that leads to the correct equilibrium solution. We refer to our scheme
as OD2-W, with W denoting wetting, and show that it leads to a consistent
discrete energy law.

We also develop a new adaptive time-stepping strategy which, combined with
adaptation in space, leads to a fully adaptive finite element method. An
excellent introduction to the finite element method and corresponding mixed
formulations can be found in~\cite{Ciarlet} and to mesh generation
and adaptive refinement in~\cite{Frey}.

\subsection{OD2-W scheme}

In this section, we assume for simplicity that $\dot m = 0$ and that $\theta$ is uniform on $\partial \Omega$.
We denote by $\Delta t$ the time step,  and by $\phi^n$ and $\mu^{n+1/\alpha}$
the numerical approximations of $\phi$ and $\mu$ at times $t^n$ and $t^n + \frac{1}{\alpha}\Delta t$, respectively.
To define a discretization in time of the CH system appropriate for wetting phenomena,
we follow the approach proposed in~\cite{guillen2013linear} to design an optimal dissipation scheme,
and consider the following generic implicit-explicit numerical scheme:
given $\phi^n \in H^1(\Omega)$, find $(\phi^{n+1}, \mu^{n+\frac{1}{\alpha}}) \in H^1(\Omega) \times H^1(\Omega)$ such that,
\begin{subequations}
    \label{VF}%
    \begin{align}
        \label{VF_phi}
        \ip{\delta_t \phi^{n+1}}{ \psi} &+ \ip{b \, \grad \mu^{n+\frac{1}{\alpha}}}{\grad \psi} = 0  \quad \quad & \forall \psi \in H^1(\Omega),\\
        \label{VF_mu}
        \ip{\mu^{n + \frac{1}{\alpha}}}{ \nu} &= \energyB\,\ip{\grad \phi^{n + \frac{1}{\alpha} + \beta}}{ \grad \nu} \nonumber \\
        & \quad + \energyA \ip{\hat f_m(\phi^n,\phi^{n+1})}{\nu} + \ip{\hat f_w(\phi^n, \phi^{n+1})}{\nu}_{\partial\Omega} \quad & \forall \nu \in H^1(\Omega).
    \end{align}
\end{subequations}
In these expressions, $\hat f_m, \hat f_w$ are functions to be specified,
linear in their second argument.
The parameter $\alpha \in \{1,2\}$ determines the accuracy of the numerical scheme,
and the parameter $\beta \in [0, 1 - 1/\alpha]$ controls the numerical diffusion.
The function $\phi^{n + \frac{1}{\alpha} + \beta}$ is defined by linear interpolation
between $\phi^n$ and $\phi^{n+1}$,
\begin{equation*}
    \phi^{n + \frac{1}{\alpha} + \beta} := \left(1 - \frac{1}{\alpha} - \beta\right) \, \phi^n + \left(\frac{1}{\alpha} + \beta\right) \, \phi^{n+1},
\end{equation*}
and $\delta_t \phi^{n+1}$ is the approximation of the time derivative of $\phi$ given by
\begin{equation*}
    \frac{\phi^{n+1} - \phi^n}{\Delta t}.
\end{equation*}
In most numerical experiments presented in this chapter, we consider the case $(\alpha,\beta) = (2,0)$ (OD2-W),
but we note that other usual choices include $(\alpha,\beta) = (1,0)$ (OD1-W) and $(\alpha,\beta) = (2, \mathcal O(\Delta t))$ (OD2mod-W).
By taking $\psi = \mu^{n+\frac{1}{\alpha}}$ and $\nu = \delta_t \phi^{n+1}$ in \eqref{VF},
we obtain
\begin{equation}
   \label{eq:energy_law_of_the_numerical_scheme}
   \frac{E(\phi^{n+1}) - E(\phi^n)}{\Delta t} = - \|\sqrt{b} \, \grad \mu^{n+\frac{1}{\alpha}} \|_{\lp{2}{\Omega}}^2 - ND(\phi^n, \phi^{n+1}), \quad n = 0, 1, \ldots,
\end{equation}
where $ND(\phi^{n},\phi^{n+1})$,
representing the non-physical numerical dissipation introduced by the time-stepping scheme,
can be broken down in three parts:
\begin{equation*}
    ND(\phi^n, \phi^{n+1}) = \varepsilon ND_{\textrm{philic}}(\phi^n,\phi^{n+1}) + \frac{1}{\varepsilon} ND_{\textrm{phobic}}(\phi^n,\phi^{n+1}) + ND_{\textrm{wall}}(\phi^n,\phi^{n+1}),
\end{equation*}
with
\begin{subequations}
\label{eq:numerical_dissipation}
\begin{align}
    ND_{\textrm{philic}}(\phi^n,\phi^{n+1}) &= \Delta t \int_{\Omega} \, \left[\frac{1}{\alpha} - \frac{1}{2} + \beta \right] \, \abs{\grad \delta_t \phi^{n+1}}^2  \, \d \Omega, \\
    ND_{\textrm{phobic}}(\phi^n,\phi^{n+1}) &= \int_{\Omega} \left( \hat f_m(\phi^n,\phi^{n+1})\, \delta_t  \phi^{n+1} - \frac{1}{\Delta t}(F_m(\phi^{n+1}) - F_m(\phi^n)) \right) \, \d\Omega, \\
    ND_{\textrm{wall}}(\phi^n,\phi^{n+1})   &= \int_{\partial \Omega} \left( \hat f_w(\phi^n, \phi^{n+1}) \, \delta_t \phi^{n+1} - \frac{1}{\Delta t} (F_w(\phi^{n+1}) -  F_w(\phi^{n})) \right) \, \dsurf.
\end{align}
\end{subequations}
Notice that the philic dissipation is always nonnegative, with
$ND_{\textrm{philic}}(\cdot, \cdot) = 0$ if $(\alpha,\beta) = (2,0)$ (OD2-W),
$ND_{\textrm{philic}}(\cdot, \cdot) = \mathcal O (\Delta t^2)$ if $(\alpha,\beta) = (2,\mathcal O(\Delta t))$ (OD2mod-W),
and $ND_{\textrm{philic}}(\cdot, \cdot) = \mathcal O (\Delta t)$ if $(\alpha,\beta) = (1,0)$ (OD1-W).
The two other terms can be expanded using Taylor's formula,
taking into account that $F_m$ is a polynomial of degree 4
and using the integral form of the remainder:
\begin{subequations}
    \begin{align}
        \nonumber
        ND_{\textrm{phobic}}(\phi^n,\phi^{n+1}) &= \int_{\Omega}\delta_t\phi^{n+1}\Big( \hat f_m(\phi^n,\phi^{n+1}) - f_m(\phi^n) - \frac{1}{2} \, \Delta t\, f_m'(\phi^n) \, \delta_t \phi^{n+1} \\
            \label{eq:expression_of_nd_phobic}
        &- \frac{1}{6} \, \Delta t^2\, f_m''(\phi^n) \, (\delta_t\phi^{n+1})^2 - \frac{1}{24} \, \Delta t^3\, f_m'''(\phi^n) \, (\delta_t\phi^{n+1})^3\Big) \, \d\Omega, \\
            \nonumber
            ND_{\textrm{wall}}(\phi^n,\phi^{n+1}) &= \int_{\partial\Omega}\delta_t\phi^{n+1}\Big(\hat f_w(\phi^n,\phi^{n+1}) - f_w(\phi^n) - \frac{1}{2} \, \Delta t\, f_w'(\phi^n) \, \delta_t \phi^{n+1} \Big) \\
            \label{eq:expression_of_nd_wall}
            &- \frac{1}{2\, \Delta t} \, \int_{\phi^n}^{\phi^{n+1}}f_w''(\phi)\,(\phi - \phi^n)^2 \, \d \phi \, \dsurf.
    \end{align}
\end{subequations}
This suggests the following choices for the functions $\hat f_m$ and $\hat f_w$:
\begin{subequations}
    \begin{align}
        \label{eq:definiton_hat_fm}
        \hat f_m(\phi^n, \phi^{n+1})
        &= f_m(\phi^n) + \frac{1}{2} \, \Delta t\, f_m'(\phi^n) \, \delta_t \phi^{n+1} \nonumber \\
        &=  \left[\frac{3}{2} (\phi^n)^2 \phi^{n+1} - \frac{1}{2} (\phi^n)^3 - \frac{\phi^{n} + \phi^{n+1}}{2}\right],  \\
        \nonumber
        \hat f_w(\phi^n, \phi^{n+1})
        &= f_w(\phi^n) + \frac{1}{2} \, \Delta t\, f_w'(\phi^n) \, \delta_t \phi^{n+1} \\
        \label{eq:definiton_hat_fw}
         &= - \frac{\sqrt{2}}{2}\cos(\theta) \times \begin{cases}
             2 + \phi^n + \phi^{n+1}
             \quad &\text{ if }\phi^n < - 1; \\
             1 - \phi^n \phi^{n+1}
             \quad &\text{ if }\phi^n \in [-1, 1] \\
             2 - \phi^n - \phi^{n+1}
             \quad &\text{ if } \phi^n > 1, \\
         \end{cases}\\
        &= - \frac{\sqrt{2}}{2}\cos(\theta) \times \big( 1 + (1 -\phi^n) \wedge 0 + (1 + \phi^n) \wedge 0 - (-1 \vee \phi^n \wedge 1)  \, \phi^{n+1} \big) \nonumber,
    \end{align}
\end{subequations}
where the last expression is convenient for programming purposes.
We note that
this methodology to derive a second-order scheme can be applied \emph{mutatis mutandis} when using the unmodified wall energy~\eqref{eq:cahn-hilliard:wallEnergy},
although we haven't been able to prove that the weak formulation admits a solution in that case.
Doing so leads to $\hat f_w(\phi^n, \phi^{n+1}) = -(\sqrt{2}/2) \, \cos{\theta} \, (1 - \phi^n \, \phi^{n+1})$,
which coincides with~\eqref{eq:definiton_hat_fw} when $\phi^n \in [-1, 1]$.
In either case, we have the following property:
\begin{property}
    Assume that $\alpha = 2$ and $\beta = 0$.
    Then the numerical dissipation term in~\cref{eq:energy_law_of_the_numerical_scheme} is such that
    \begin{equation*}
        |ND(\phi(t^n),\phi(t^{n+1}))| \leq C \Delta t^2,
    \end{equation*}
    with $C := \left(C_1 \, \|\phi\|_{C([0,T],L^{\infty}(\Omega))} \, \|\partial_t \phi\|_{C([0,T],L^3(\Omega))}^3 + C_2 \, \| \partial_t\phi\, \|_{C([0,T],L^3(\partial \Omega))}^3 \right)$,
    provided that all the terms in the definition of $C$ are well-defined.
\end{property}
\begin{proof}
In \cite{guillen2013linear}, the authors show that:
\begin{equation*}
    \begin{aligned}
        &ND_{\textrm{philic}}(\cdot,\cdot) = ND_{\textrm{philic}}(\phi^n,\phi^{n+1}) = ND_{\textrm{philic}}(\phi(t^n),\phi(t^{n+1})) = 0; \\
        & \abs{ND_{\textrm{phobic}}(\phi(t^n),\phi(t^{n+1}))} \leq \Delta t^2 \, \left(C_1 \, \|\phi\|_{C([0,T],L^{\infty}(\Omega))} \, \|\partial_t \phi\|^3_{C([0,T],L^3(\Omega))}\right).
    \end{aligned}
\end{equation*}
For the wall term, we obtain from \cref{eq:expression_of_nd_wall,eq:definiton_hat_fw}:
\begin{equation*}
    \begin{aligned}
        \abs{ND_{\textrm{wall}}(\phi(t_n),\phi(t_{n+1})} &\leq C_2 \, \Delta t^2 \int_{\partial\Omega} \abs{\delta_t\phi^{n+1}}^3  \dsurf \\
                                              &\leq C_2 \, \Delta t^2 \, \| \partial_t\phi \|^3_{C([0,T],L^3(\partial \Omega))}.
    \end{aligned}
\end{equation*}
\end{proof}


In addition to the energy law~\eqref{eq:energy_law_of_the_numerical_scheme},
the numerical scheme \eqref{VF} satisfies a discrete version of the conservation law \eqref{Mlaw} presented in \cref{sec:mathematical_model},
which can be seen by choosing $\psi = 1$ in \cref{VF_phi}.

\begin{property}
    The numerical solution satisfies the following mass conservation law:
\begin{equation*}
    \int_{\Omega} \phi^{n} \, \d\Omega = \int_{\Omega} \phi^0 \, \d\Omega
    \quad \text{ for } n = 0, 1, 2, \dots
\end{equation*}
\end{property}

\subsubsection{Space discretization and adaptive mesh refinement}
Our approach for mesh adaptation is based on a method proposed in \cite{meshAdaptation},
and implemented through the FreeFem++ functions \emph{adaptmesh} (in 2D) and \emph{mshmet} (in 3D).
The idea of the method is to define a metric on the computational domain based on the solution at the current time step,
and to use for the next time step a mesh that is uniform in that metric.
The metric we consider corresponds to the following metric tensor,
depending only on the phase field $\phi$:
\begin{equation*}
    G(\vect x) = R(\vect x)\, \text{diag}(\tilde \lambda_i(\vect x))\, R(\vect x)^T, \qquad
    \tilde{\lambda}_i(\vect x) = \min \left( \max \left( \frac{1}{\gamma} |\lambda_i(\vect x)|,\frac{1}{h^2_{\mathrm{max}}}\right),\frac{1}{h^2_{\mathrm{min}}}\right),
\end{equation*}
where $(\lambda_i(\vect x))_{i = 1}^d$ are the eigenvalues of the Hessian of $\phi$ at $\vect x$,
$R(\vect x)$ is the matrix containing the associated orthonormal eigenvectors,
and $\gamma > 0$ is a parameter controlling the interpolation error.
A standard algorithm of Delaunay type is used to generate a mesh that is equilateral and uniform with characteristic length 1 in that metric.
This mesh definition ensures that the interpolation error of the phase field is roughly equi-distributed over the parts of the domain
where $ h_{\max}^{-2} \leq \frac{1}{\gamma} \max_{i \in \{1, 2, \dotsc,  d\}} \abs{\lambda_i} \leq h_{\min}^{-2}$.

In most of the simulations presented in the next section,
we set $h_{\min}$ to a value lower than or equal to $\varepsilon/5$,
to ensure that enough mesh points are available for the discretization of the interface region in its normal direction,
and $h_{\max}$ to a value small enough that a good approximation of the chemical potential is possible.
For 3D simulations, however,
choosing $h_{\min} \leq \varepsilon / 5$ when $\varepsilon$ is of the order of $0.01$ leads to a prohibitive computational cost;
in these cases we have thus used a less precise value, as specified in the relevant sections.

For a given mesh $\mbox{$\cal{T}$} = \bigcup_{i=1}^{N_T} T_i$,
we use the standard finite element space
\begin{equation*}
     V_h = \{\phi \in  C(\Omega): \phi_{|T_i} \in P_\rho \text{~for~} i = 1, \dots, N_T \}, \\
\end{equation*}
with $P_\rho$ the space of polynomials of degree $\rho$.
In the numerical experiments below,
we used both quadratic elements ($\rho = 2$) and linear ones ($\rho = 1$).
Space discretization is achieved by replacing $H^1(\Omega)$ by $V_h$ in the
variational formulation~\eqref{VF}, leading to a sparse unsymmetric linear system at each iteration.

\subsubsection{Time step adaptation}
\label{ssub:time-step_adaptation}

Here we assume that $\dot m = 0$ in the boundary condition~\eqref{eq:summary_system_boundary_condition_mu}.
From~\cref{Mlaw,eq:energy_conservation_cahn_hilliard}, this implies that $M(\phi)$ is constant in time and $E(\phi)$ decreases.
Numerical exploration suggests that large free-energy variations are
usually caused by topological changes of interfaces, corresponding to physical phenomena such as the coalescence of droplets.
Since capturing such phenomena precisely is crucial to the accuracy of the solution,
we propose an adaptive strategy that aims to limit the variation of
free energy at each time step.
We adapt the time step based on the dissipation of free energy,
\[
    \Delta E^{n+1} := - \Delta t^{n+1} \, b \, \|\grad \mu^{n+\frac{1}{\alpha}}\|_{L^2}^2,
\]
which is equal to $E(\phi^{n+1}) - E(\phi^n)$ up to numerical dissipation.
Here $\Delta t^{n+1} := t^{n+1} - t^n$.
Six parameters enter in our time-adaptation scheme:
\begin{itemize}
    \item $\Delta t_{\min}, \Delta t_{\max}$: the time steps below which we stop refining and beyond which we stop coarsening, respectively.
    \item $\Delta E_{\min}$: the variation of free energy below which we increase the time step at the next iteration.
    \item $\Delta E_{\max}$: the variation of free energy beyond which we refine the time step and recalculate the numerical solution.
    \item $f > 1$: the factor by which the time step is multiplied or divided at each adaptation.
    \item $0 < \zeta \ll 1$: a small parameter controlling the maximum amount by which the free energy is allowed to increase at each time step.
        In the numerical experiments presented in \cref{sec:simulations}, we always take $\zeta = 0.01$.
\end{itemize}

\begin{algorithm}[H]
    \KwData{$\Delta t_{\min}, \Delta t_{\max}, \Delta E_{\min}, \Delta E_{\max}, f, \phi^n, \Delta t^{n+1}$}
  Compute a solution $(\phi^*,\mu^*)$ of (\ref{VF}) using time step $\Delta t^{n+1}$ \;
  Compute $\abs{\Delta^* E} :=  \Delta t^{n+1} \, b \, \|\grad \mu^{*}  \|_{L^2}^2$ \;
  \eIf{$(\abs{\Delta^* E}  > \Delta E_{\max} \land \Delta t^{n} > \Delta t_{\min}) \lor (E(\phi^*) - E(\phi^n) > \zeta \, \Delta E_{\max})$}{
      Set $\Delta t^{n+1} = \frac{\Delta t^{n+1}}{f}$ and go back to 1\;
    }{
      $\phi^{n+1} = \phi^*$ \;
      \If {$(\abs{\Delta^* E} <  \Delta E_{\min} \land \Delta t^{n+1}< \Delta t_{\max})$}
      {$\Delta t^{n+2} = f \Delta t^{n+1}$ \;}
    }
    $n = n + 1$ and go back to 1.
    \caption{Time step adaptation}
    \label{algorithm:timeStepAlgo}
\end{algorithm}
The condition $(E(\phi^*) - E(\phi^n)) > \zeta \, \Delta E_{\max}$ serves to guarantee that the method does not blow up.
The choice of a nonzero right-hand side is motivated by the fact that,
when the system is close to equilibrium,
it can happen that $E(\phi^*) > E(\phi^n)$.
This is because,
in contrast with the sign of $ND_{\textrm{philic}}(\phi^n, \phi^{n+1})$,
which is always positive or zero according to~\cref{eq:numerical_dissipation},
the signs of $ND_{\textrm{phobic}}(\phi^n, \phi^{n+1})$ and $ND_{\textrm{wall}}(\phi^n, \phi^{n+1})$ are in general unknown.

In the numerical experiments presented in \cref{sec:simulations},
we chose $\Delta t_{\min} = 0$.
Since the numerical dissipation term scales as $\Delta t^2$ for the method OD2-W,
the inequality $E(\phi^{n+1}) \leq E(\phi^n) + \zeta \, \Delta E_{\max}$ will always hold for $\Delta t$ small enough,
so the refinement process is guaranteed to terminate at each iteration.

If the condition $\dot m = 0$ is not satisfied,
then the adaptation scheme we proposed,
being based on limiting the decrease and preventing the increase (beyond a very small tolerance equal to $\zeta \Delta E_{\textrm{max}}$) of free energy at each time step, is not applicable.
Indeed, by \cref{eq:energy_conservation_cahn_hilliard}, a nonzero mass flux can cause a positive variation of the total free energy,
which our scheme would be unable to distinguish from a nonphysical increase of energy caused by numerical errors.
In this nonzero flux situation,
the adaptation scheme presented in~\cite{guillen2014second},
which is based on limiting the numerical dissipation,
could be used instead (with suitable modifications to accommodate for the presence of a flux).

\section{Numerical results}
\label{sec:simulations}

The new numerical method is applied on a number of test cases.
We have used \emph{FreeFem++}~\cite{hecht2012new} for the implementation off the finite element method and 2D mesh adaptation,
\emph{umfpack}~\cite{davisUMFPACK} for the linear solver,
\emph{mshmet}~\cite{mshmet} and \emph{tetgen}~\cite{si2015tetgen} for the mesh adaptation in 3D,
and \emph{gmsh}~\cite{geuzaine2009gmsh} for the description of the geometry, post-processing and 3D visualization.
In \cref{sub:equilibrium_contact_angle} we check that the numerical scheme leads
to the correct equilibrium solution in the simple case of a droplet spreading
on a philic or phobic substrate. In \cref{sub:convergence_of_the_method} we
study the convergence of the method with respect to the time step and the
mesh size, when a uniform mesh and a constant time step are used. In
\cref{sub:convergence_of_the_method} we illustrate the time-adaptation scheme
in the case of two droplets coalescing on a substrate. Finally,
\cref{sub:wetting_in_complex_geometries_and_with_heterogeneous_substrates}
demonstrates the ability of the numerical scheme to scrutinize wetting
phenomena in more complicated geometries, and in the presence of
heterogeneous substrates.
The code used for the simulations is available on line, see~\cite{gitcode}.

\subsection{Equilibrium contact angle}
\label{sub:equilibrium_contact_angle}

We consider a 2D sessile droplet on a flat substrate where we impose the
no-flux condition and the wetting condition~\eqref{eq:wettingBC}
with the modified wall energy~\eqref{eq:modified_wall_energy} and uniform contact angle $\theta$:
\begin{equation}
    \grad \mu \cdot \vect n = 0, \quad \varepsilon \grad \phi \cdot \vect{n} = - f_w(\phi)
    \label{eq:test1_boundary_condition}
\end{equation}
Our aim in this section is to check that our method is able to accurately capture the imposed contact angle, $\theta$.
\Cref{fig:contact_angle_test} shows the equilibrium position of a droplet for different values of $\theta$,
for $b = 1$ and $\varepsilon = 5 \times 10^{-3}$.
In all cases we used the scheme OD2-W,
with quadratic basis functions and adaptation in space only using the parameters $h_{\max} = 10 \, h_{\min} = 0.01$,
and we computed the contact angle of the $\phi = 0$ isoline at the substrate.
A very good agreement is achieved between the imposed equilibrium contact angle and the observed numerical one.

\begin{figure}
    \centering {%
        \begin{minipage}[b]{.33\linewidth}
            \centering
            \includegraphics[width=\textwidth]{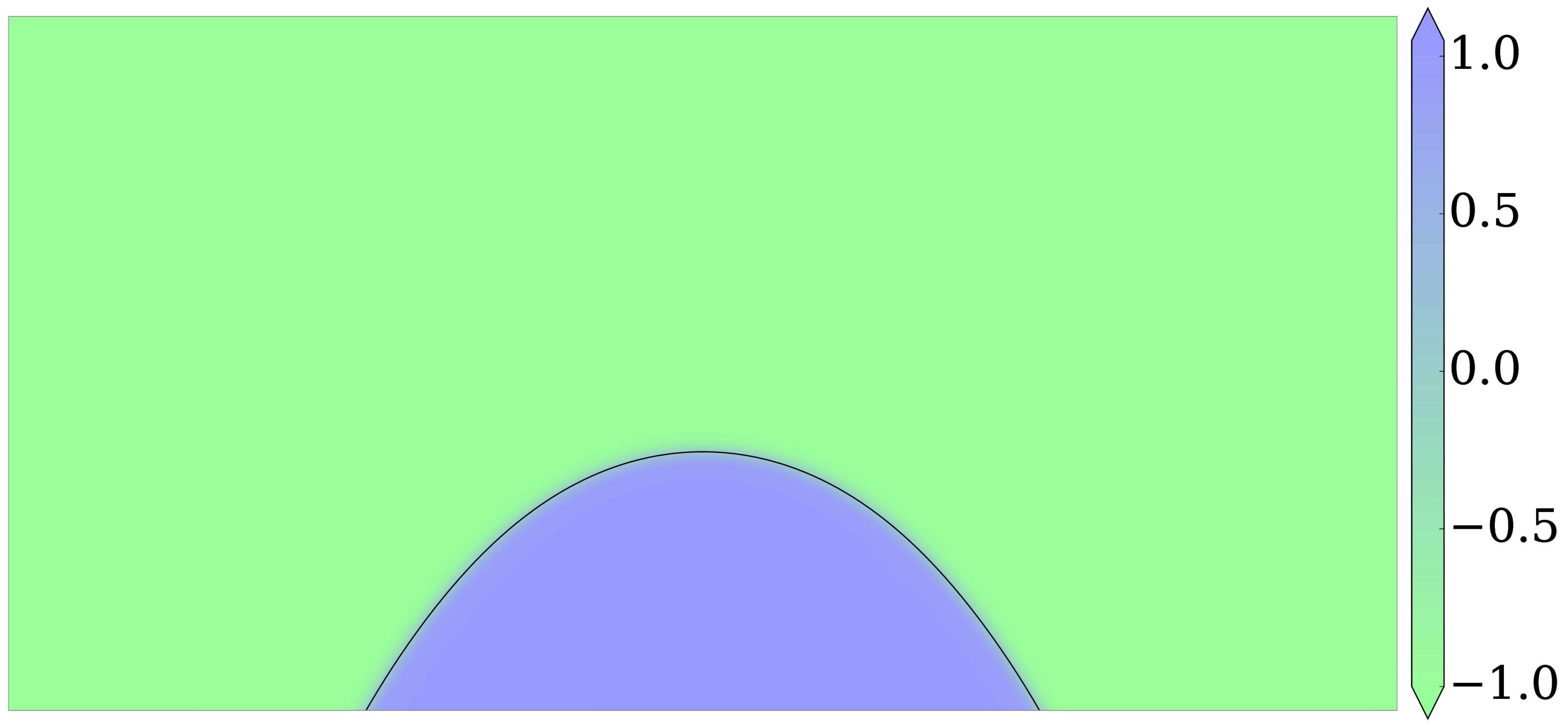}
            \subcaption{$\theta = \pi/3$, $\theta^*/\theta = 0.991$.}
        \end{minipage}%
        \begin{minipage}[b]{.33\linewidth}
            \centering
            \includegraphics[width=\textwidth]{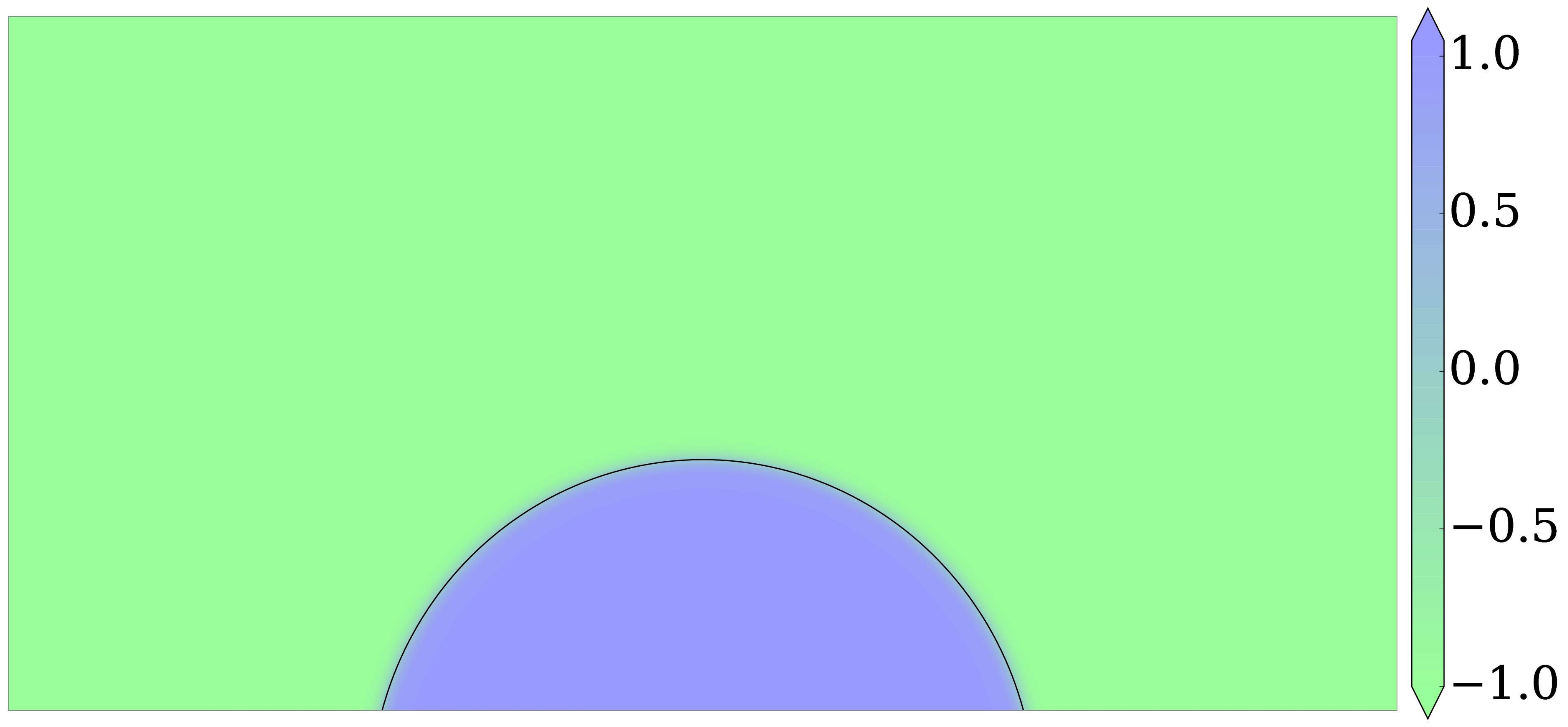}
            \subcaption{$\theta = 5 \pi /12$, $\theta^*/\theta = 0.985$.}
        \end{minipage}%
        \begin{minipage}[b]{.33\linewidth}
            \centering
            \includegraphics[width=\textwidth]{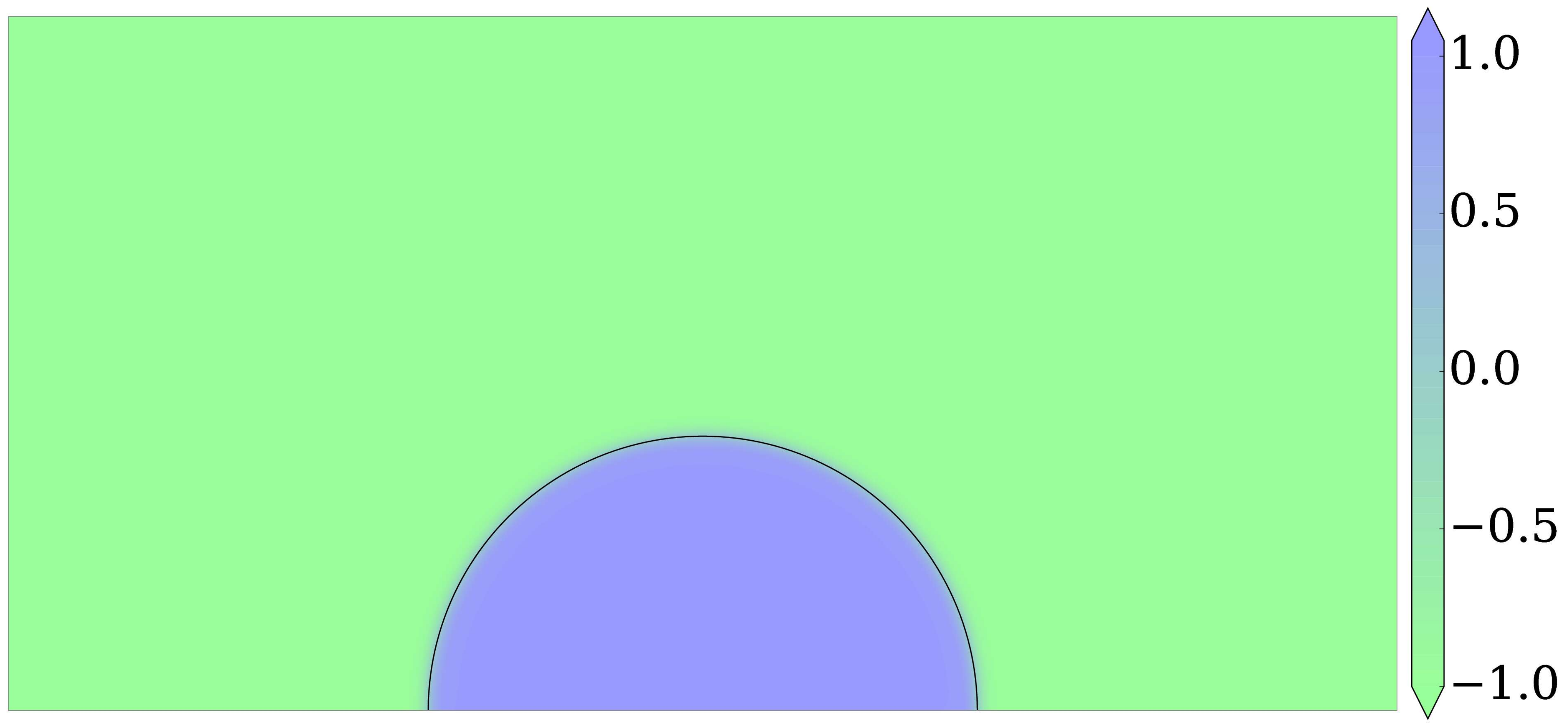}
            \subcaption{$\theta = \pi/2$, $\theta^*/\theta = 0.980$.}
        \end{minipage}%

        \begin{minipage}[b]{.33\linewidth}
            \centering
            \includegraphics[width=\textwidth]{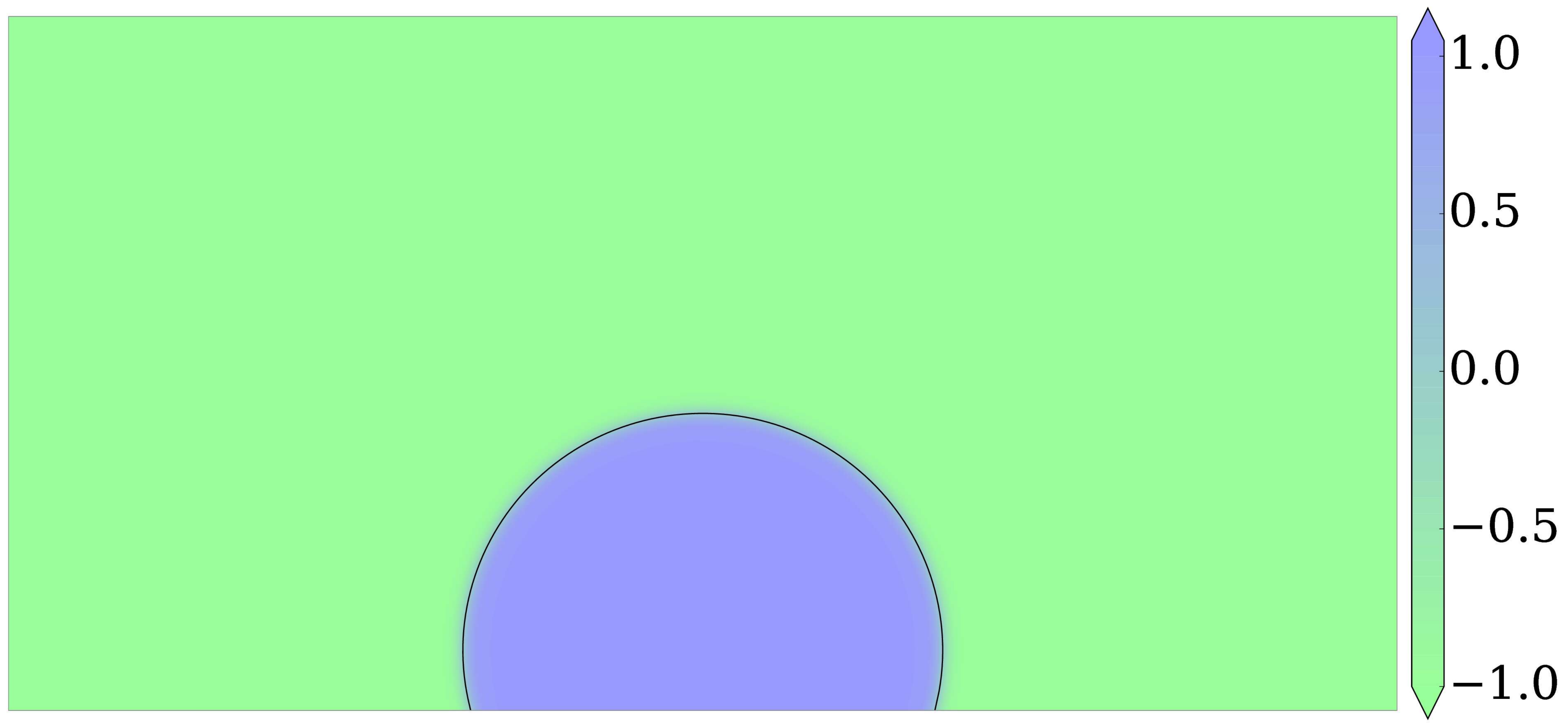}
            \subcaption{$\theta = 7\pi/12$, $\theta^*/\theta = 0.979$.}
        \end{minipage}%
        \begin{minipage}[b]{.33\linewidth}
            \centering
            \includegraphics[width=\textwidth]{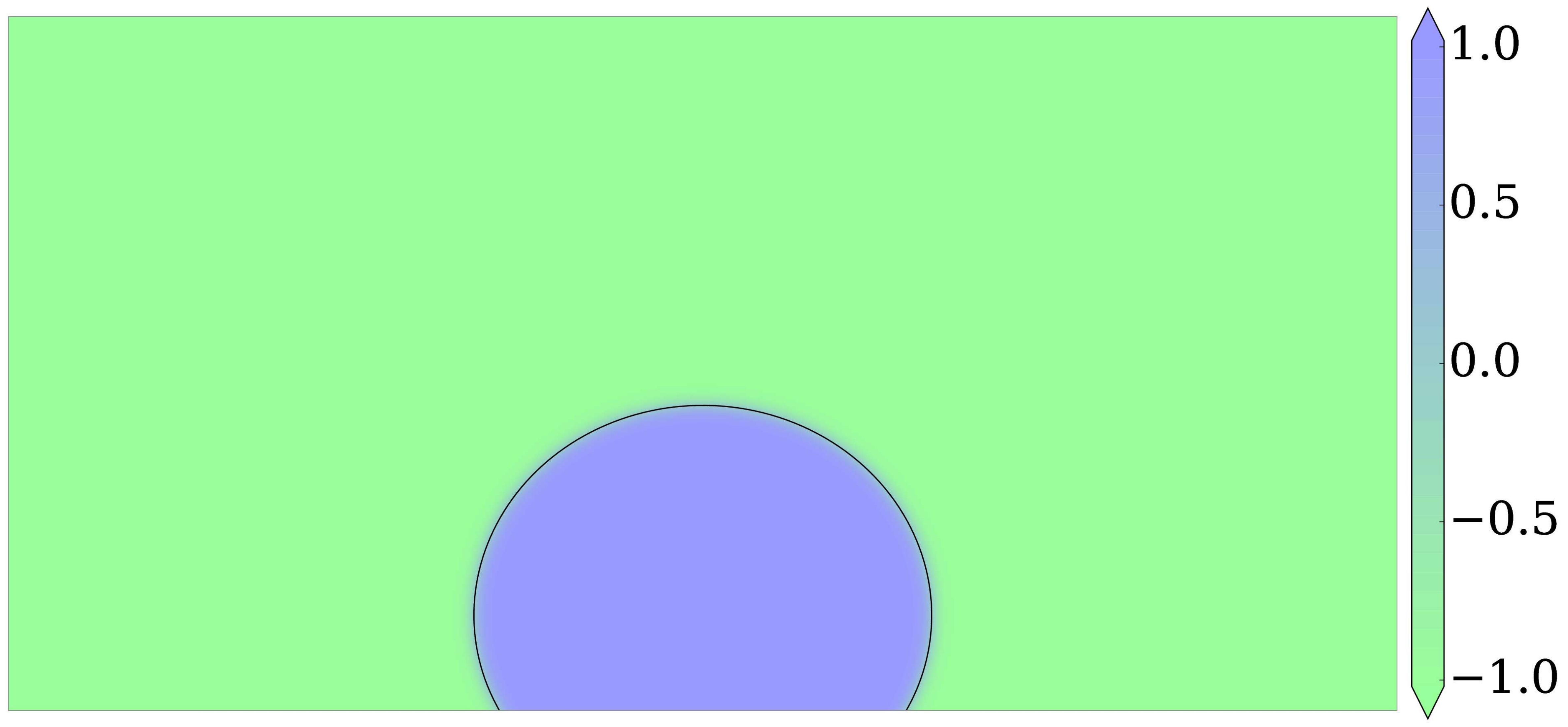}
            \subcaption{$\theta = 2\pi/3$, $\theta^*/\theta = 0.978$.}
        \end{minipage}%

        \caption{%
            Equilibrium position of a droplet on substrates with different wetting properties.
            In all cases, $\theta$ denotes the angle appearing in boundary condition~\eqref{eq:test1_boundary_condition}
            and $\theta^*$ denotes the angle calculated numerically.
            Blue corresponds to $\phi=1$ and green to $\phi=-1$.%
        }\label{fig:contact_angle_test}
    }
\end{figure}

\subsection{Convergence of the method}
\label{sub:convergence_of_the_method} Here, we study the convergence of the
method when both time step and mesh size decrease. The problem we considered
to that purpose is the coalescence of two adjacent sessile droplets as they
spread on a flat substrate.
For the simulation,
we used the initial condition
\begin{equation}
    \label{eq:initial_condition_test_convergence}
    \phi(x,0) = 1 -\tanh\left(\frac{\sqrt{(x-x_1)^2 + y^2} - r}{\sqrt 2 \varepsilon}\right) - \tanh\left(\frac{\sqrt{(x-x_2)^2 + y^2} - r}{\sqrt 2 \varepsilon}\right),
\end{equation}
in the domain $[0,2] \times [0, 0.5]$, with $x_1 = 0.65$, $x_2 = 1.35$, $r = 0.25$,
and at the boundary we imposed a uniform contact angle, $\theta = \pi/4$,
using the wall energy~\eqref{eq:modified_wall_energy}.
Only linear elements were used.

For the convergence as $h \to 0$, we solved the problem numerically for
several values of $h$, without mesh adaptation and for $\varepsilon = 0.1$,
so that enough data points could be generated at a reasonable numerical cost.
Since the exact solution to the CH equation in this case is not known
analytically, we calculated the error by comparison of the numerical
solutions to the solution obtained with the smallest value of $h$. Results
are presented in \cref{fig:convergence_mesh_size_without_adaptation}. As we
can see, the observed convergence rate is almost equal to 2, which is the
optimal rate in the case of linear basis functions.

\begin{figure}
    \centering {%
        \begin{minipage}[b]{.5\linewidth}
            \includegraphics[width=\textwidth]{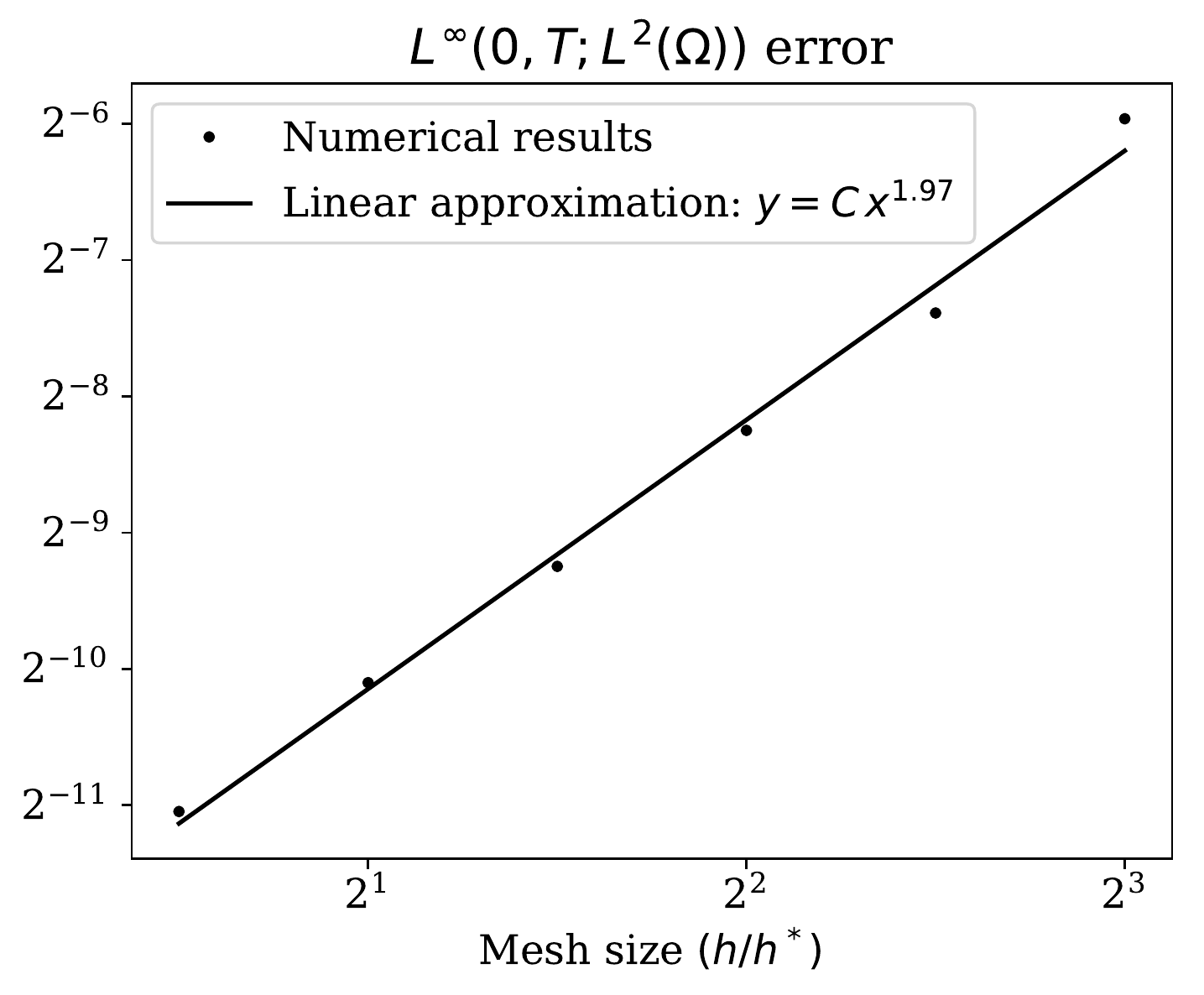}
        \end{minipage}%
        \caption{Convergence of OD2-W with respect to the mesh size, without mesh adaptation.
        In this case, $h$ corresponds to the uniform mesh size, given as input to the mesh generator of \emph{gmsh}, and $h^* = 0.01$.}
        \label{fig:convergence_mesh_size_without_adaptation}
    }
\end{figure}

Now we address the convergence with respect to the time step. For this case,
we used the parameters $\varepsilon = 0.1$, $b=10^4$, and the minimum time
step we considered was $\Delta t^* := 0.00665.$ In
\cref{fig:convergence_time_step_without_adaptation}, we present convergence
curves for OD1-W, OD2-W, and OD2mod-W. We note that the convergence rates are
close to the expected ones, and that the use of OD2 gives significantly more
accurate results than the other two methods. In
\cref{fig:convergence_time_numerical_dissipation}, the total numerical
dissipation produced by the numerical schemes is presented. Here too,
numerical results agree with the theoretical results of
\cref{sec:numerical_method_cahn-hilliard}.

\begin{figure}
    \centering {%
        \begin{minipage}[b]{.32\linewidth}
            \includegraphics[width=\textwidth]{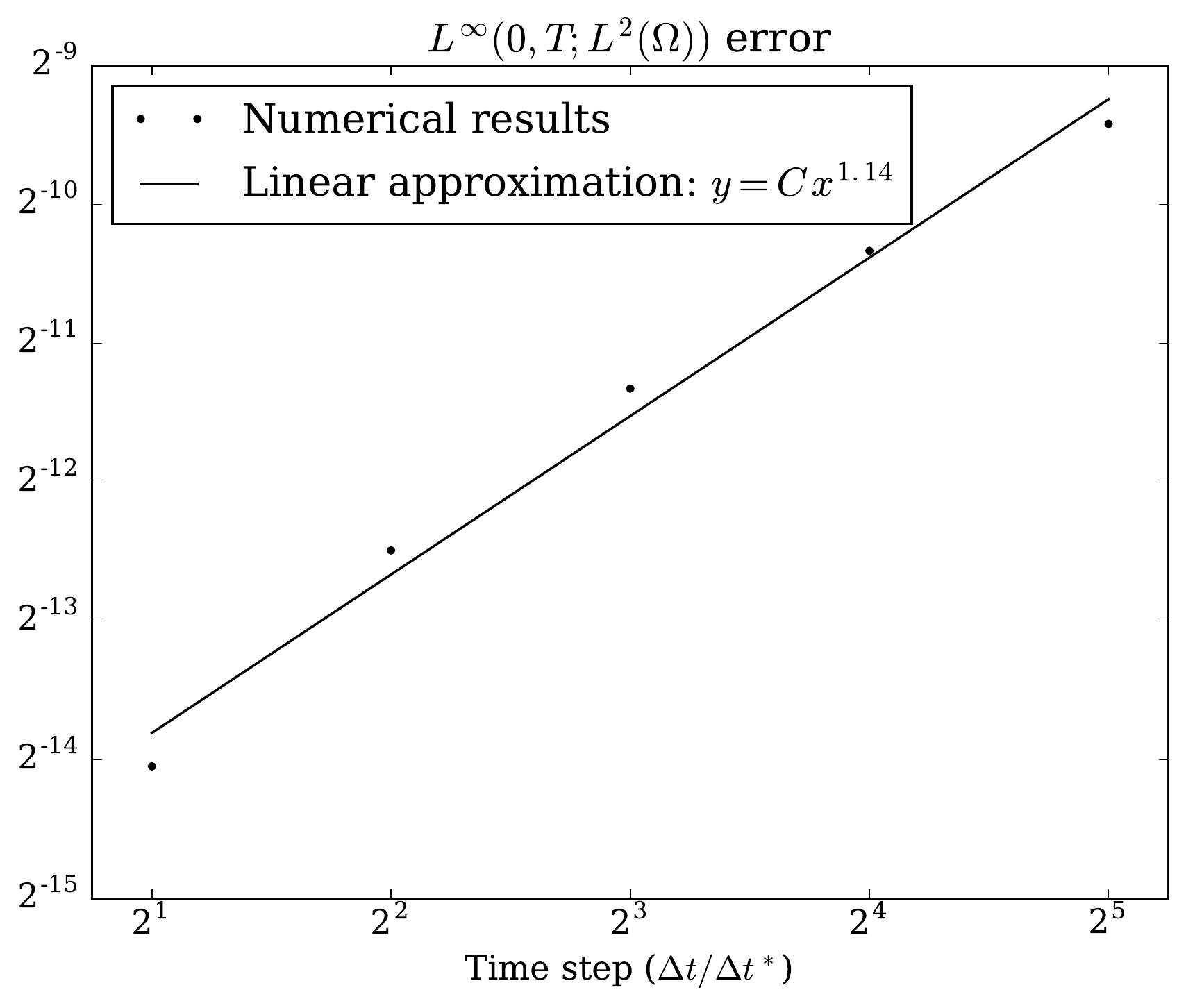}
            \subcaption{OD1-W}
        \end{minipage}%
        \begin{minipage}[b]{.32\linewidth}
            \includegraphics[width=\textwidth]{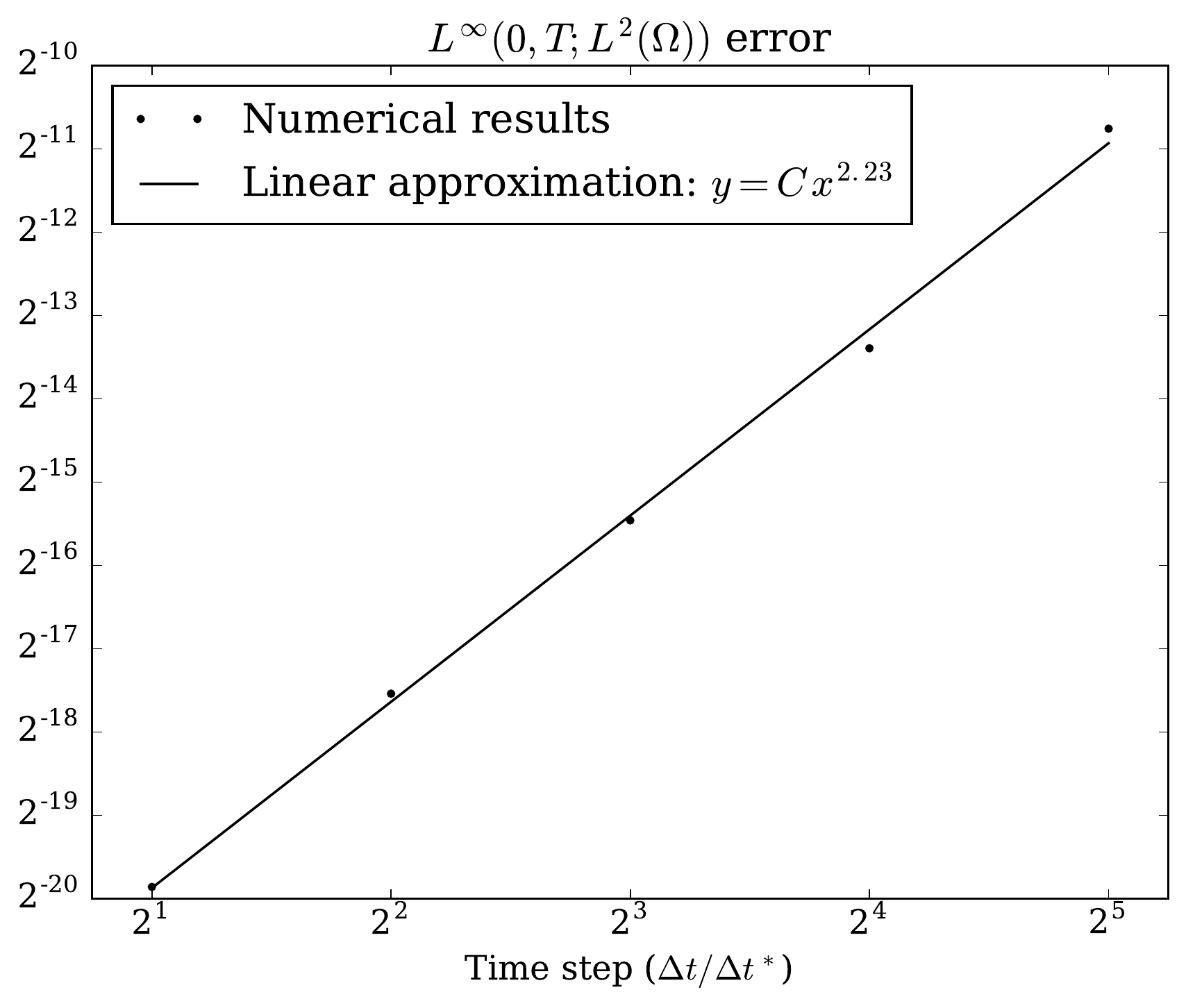}
            \subcaption{OD2-W}
        \end{minipage}%
        \begin{minipage}[b]{.32\linewidth}
            \includegraphics[width=\textwidth]{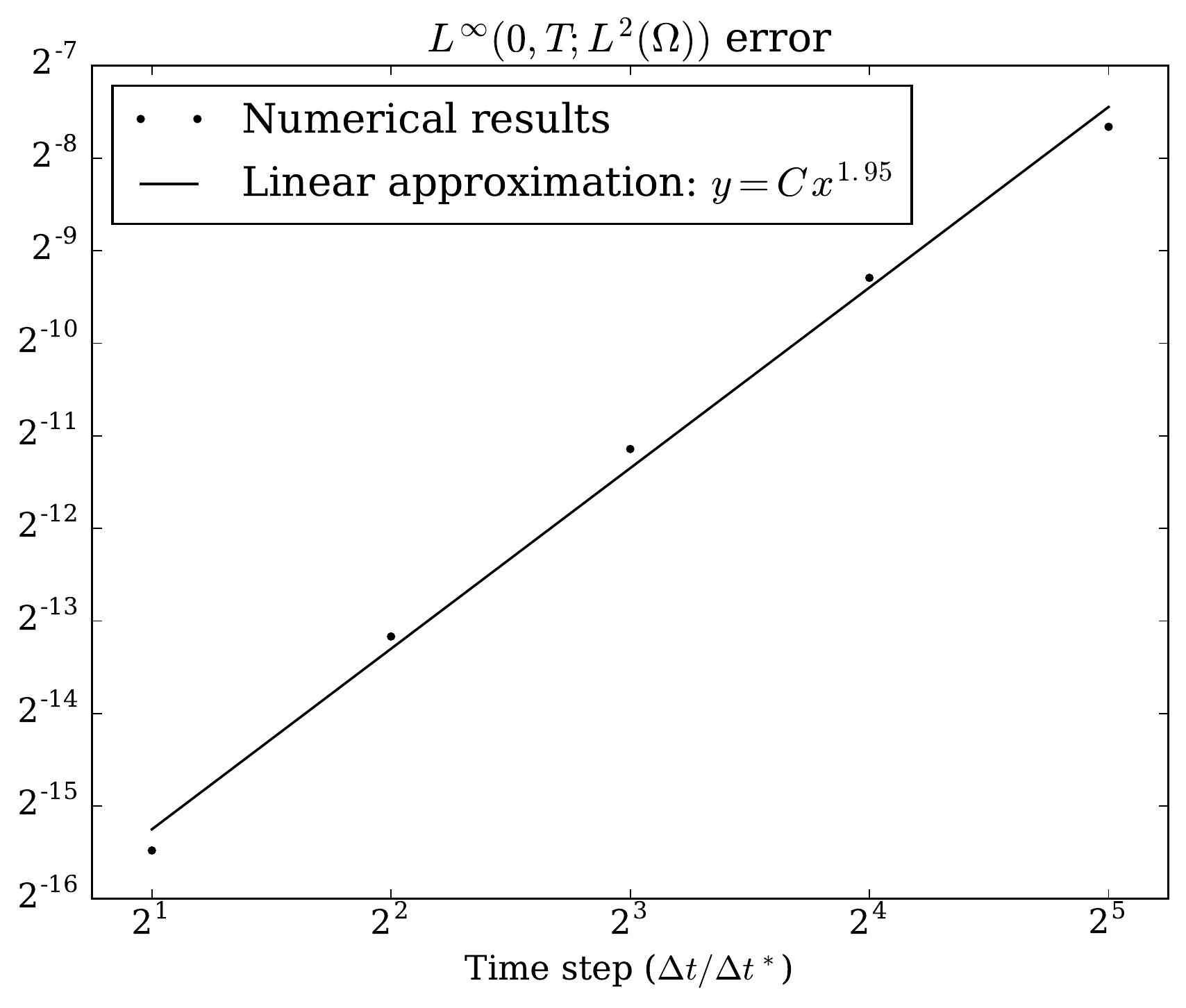}
            \subcaption{OD2mod-W with $\beta = 10 \, \Delta t$}
        \end{minipage}%
        \caption{%
            Convergence of the numerical method with respect to the time step, without mesh adaptation.
            In the case of OD1-W, the rate of convergence is close to the expected value of $1$.
            In the other two cases, the rate of convergence is close to the expected value of $2$.
        }
        \label{fig:convergence_time_step_without_adaptation}
    }
\end{figure}

\begin{figure}
    \centering {%
        \begin{minipage}[b]{.32\linewidth}
            \includegraphics[width=\textwidth]{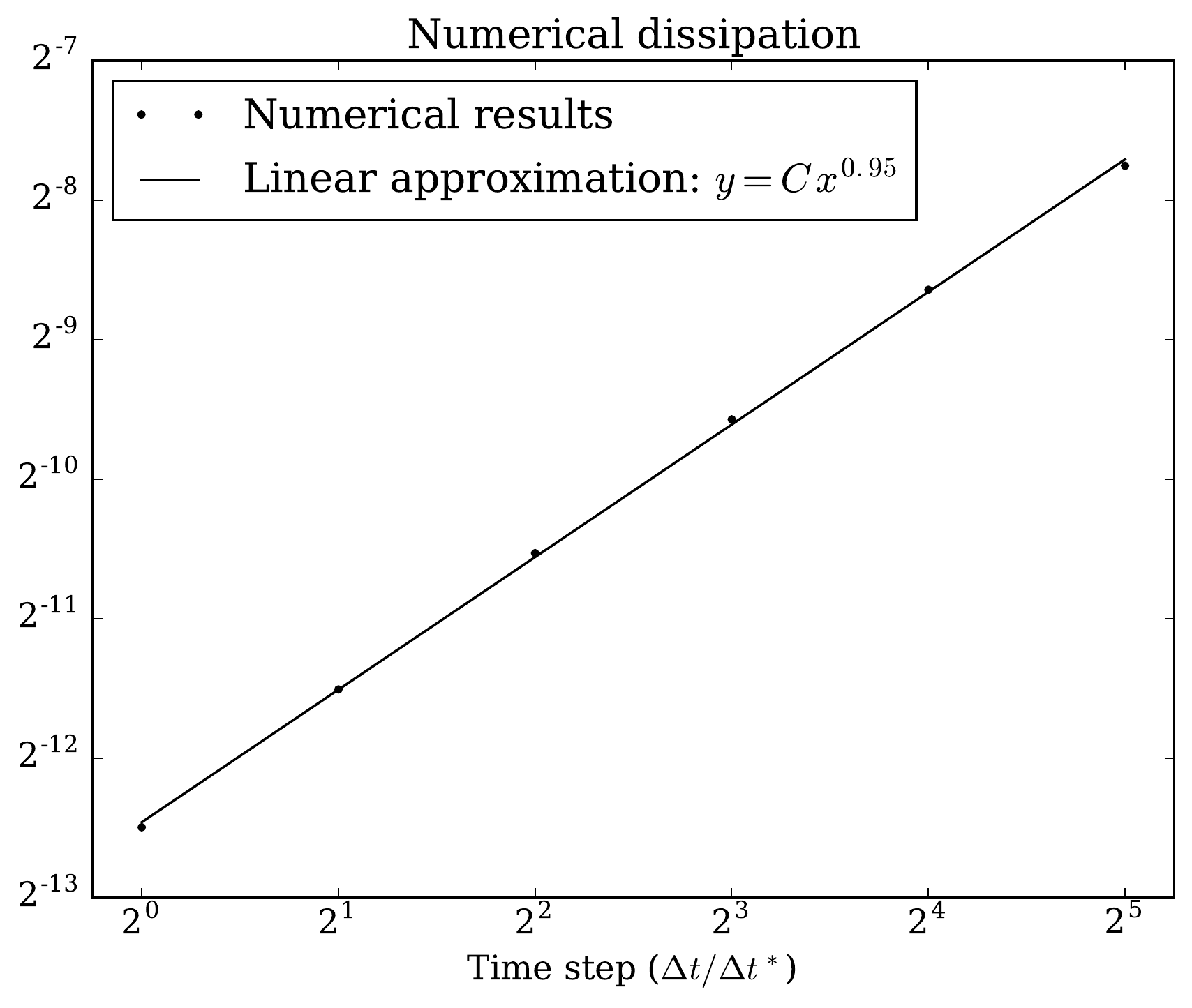}
            \subcaption{OD1-W}
        \end{minipage}%
        \begin{minipage}[b]{.32\linewidth}
            \includegraphics[width=\textwidth]{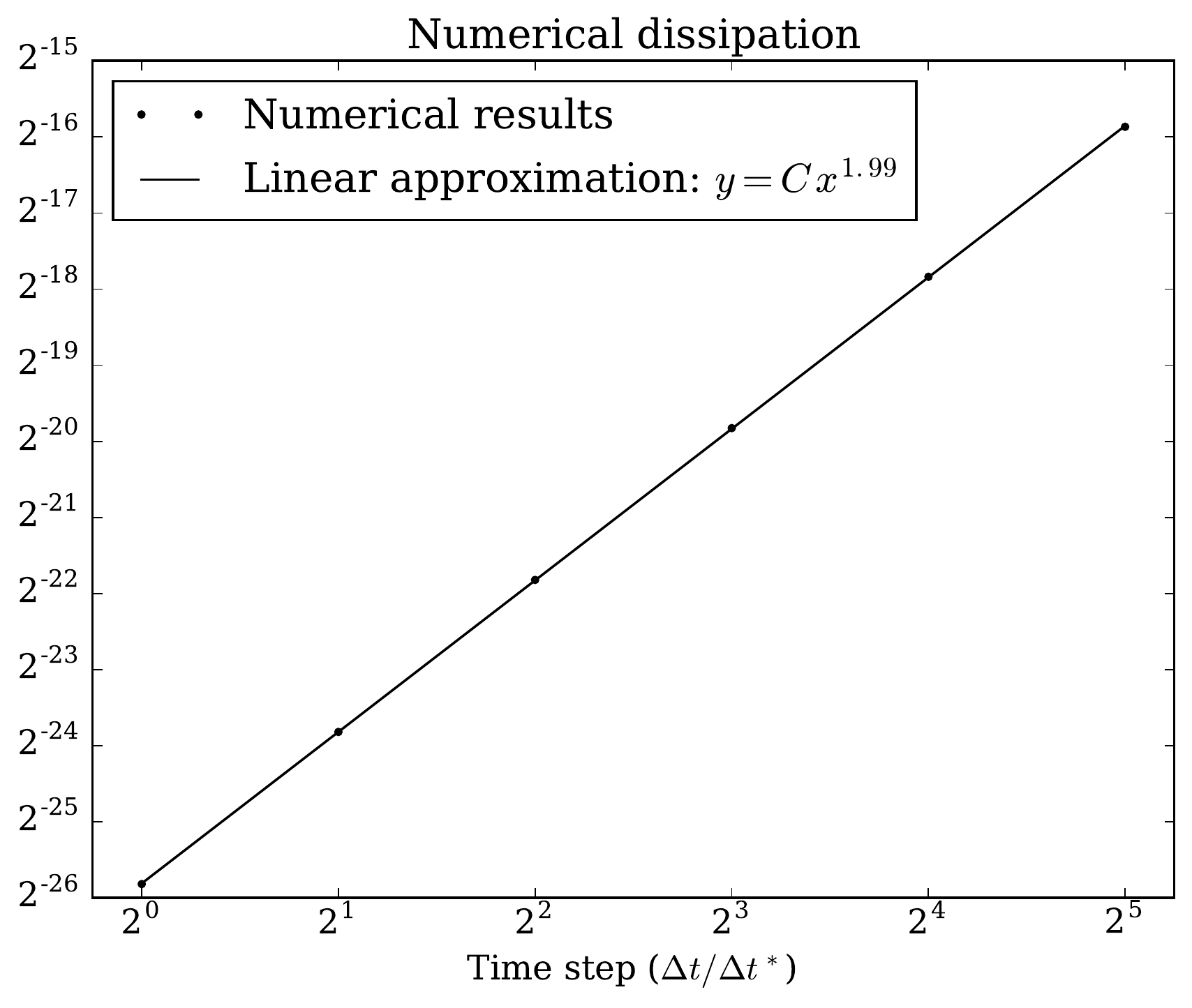}
            \subcaption{OD2-W}
        \end{minipage}%
        \begin{minipage}[b]{.32\linewidth}
            \includegraphics[width=\textwidth]{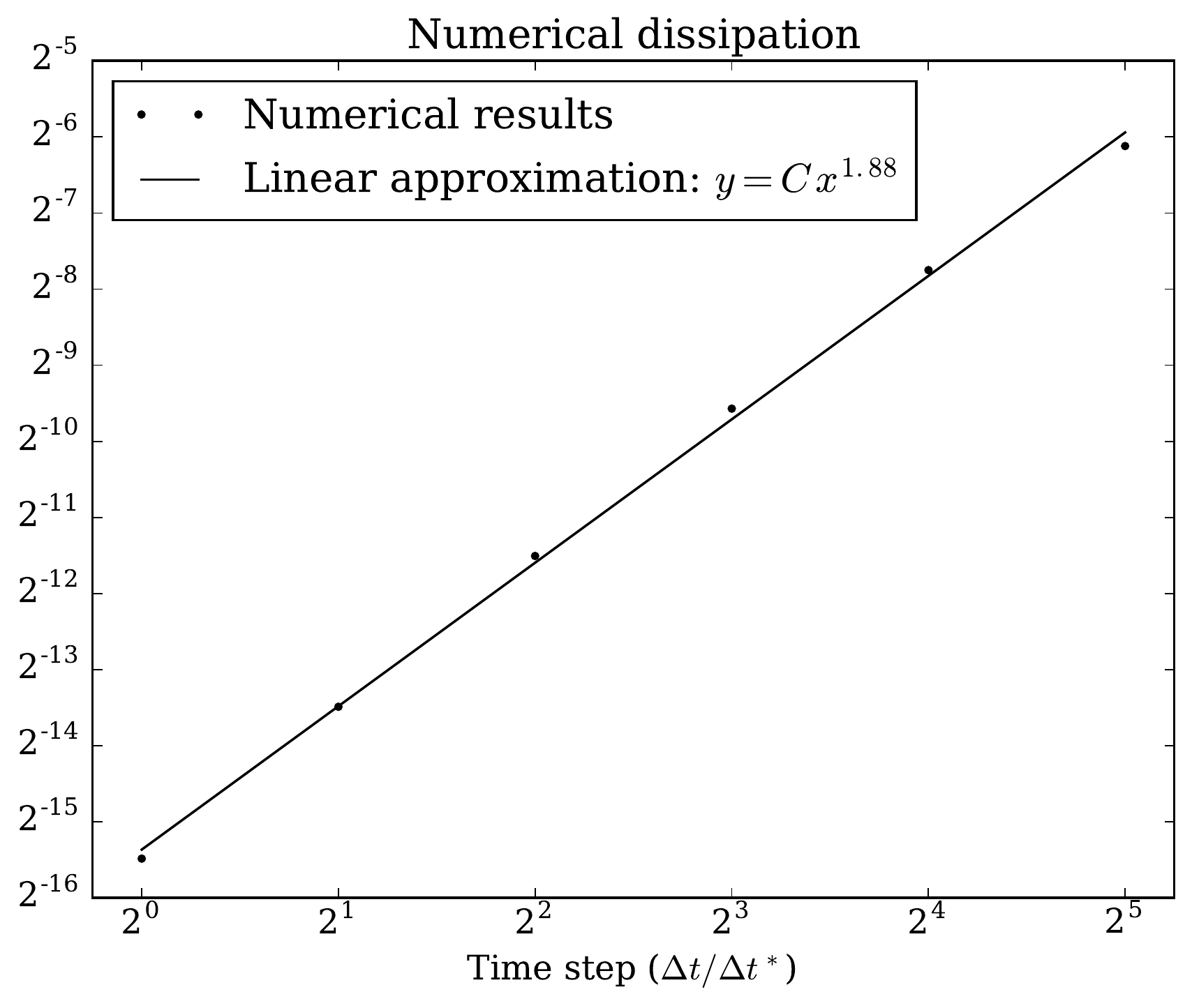}
            \subcaption{OD2mod-W with $\beta = 10 \, \Delta t$}
        \end{minipage}%
        \caption{%
            Total numerical dissipation generated by the numerical schemes in the simulation used to produce~\cref{fig:convergence_time_step_without_adaptation}.
            OD2-W is by far the scheme producing the least numerical dissipation, even for relatively large time steps.
            OD2mod-W, on the other hand, introduces significant numerical dissipation for large time steps,
            owing to the large value of $\beta$ that was chosen for the simulation,
            but is much less dissipative than OD1-W for small time steps.
        }
        \label{fig:convergence_time_numerical_dissipation}
    }
\end{figure}

\subsection{Time-adaptation scheme}%
\label{sub:time_adaptation_scheme}

In this section, we examine the performance of the adaptive time-stepping scheme in the case of two droplets evolving on a chemically homogeneous substrate.
We start from the situation where $\phi = -1$ everywhere except in two half-circles,
of radius $r = 0.25$ and centred at $(0.65,0)$ and $(1.35,0)$, where $\phi = 1$.
We used quadratic basis functions and  the following parameters:
$b=10^{-4}$,
$\varepsilon=0.01$,
$f = \sqrt{2}$,
$\Delta t_0 = 0.02$,
$\Delta t_{\min} = 0$,
$\Delta t_{\max} = 16\Delta t_0$,
$\Delta E_{\min} = 0.0001$,
$\Delta E_{\max} = 0.0002$,
$h_{\max}=0.05$
$h_{\min}=0.001$,
and for $\theta$ we considered three values: $\pi/4$, $\pi/2$, $3\pi/4$.

Snapshots of the phase field and of the chemical potential at different times of the simulation are presented
in~\cref{fig:coalescence_pio4,fig:coalescence_3pio4}
for the case $\theta = \pi/4$ and $\theta = 3\pi/4$, respectively.
The case $\theta = \pi/2$ is less interesting because,
in view of the initial condition,
the droplets remain essentially motionless throughout the simulation;
we do not present snapshots of the solution in that case.

\begin{figure}
  \begin{center}
    \includegraphics[width=0.4\textwidth]{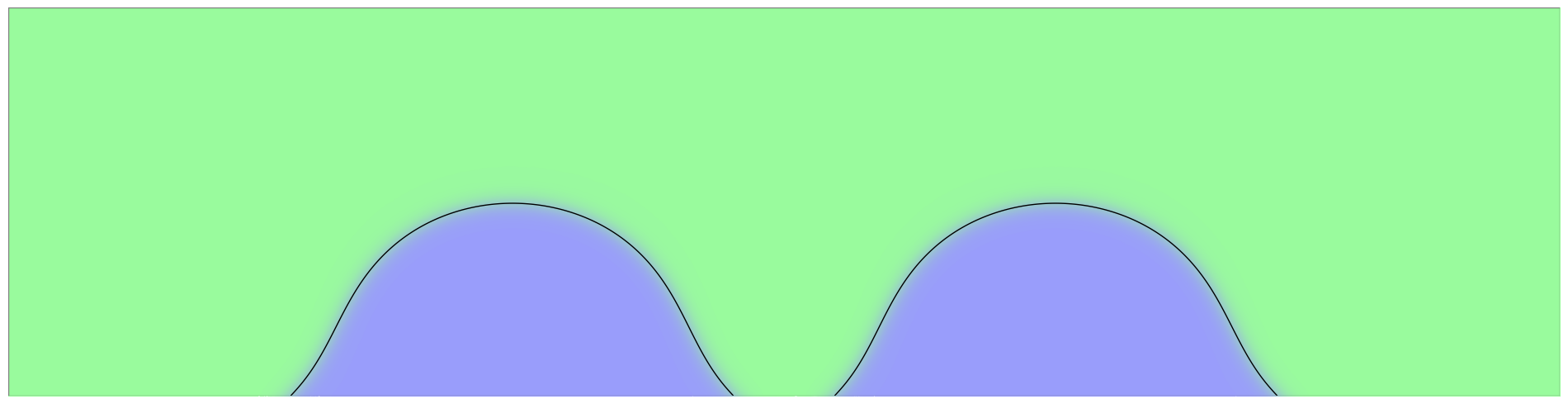} \includegraphics[width=0.4\textwidth]{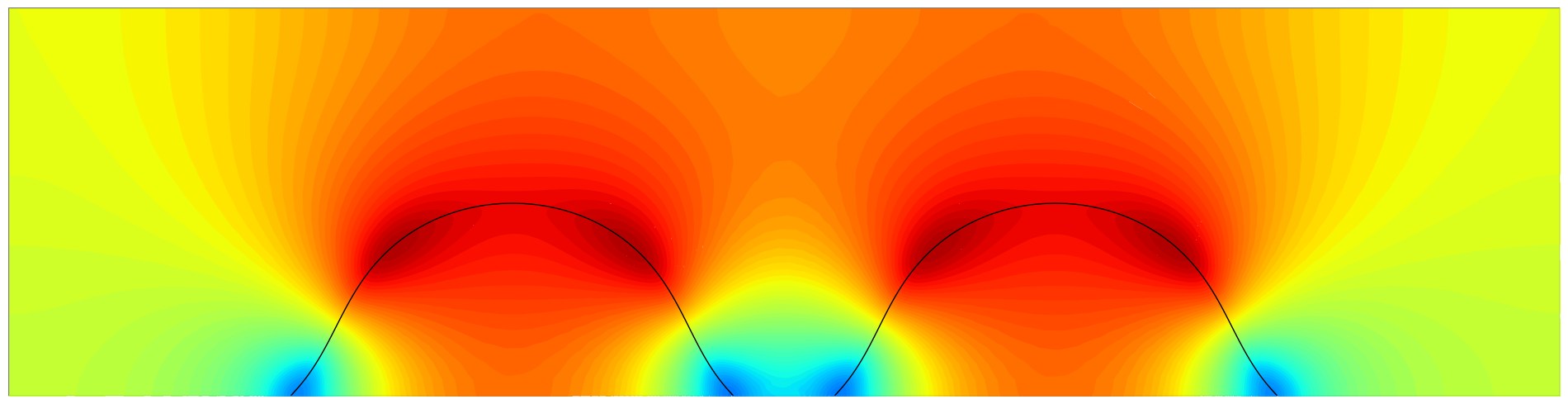}
    \includegraphics[width=0.4\textwidth]{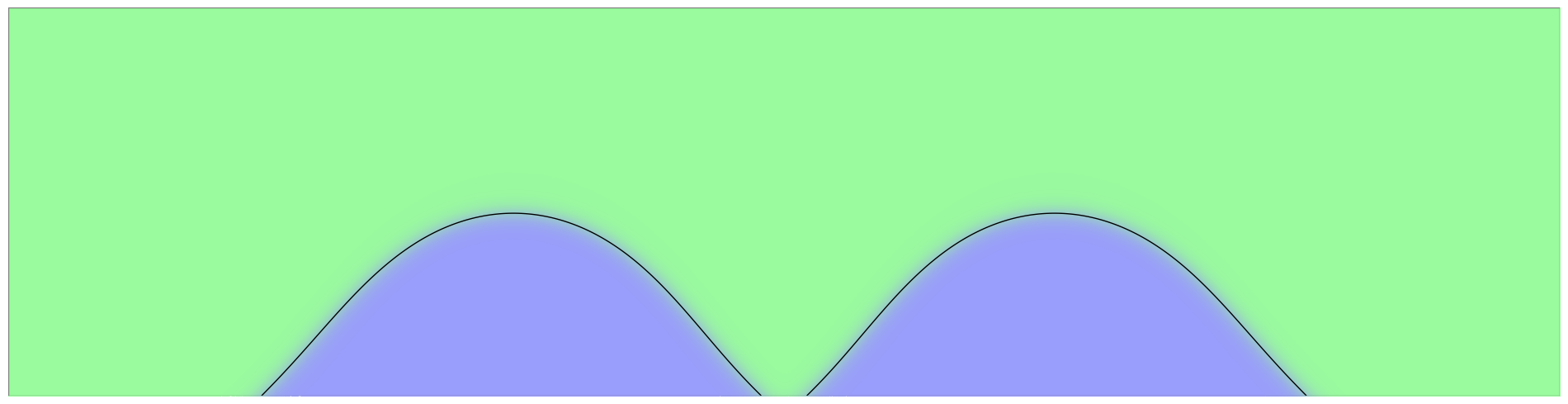} \includegraphics[width=0.4\textwidth]{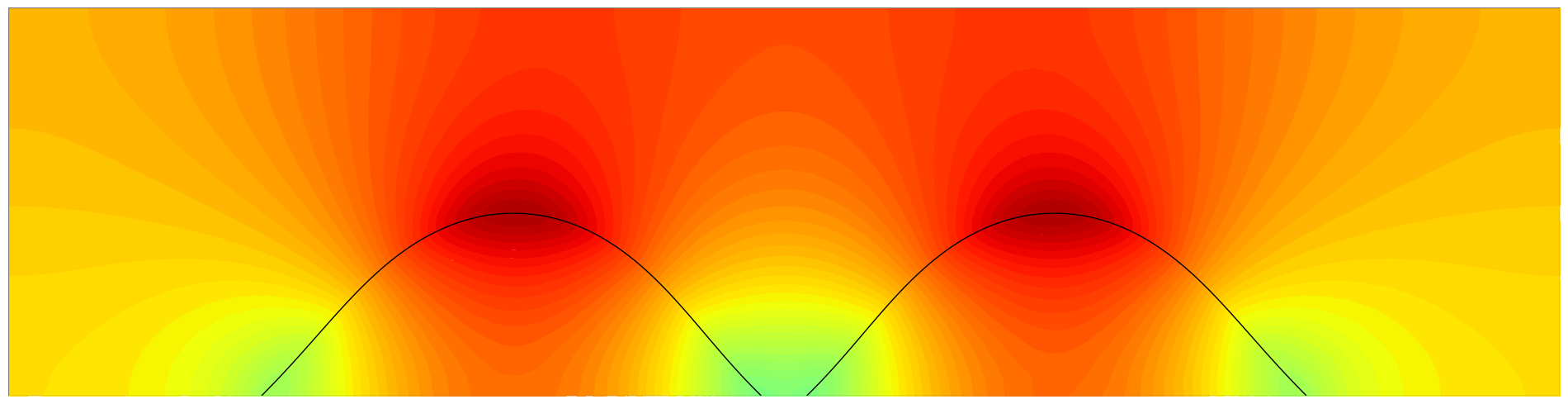}
    \includegraphics[width=0.4\textwidth]{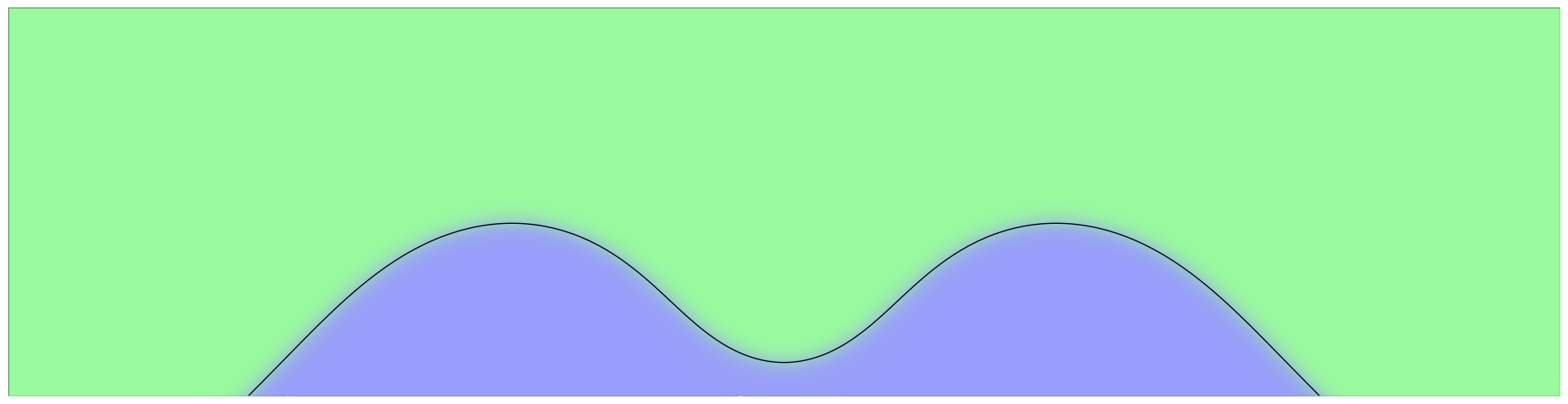} \includegraphics[width=0.4\textwidth]{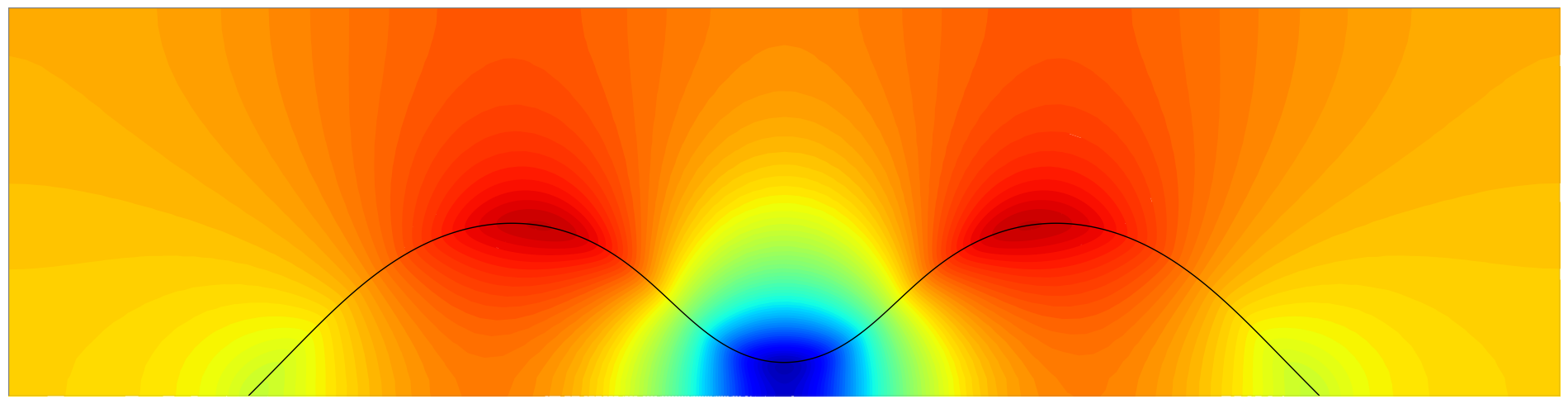}
    \includegraphics[width=0.4\textwidth]{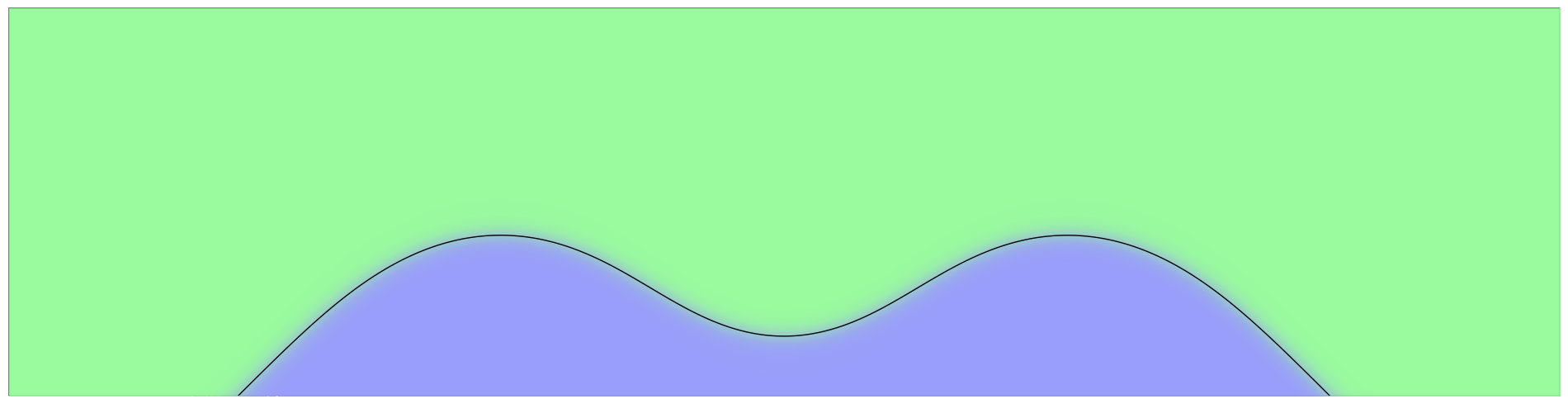} \includegraphics[width=0.4\textwidth]{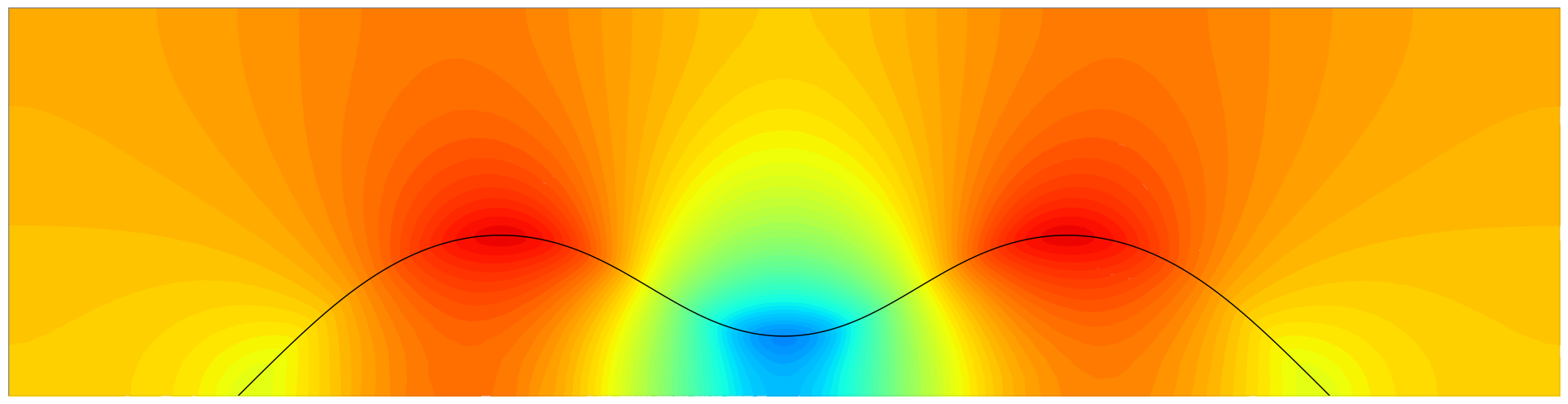}
    \includegraphics[width=0.4\textwidth]{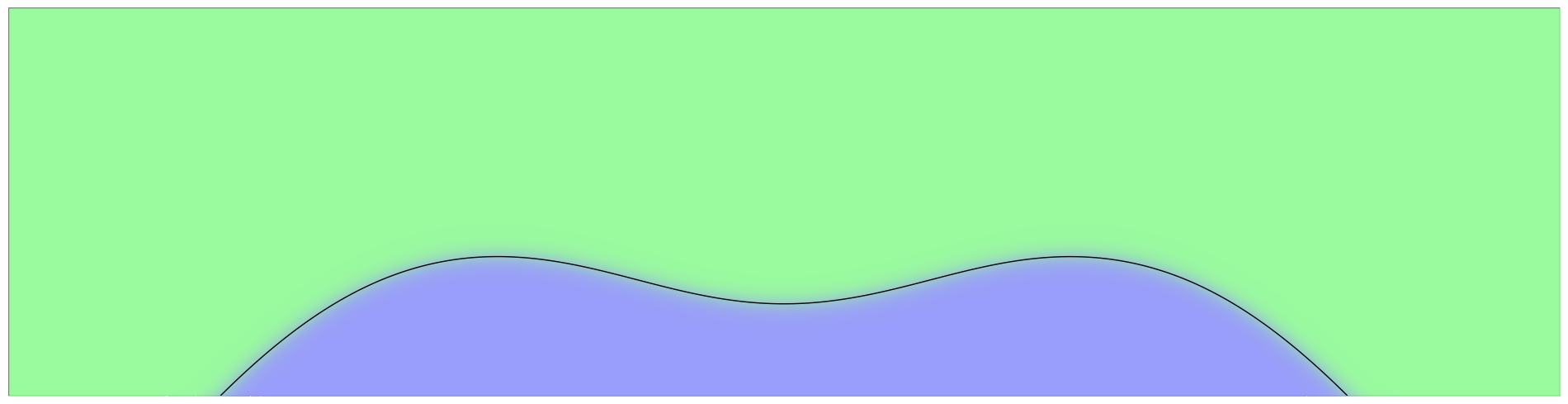} \includegraphics[width=0.4\textwidth]{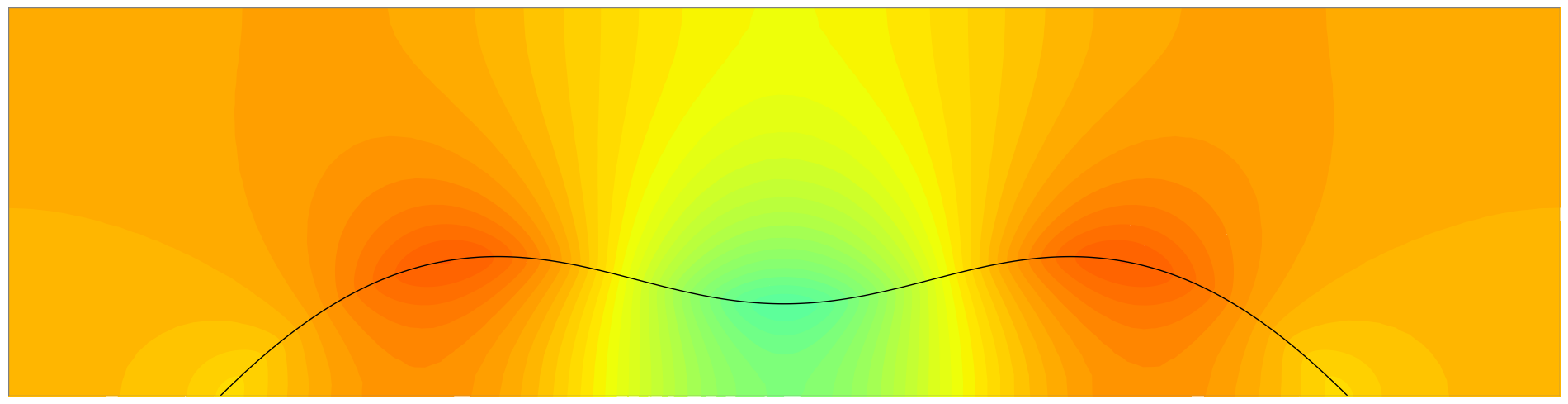}
    \includegraphics[width=0.4\textwidth]{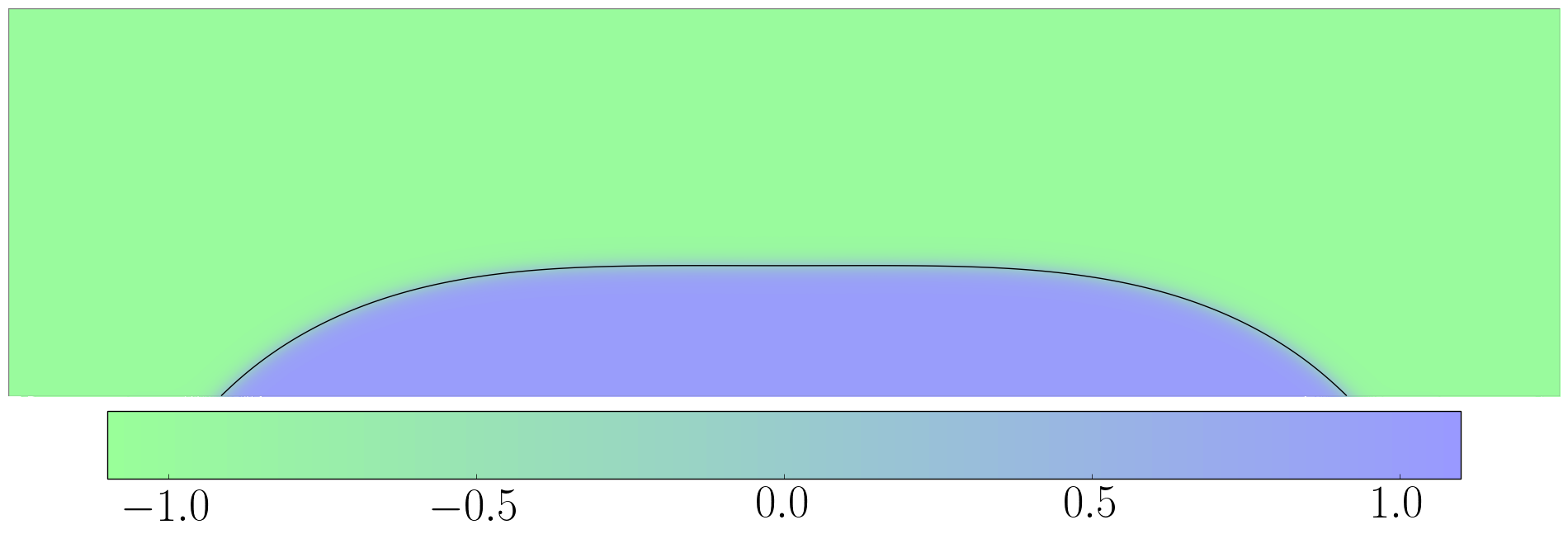} \includegraphics[width=0.4\textwidth]{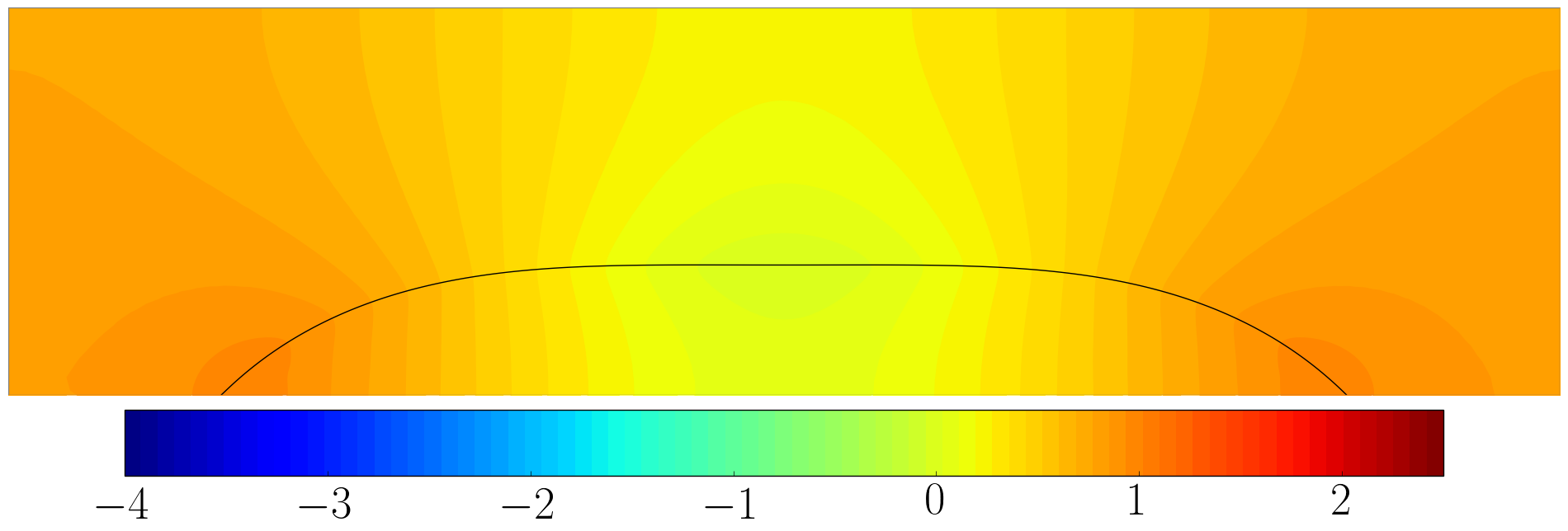}
  \end{center}
  \caption{Phase field and chemical potential during the coalescence of two sessile droplets on a hydrophilic substrate for $\theta = \pi/4$.
      The snapshots correspond to iterations 500, 1000, 1500, 2000, 2500, and 3000, which correspond to times 2.26, 19.55, 34.13, 49.03, 83.20 and 273.64.
      Blue color represents phase $\phi=1$ and green phase $\phi=-1$.}%
      \label{fig:coalescence_pio4}
\end{figure}

\begin{figure}
  \centering
     \centering
      \includegraphics[width=0.49\textwidth]{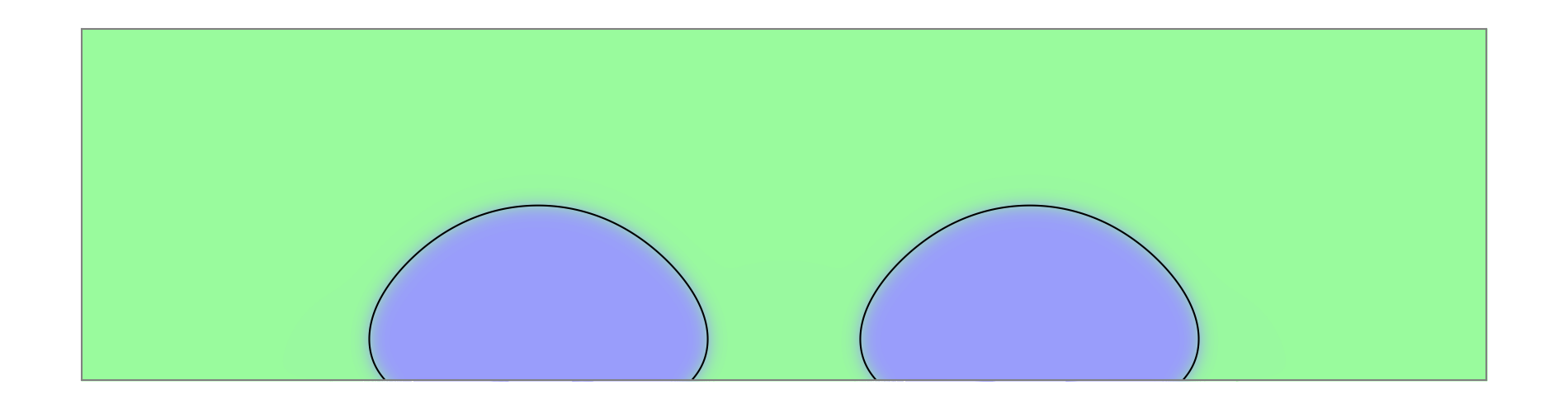} \includegraphics[width=0.49\textwidth]{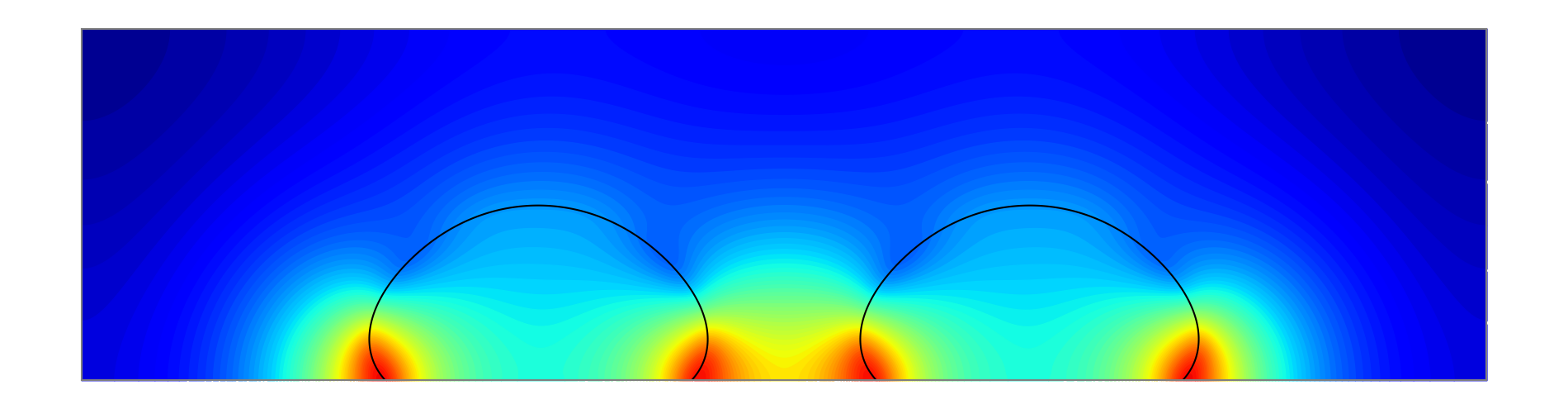}

      \includegraphics[width=0.49\textwidth]{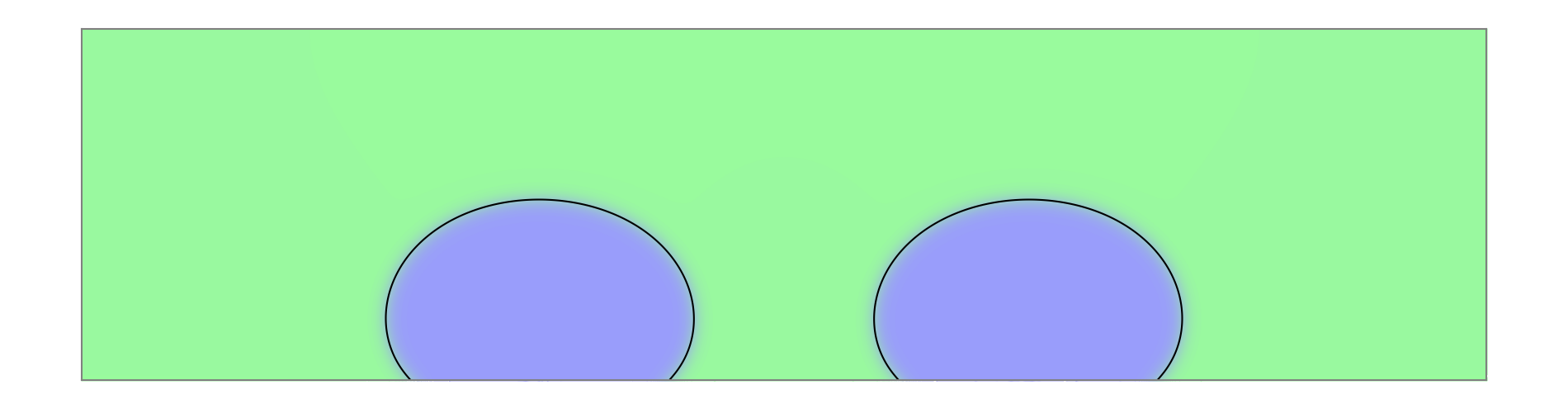} \includegraphics[width=0.49\textwidth]{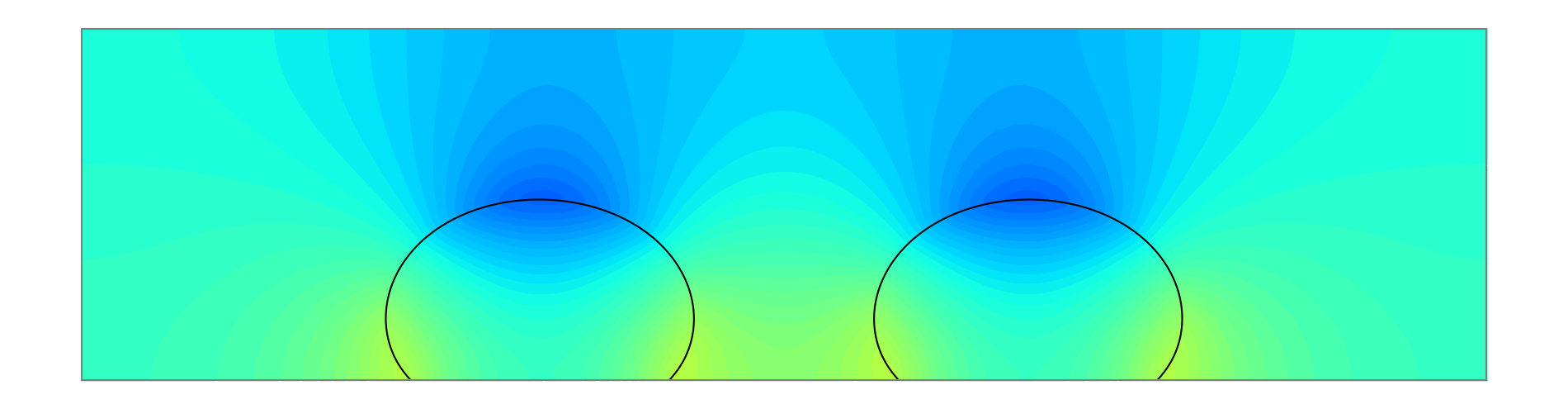}

      \includegraphics[width=0.49\textwidth]{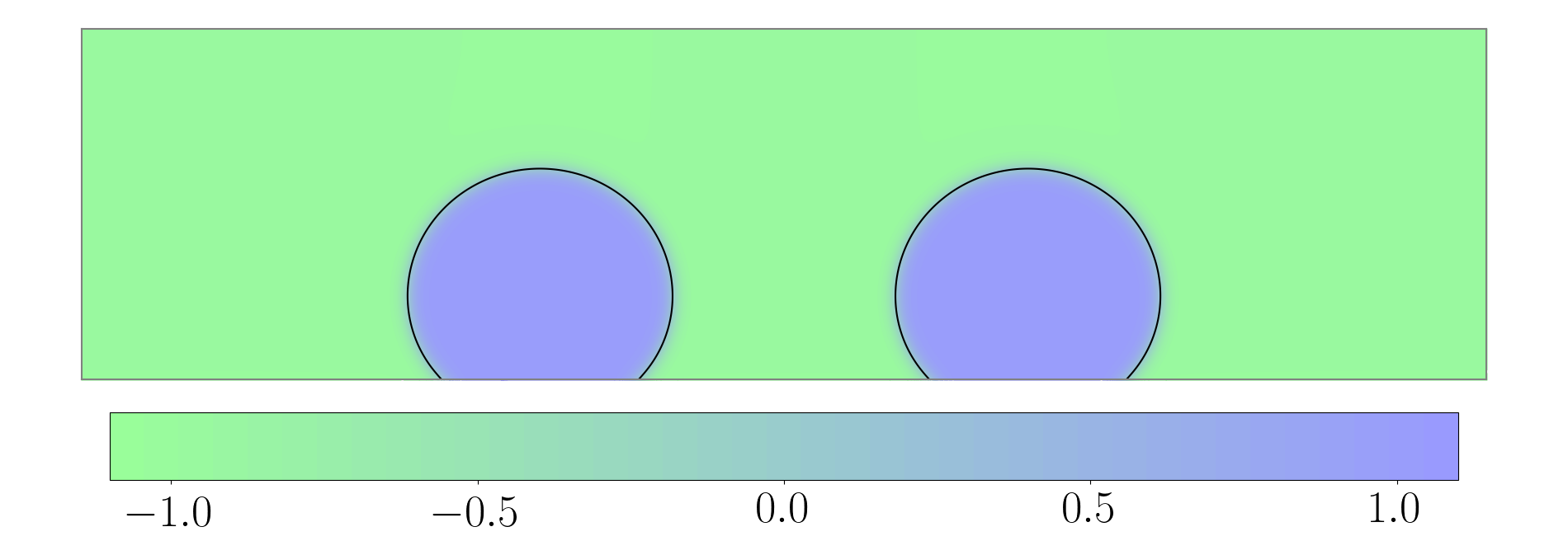} \includegraphics[width=0.49\textwidth]{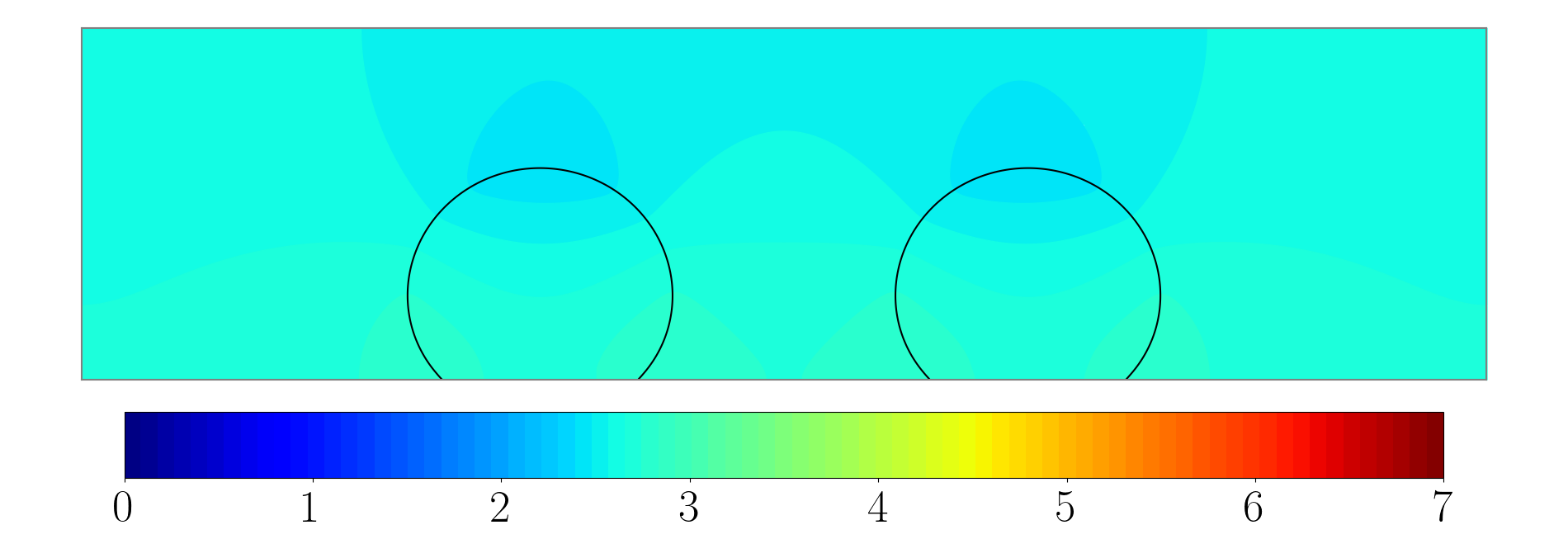}
  \caption{%
      Phase field and chemical potential when the contact angle is set to $3\pi/4$.
      The snapshots correspond to iterations 500, 1000, 1500, which correspond to times 1.47, 17.32, and 115.71.
  }%
  \label{fig:coalescence_3pio4}
\end{figure}

The evolution of the time step, of the number of recalculations,
and of the free energies is displayed in \cref{fig:coalescence_data}.
In all three cases,
the time step is refined several times at the first iteration,
which can be explained as follows.
On one hand, this refinement originates from the discontinuity of the initial condition at the interface between the two fluids.
On the other hand, in the cases $\theta = \pi / 4$ and  $\theta = 3 \pi / 4$,
it is also a consequence of the fact that the initial condition does not satisfy the wetting boundary condition.
Since the initial angle between the interface and the substrate is equal $\pi/2$,
the number of recalculations performed at the first iteration is higher for $\theta = \pi/4, 3\pi/4$
than for $\theta = \pi/2$.
After the initial refinement,
the time step steadily increases to its maximum allowed value
for $\theta = \pi/2$ and $\theta = 3 \pi/4$,
but when $\theta = \pi/4$ a second refinement occurs to capture the coalescence of the droplets.

In this latter case,
we observe,
simultaneous with the second refinement of the time step,
an increase in the rate of dissipation of free energy.
After the formation of a new stable interface,
the total free energy continues to decrease,
but more slowly,
as a new droplet,
formed by the merging of the two original droplets, moves towards its equilibrium position.
We clearly identify the
coalescence time by looking at the singularity in the curve corresponding to the mixing energy.
This energy increases before coalescence, as the interfaces are being
stretched, and it decreases steadily after. The wall energy, on the other
hand, decreases at first and increases in the later stage of the simulation.
As prescribed by \cref{algorithm:timeStepAlgo},
the time step detects the variations of free energy;
it decreases when the rate of variation of the total free energy increases,
and conversely.

For comparison purposes,
we also included in~\cref{fig:coalescence:total_free_energy,fig:coalescence:interior_free_energy,fig:coalescence:wall_free_energy} data corresponding to the
case where a fixed time step is used for the simulations presented in this section.
There does not currently exist any result with conditions on the time step that ensure the stability of OD2,
and we haven't been able to show stability results for OD2-W either.
In practice, we observed that
the time step required to ensure stability of OD2-W with the set of parameters we use in this test case
would lead to a very high computational cost.
We point out that,
contrary to what we expected,
the time step required to achieve stable integration in time with the modified wall energy~\eqref{eq:modified_wall_energy},
which we use in this paper,
seems to be generally smaller than with the cubic formulation~\eqref{eq:cahn-hilliard:wallEnergy}.
To keep the computational cost at a reasonable level,
we carried out the simulations with a fixed time step using the method OD1-W,
the greater stability of which enabled us to choose $\Delta t = 0.02$.
In \cref{fig:coalescence:total_free_energy,fig:coalescence:interior_free_energy,fig:coalescence:wall_free_energy},
we see that, for the same contact angle,
the curves corresponding to a fixed and an adaptive time step are almost undistinguishable.
The agreement is also very good at the level of the phase field and chemical potential,
although we do not present snapshots of the solutions obtained with a fixed time step.

The CPU times corresponding to the three contact angles considered
are presented in \cref{table:cpu_times_time_adaptation}.
Without adaptation, the simulations take significantly longer to run,
which is consistent with the fact that more iterations (20000) were necessary to reach the final time.
In addition, among the simulations that used an adaptive time-step,
the difference between the CPU times is also significant,
with the case $\theta = \pi/4$ taking more than twice as long as the case $\theta = \pi/2$.
\begin{table}[ht!]
    \centering
    \begin{tabular}{|l|l|l|}
        \hline
        Contact angle & Adaptive time step & Fixed time step \\
        \hline
        $\pi/4$ & 44:15:17 & 130:16:16 \\ \hline
        $\pi/2$ & 21:38:38 & 128:35:20 \\ \hline
        $3\pi/4$ & 31:12:18 & 122:33:28 \\ \hline
    \end{tabular}
    \caption{%
        CPU times (hh:mm:ss) using an Intel i7-3770 processor for the simulations presented in \cref{sub:time_adaptation_scheme} (two droplets on a substrate),
        with or without time-step adaptation.
        The method OD2-W was used for the simulations with an adaptive time step,
        and the method OD1-W was used for the simulations with a fixed time step.
        In both cases, an adaptive mesh was used, with the parameter $h_{\min}$ equal to $\varepsilon/10 = 0.001$.
    \label{table:cpu_times_time_adaptation}%
    }
\end{table}

\begin{figure}
  \centering
      \graphicspath{{figures/coalescence/}}
      \begin{minipage}[b]{.47\linewidth}
          \centering
          \includegraphics[width=\textwidth]{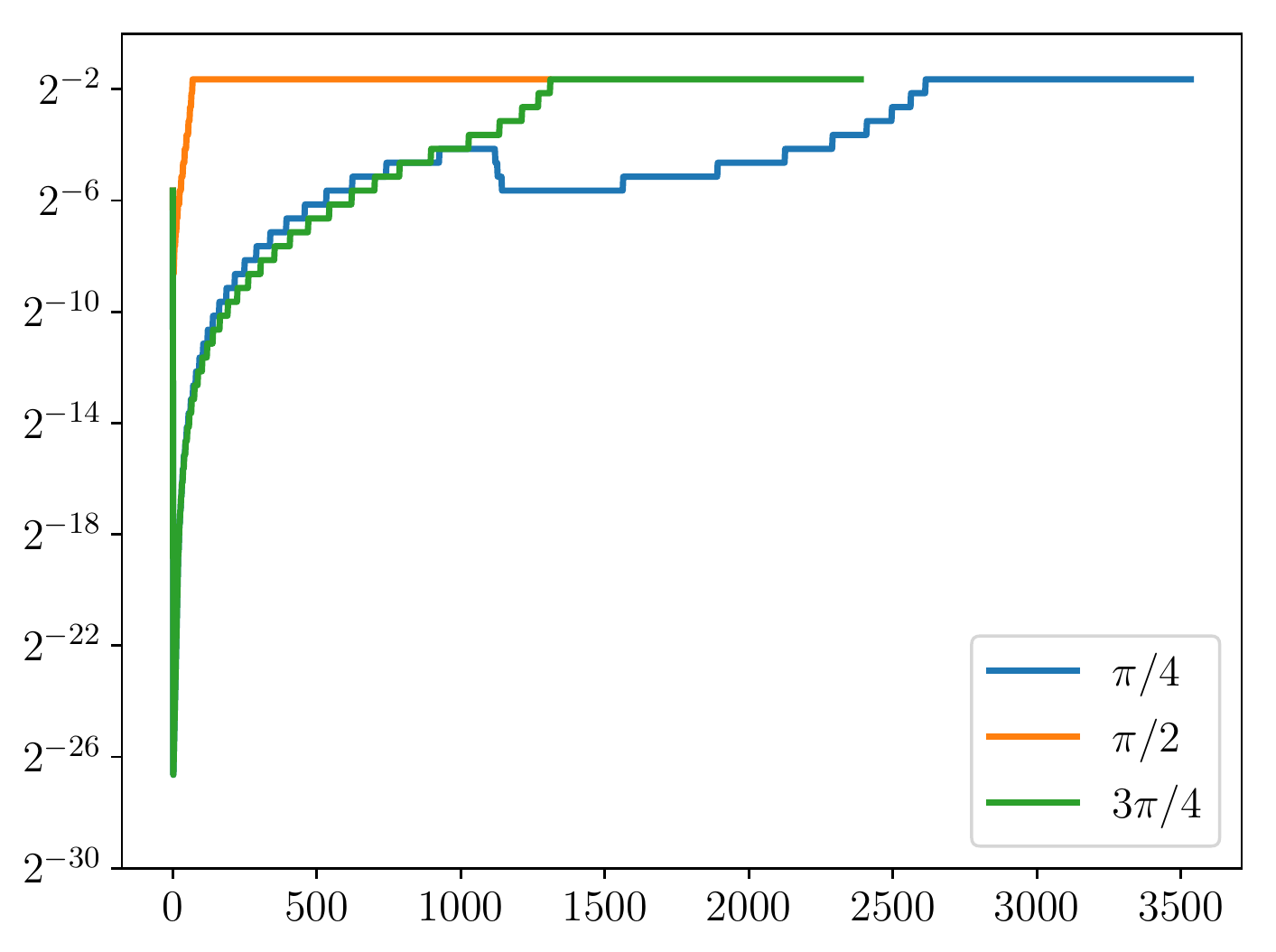}
          \subcaption{Time step vs iteration.}
      \end{minipage}%
      \begin{minipage}[b]{.47\linewidth}
          \centering
          \includegraphics[width=\textwidth]{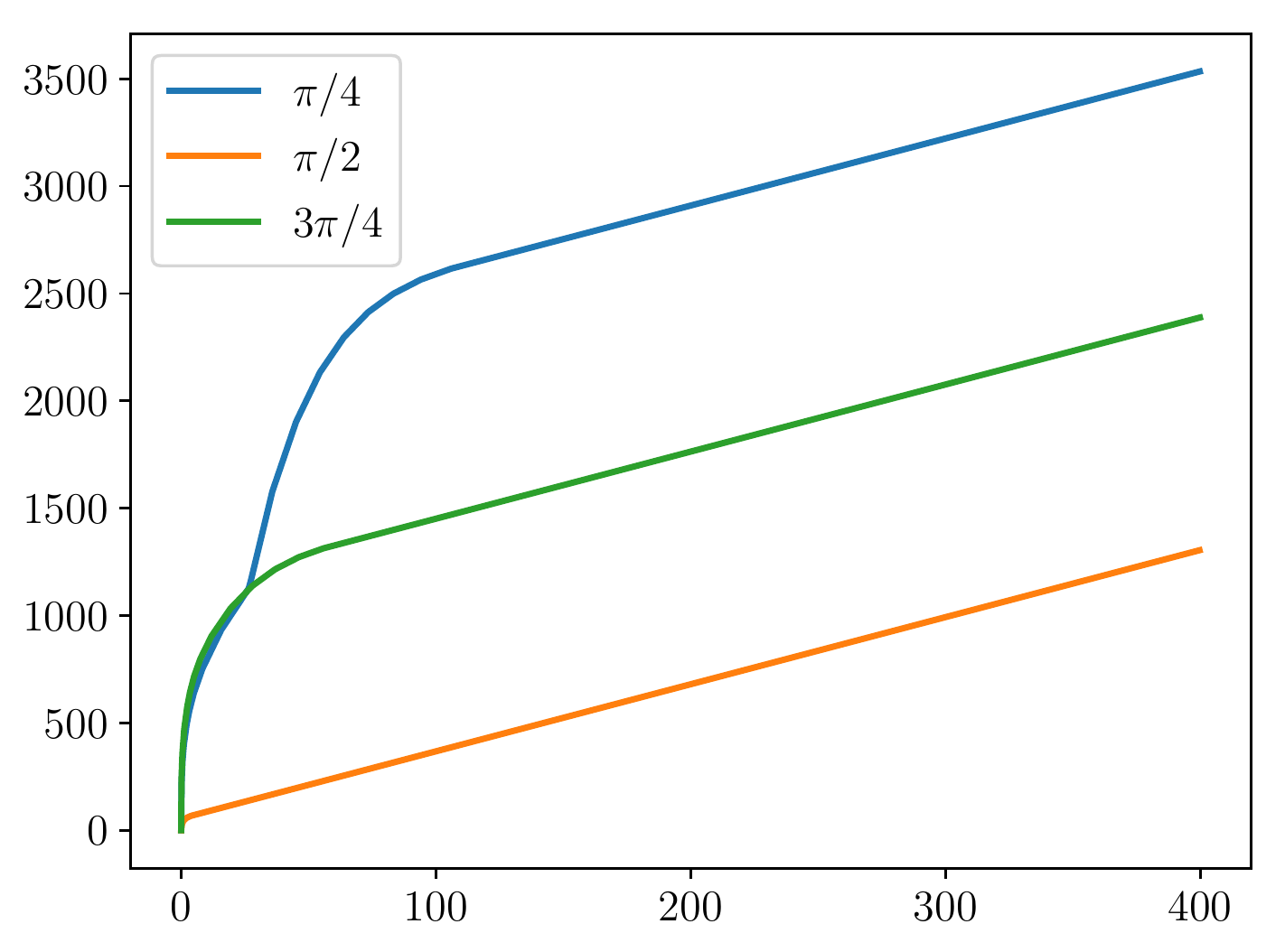}
          \subcaption{Iteration vs time.}
      \end{minipage}%

      \begin{minipage}[b]{.47\linewidth}
          \centering
          \includegraphics[width=\textwidth]{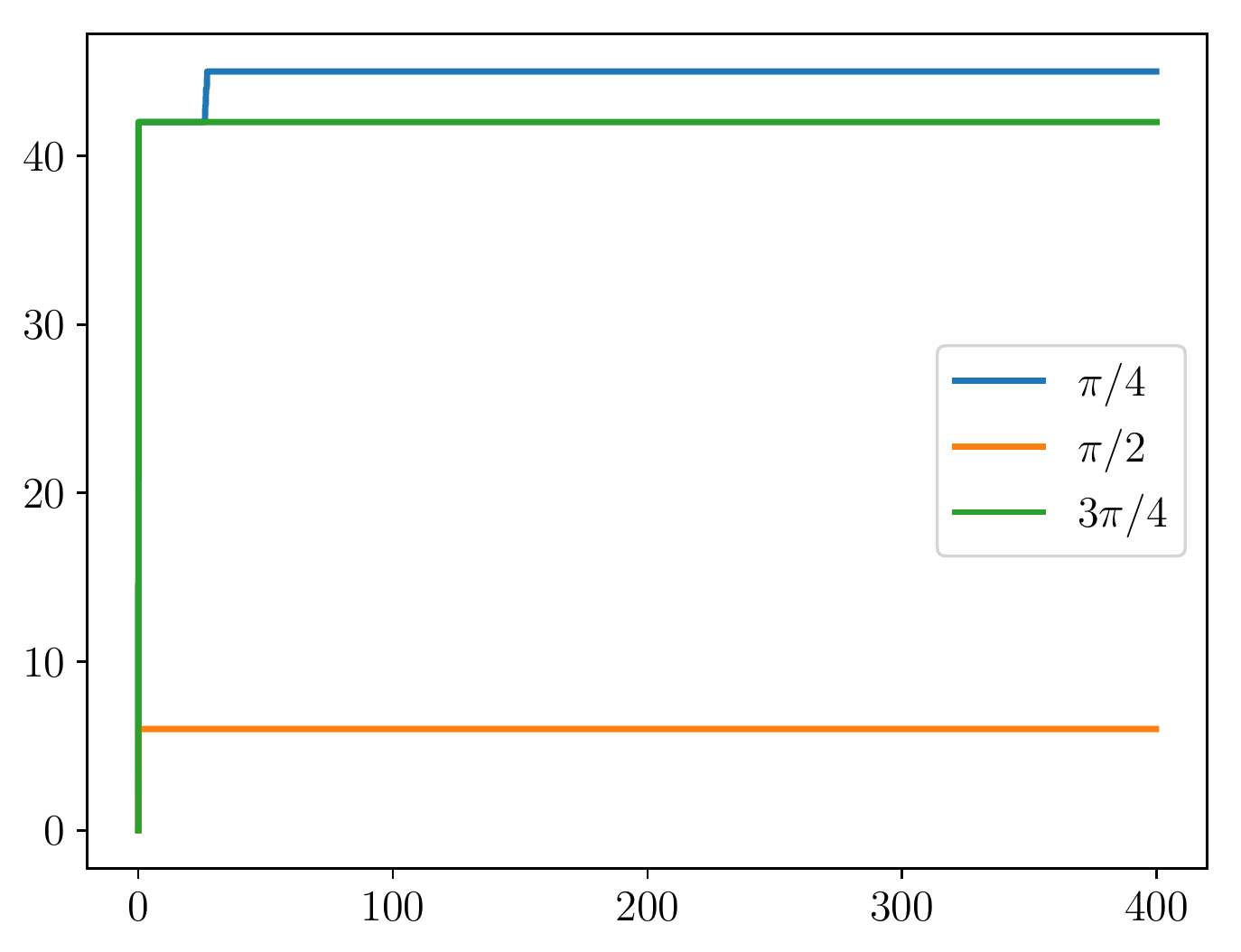}
          \subcaption{\# Recalculations vs time.}
      \end{minipage}%
      \begin{minipage}[b]{.47\linewidth}
          \centering%
          \includegraphics[width=\textwidth]{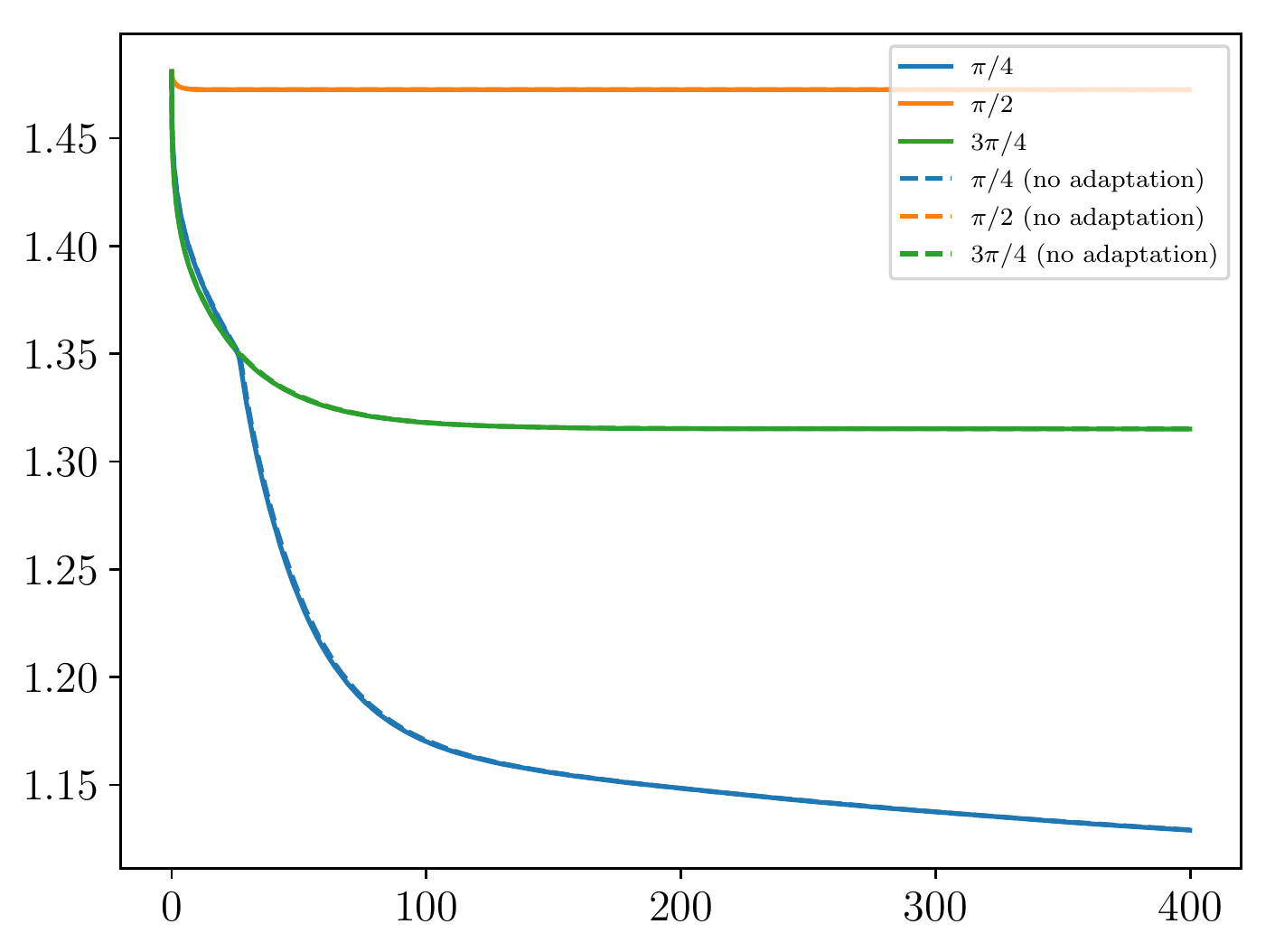}%
          \subcaption{Total free energy vs time.}%
          \label{fig:coalescence:total_free_energy}%
      \end{minipage}%

      \begin{minipage}[b]{.47\linewidth}
          \centering%
          \includegraphics[width=\textwidth]{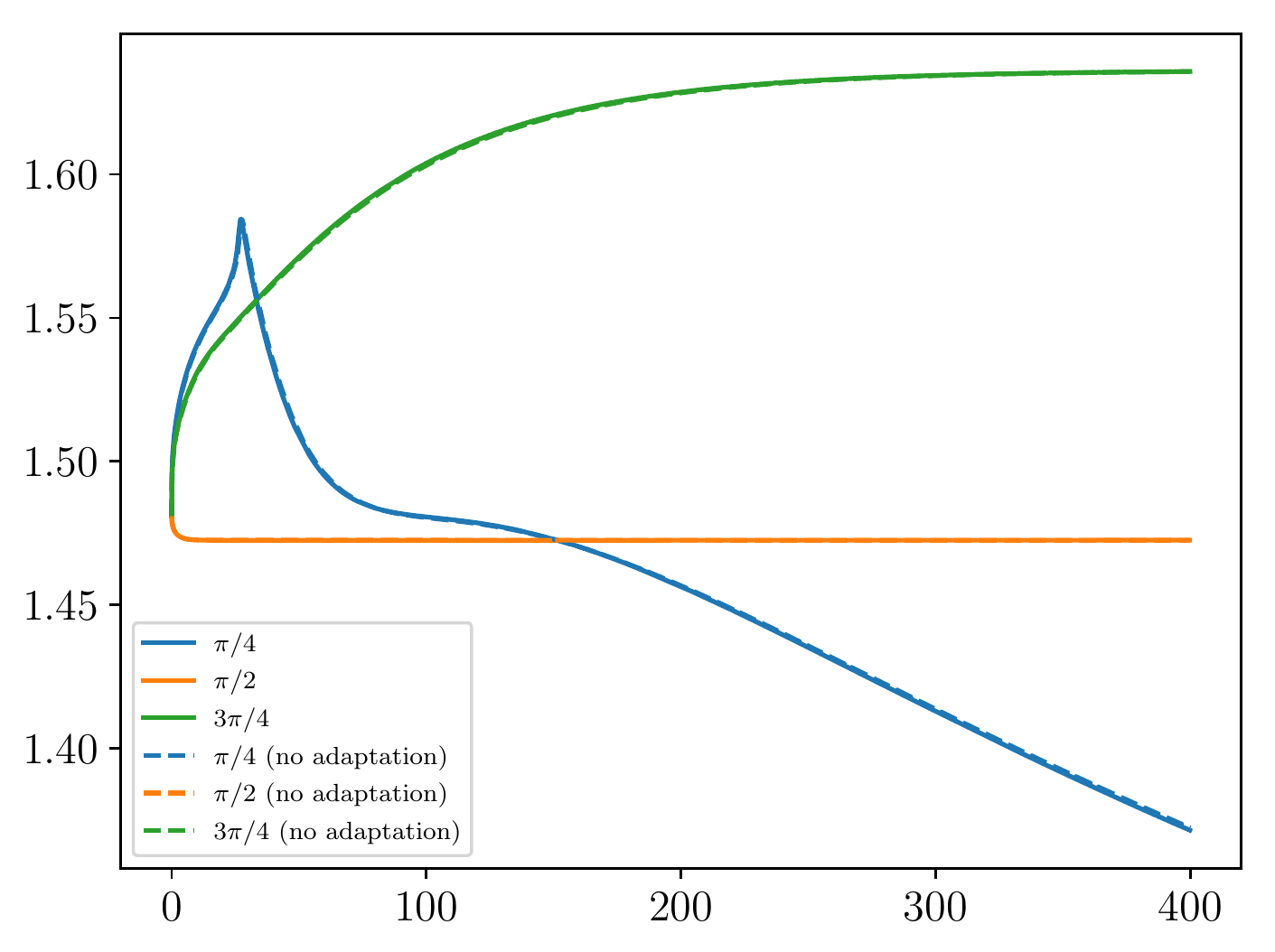}%
          \subcaption{Wall free energy vs time.}%
          \label{fig:coalescence:wall_free_energy}%
      \end{minipage}%
      \begin{minipage}[b]{.47\linewidth}
        \centering%
        \includegraphics[width=\textwidth]{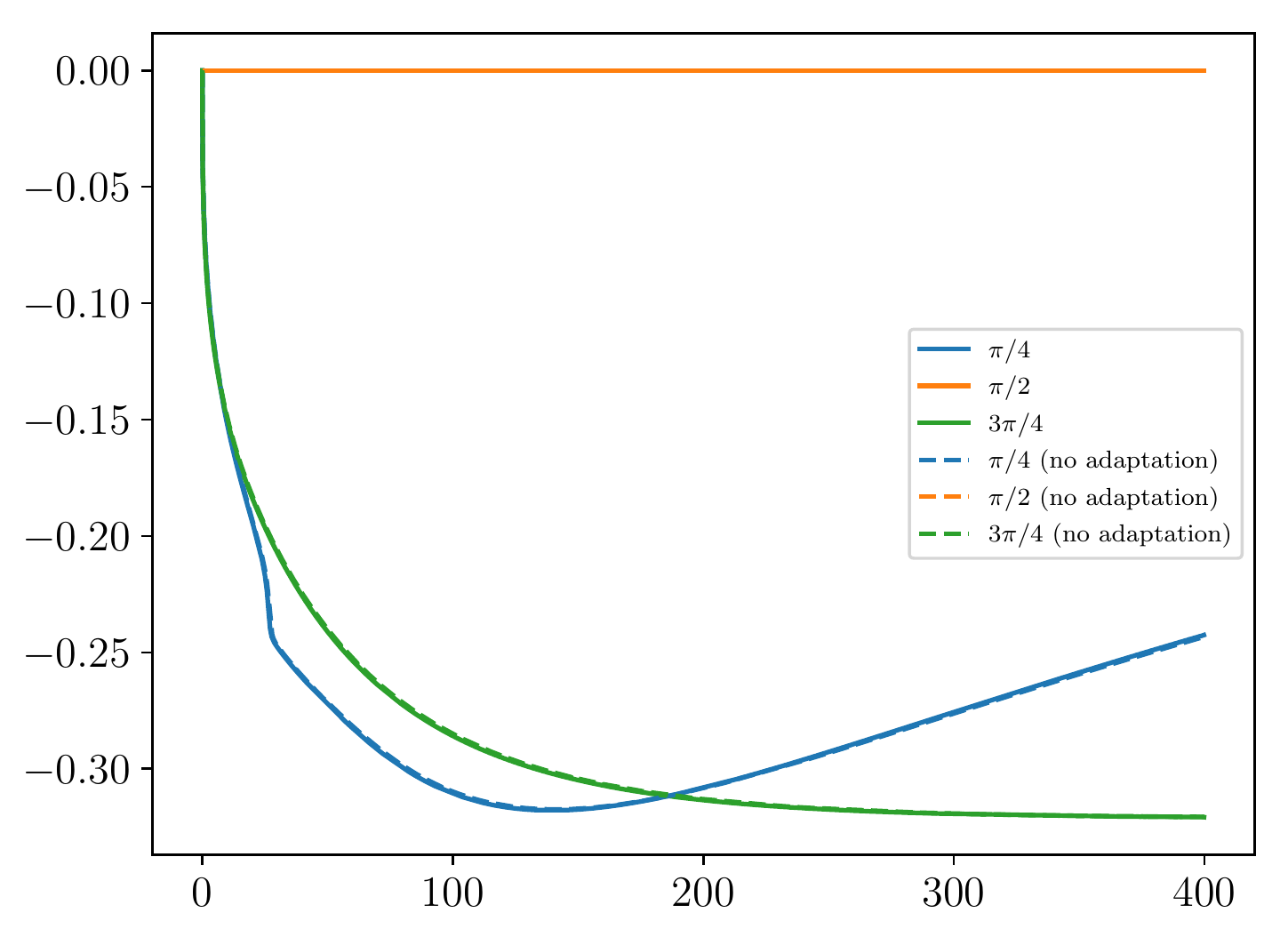}%
        \subcaption{Mixing free energy vs time.}%
        \label{fig:coalescence:interior_free_energy}%
    \end{minipage}

    \caption{%
        Simulation data for the numerical experiments presented in \cref{sub:time_adaptation_scheme} (two droplets on a substrate),
        when using the adaptive time-stepping scheme~\ref{algorithm:timeStepAlgo}.
        As expected, the total free energy decreases with time for all three values of the contact angle.
        In the case $\theta = \pi/4$,
        we note a peak in the mixing component of the free energy
        and a refinement of the time step at the coalescence time.
        See \cref{table:cpu_times_time_adaptation} for the associated computational costs;
        we notice that the simulations with an adaptive time step are about 3 times faster or more.
    }
    \label{fig:coalescence_data}
\end{figure}

\subsection{Wetting in complex geometries and with heterogeneous substrates}
\label{sub:wetting_in_complex_geometries_and_with_heterogeneous_substrates}

We now present the results of numerical experiments in more complicated and realistic settings,
in both 2D and 3D systems.
The results presented in this section were all obtained using linear basis functions.

\subsubsection{3D droplet on a chemically heterogeneous substrate}%
\label{ssub:3d_droplet_on_a_chemically_heterogeneous_substrate}

We study the dynamics of a 3D sessile droplet on a flat substrate with chemical
heterogeneities, i.e.\ the contact angle has a spatial dependence now, say
$\theta = \theta(x,y)$. This situation typically arises in electro-wetting
settings \cite{Lippmann}. It is widely accepted that the droplet shape can be
controlled using patterned substrates, e.g.~Ref~\cite{Raj2011, Zhang}, that
may also be modelled efficiently using a space varying contact
angle~\cite{Raj2011}. We consider chemical heterogeneities on the substrate
of the form
\begin{equation}
  \theta(x,y) = \theta_0 + a \cos(f_x \pi x) \cos(f_y \pi y),
  \label{chemTheta}
\end{equation}
with $\theta_0 = \frac{\pi}{2}$ the mean contact angle, $a=\frac{\pi}{6}$ the
amplitude, and $f_x=f_y=4$ the frequencies in $x$ and $y$ directions,
respectively. As initial condition we take a droplet of base radius $r_0 =
0.24$ centered at $\vect{x_0} = (0.5,0.5,0)$. The initial values of the phase
field are given as
\begin{equation}
    \phi_0(\vect{x}) = -\tanh \left(\frac{\|\vect{x} - \vect{x_0}\| - r_0}{\sqrt{2}\varepsilon} \right).
  \label{IC3D}
\end{equation}

Results are displayed in \cref{wellsIso}. The droplet, initially spherical,
spreads on the hydrophilic regions of the substrate, and retracts from the
hydrophobic patches. While we do not present any quantitative analysis of the
error in this case, we note that the wetting behaviour agrees qualitatively
with what one might expect intuitively from our understanding of wetting
phenomena. While it progresses towards equilibrium, the droplet adopts a
diamond-like shape.

For this test case, we used the method OD2-W with adaptation in space and time.
The parameters used were the following:
$b = 10^4$, $\varepsilon = 0.02$, $h_{\max} = 10 \, h_{\min} = 0.1$,
$\Delta t_0 = 0.0016$, $\Delta t_{\min} = 0$, $\Delta t_{\max} = 16 \, \Delta t_0$, $f = \sqrt{2}$, $\Delta E_{\max} = 10 \, \Delta E_{\min} = 0.0001$.
With these parameters,
the time step was refined only at the beginning of the simulation,
which is consistent with the absence of coalescence events in this case.
There were 24 recalculations at the first time step,
corresponding to a refinement of the time step by a factor $f^{24} = 4096$.

\begin{figure}
  \begin{center}
    \includegraphics[width=0.45\textwidth]{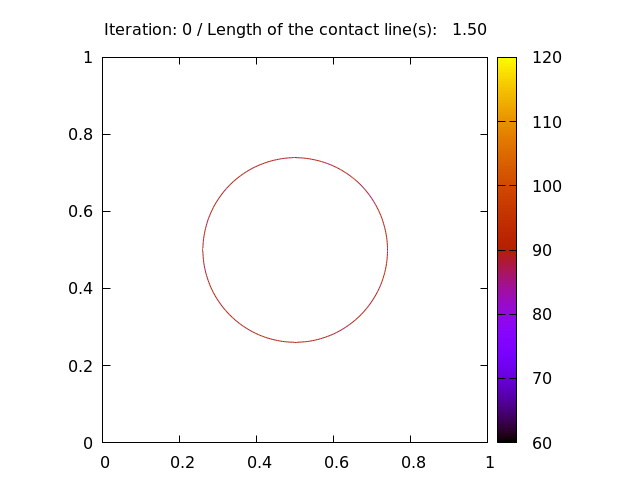}
    \includegraphics[width=0.45\textwidth]{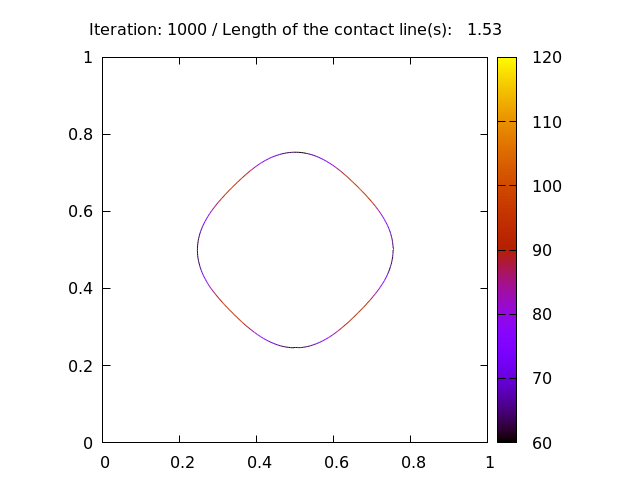}\\
    \includegraphics[width=0.45\textwidth]{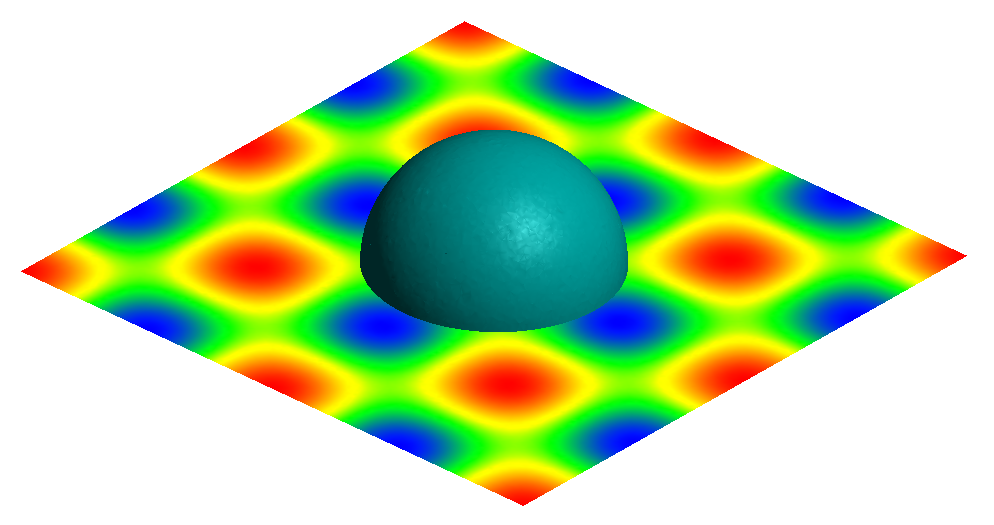}
    \includegraphics[width=0.45\textwidth]{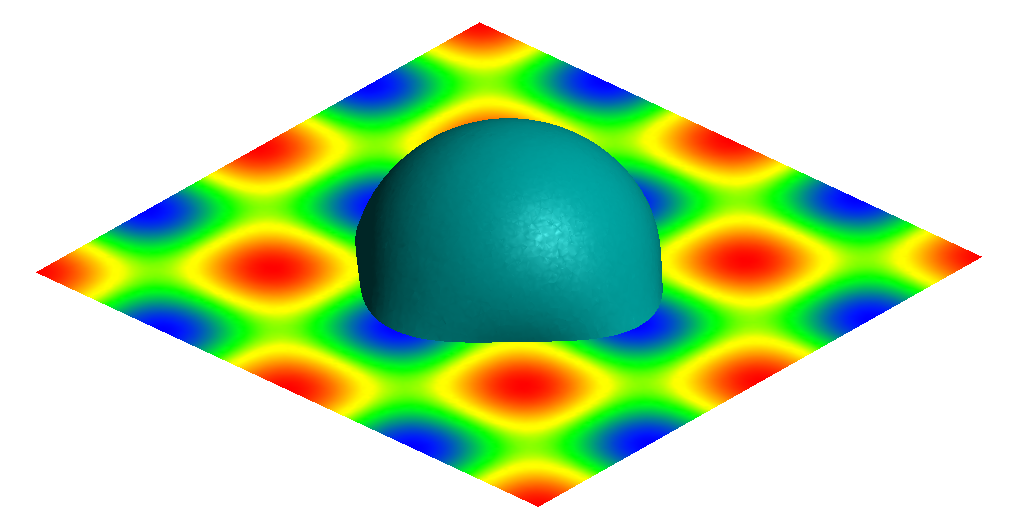}
    \includegraphics[height=0.25\textwidth]{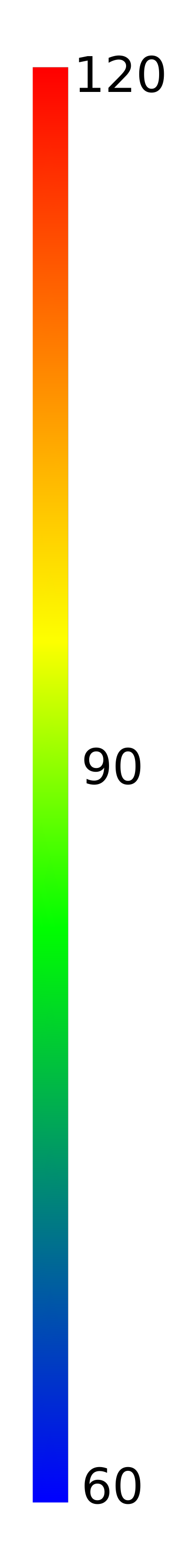}\\
  \end{center}
  \caption{Evolution of the contact line (top) and the isosurface $\phi=0$ (bottom) of the phase field,
      on a chemically heterogeneous substrate with a contact angle defined by \eqref{chemTheta}.
      The balance of the wall component and the mixing component of the free energy determines the motion of the drop.
      The field represented in the top figures is the value of the contact angle at the triple line.
      In the bottom figures, the field represented on the plane representing the substrate is the imposed contact angle.
      Interestingly, the heterogeneities of the substrate cause the length of the contact line to increase.}
  \label{wellsIso}
\end{figure}
\subsubsection{Diffusion in a 3D porous medium}

Here we consider a binary fluid in a model porous medium consisting of a cube
filled with spheres. The cube has edges of length 1, and the spheres have
radius $0.1$ and are located at positions $(1.5, 1.5, 1.5) + 2 \Delta(i,j,k)$
with $\Delta= 1/7$ and $i,j,k \in \{0,1,2\}$. We take all the substrates to
be neutral, i.e.~$\theta = \frac{\pi}{2}$, and the initial condition is the
same as used before, defined by \cref{IC3D}. In addition, we include an
inflow boundary condition at the bottom of the cube to represent a pore where
liquid can be pumped in. The radius of this pore is  $0.1$ and is located at
$(0.5, 0.5,0)$. This boundary condition can be incorporated by imposing
\begin{align}
    \label{eq:test_case_spheres_boundary_condition_chem_pot}
    \grad \mu \cdot \vect{n} &= -10, \quad \phi =1,
\end{align}
which models the situation when the component $\phi=1$ is pumped into the domain.
Under these conditions, we study how the flow is affected by the geometry of the domain.
Our results are depicted in \cref{porousIso,porousMacro}.

The imposed contact angle at the spheres is $\pi/2$, forcing the isosurface
to stay normal to the spheres as long as these are not completely covered.
Because of the boundary
condition~\eqref{eq:test_case_spheres_boundary_condition_chem_pot}, the mass
increases linearly, and the free energy increases, in agreement with
\cref{Mlaw,eq:energy_conservation_cahn_hilliard}. This case study demonstrates the ability of our method to
easily tackle complex geometries.
The parameters used for this test case are the same as in \cref{ssub:3d_droplet_on_a_chemically_heterogeneous_substrate},
except that we employed the fixed time step $\Delta t = 0.001$.

\begin{figure}
  \begin{center}
    \includegraphics[width=0.32\textwidth]{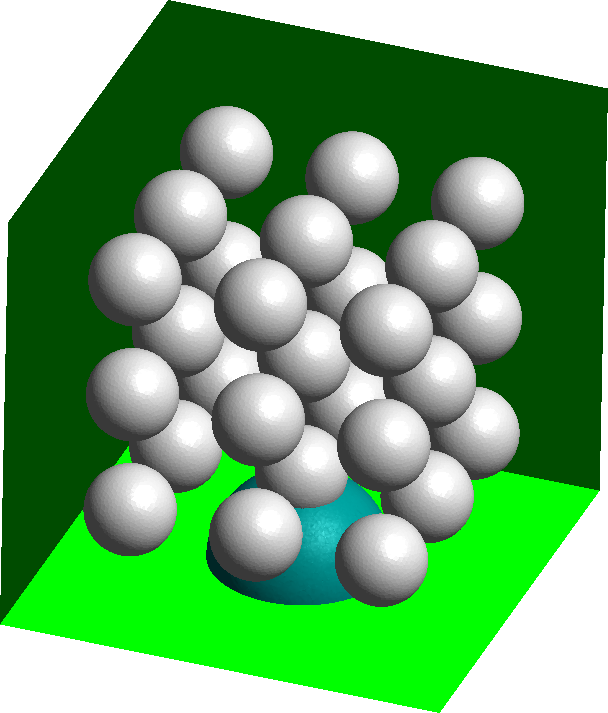}
    \includegraphics[width=0.32\textwidth]{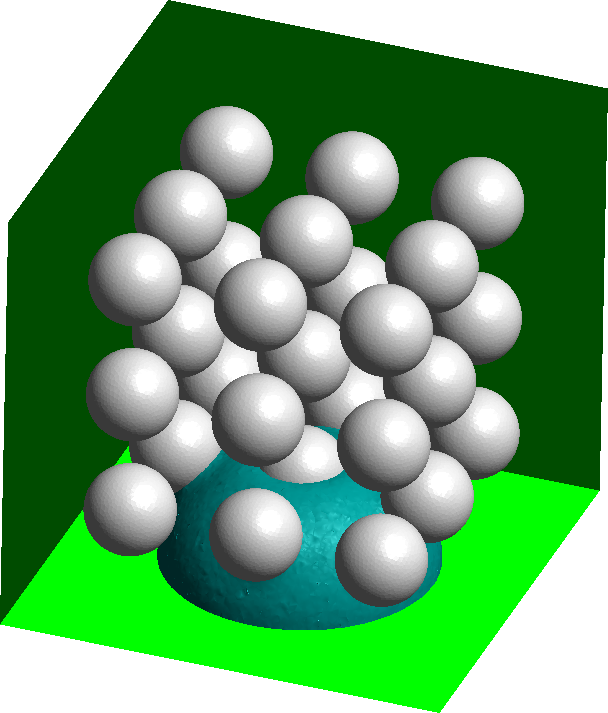}
    \includegraphics[width=0.32\textwidth]{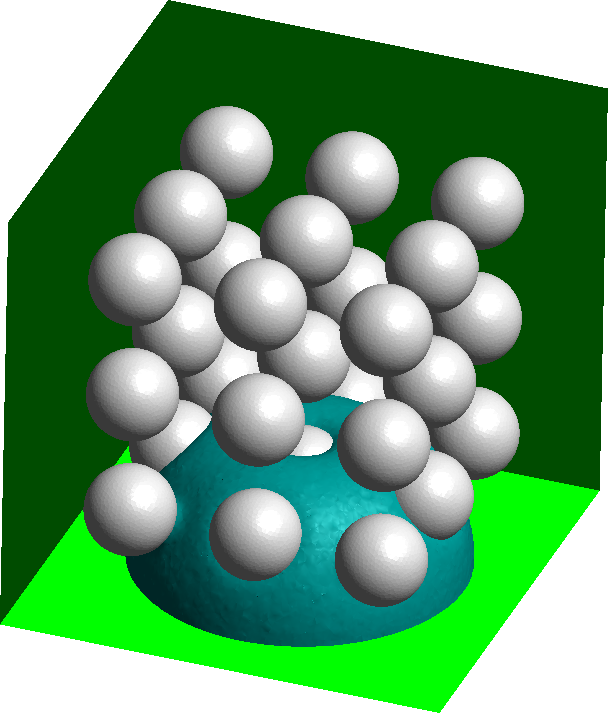}
    \includegraphics[width=0.32\textwidth]{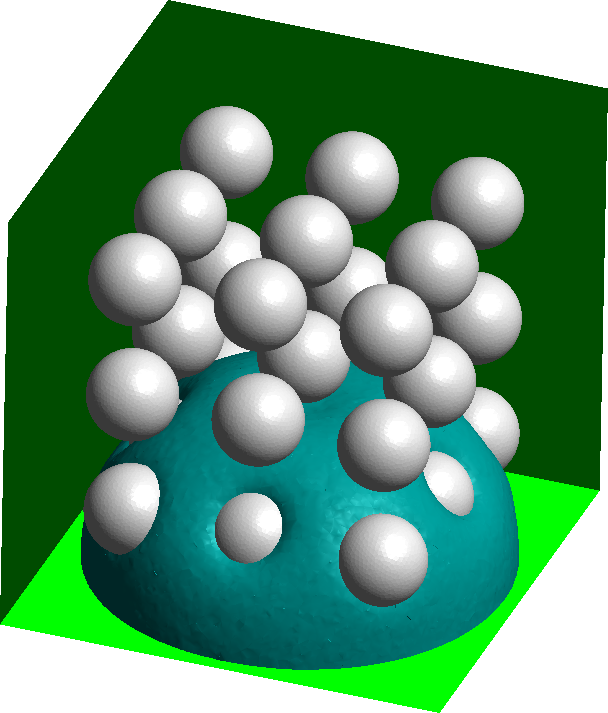}
    \includegraphics[width=0.32\textwidth]{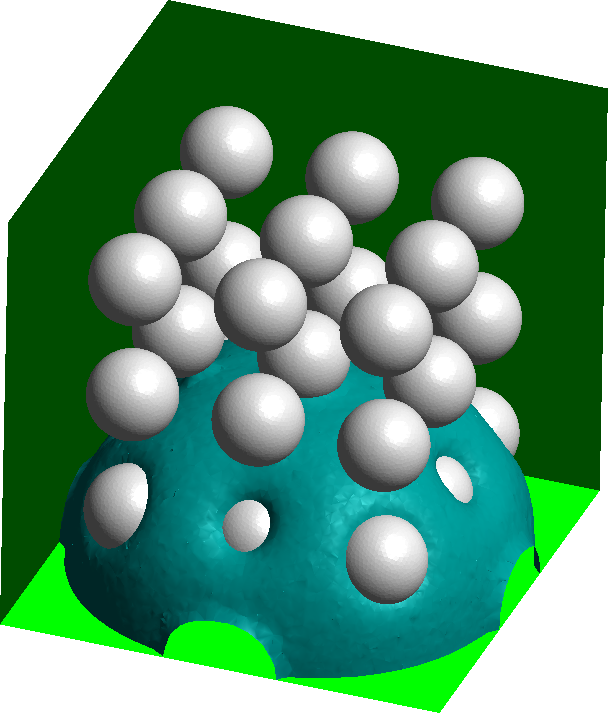}
  \end{center}
  \caption{Evolution of the isosurface $\phi=0$ of the phase field when a constant flux is imposed at the bottom boundary;
    The pictures correspond to iterations 0, 200, 400, 800 and 1000.
    Note that, because of the neutral boundary condition imposed at the spheres,
    the isosurface tends to stay normal to them as long as they are not completely covered.
  }
  \label{porousIso}
\end{figure}

\begin{figure}
  \centering {%
        \begin{minipage}[b]{.5\linewidth}
            \includegraphics[width=\textwidth]{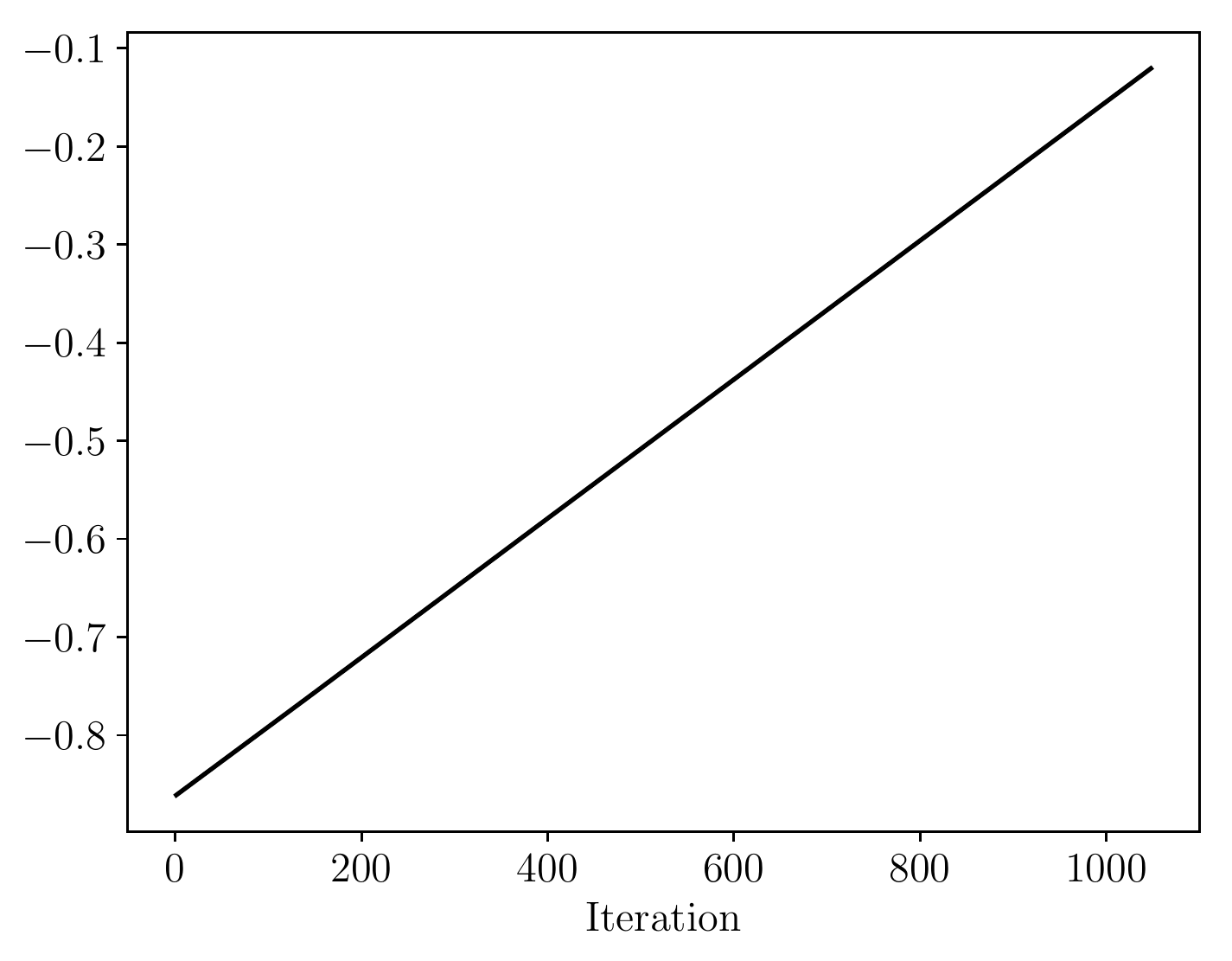}
            \subcaption{Mass, \cref{Mlaw}.}
        \end{minipage}%
        \begin{minipage}[b]{.49\linewidth}
            \includegraphics[width=\textwidth]{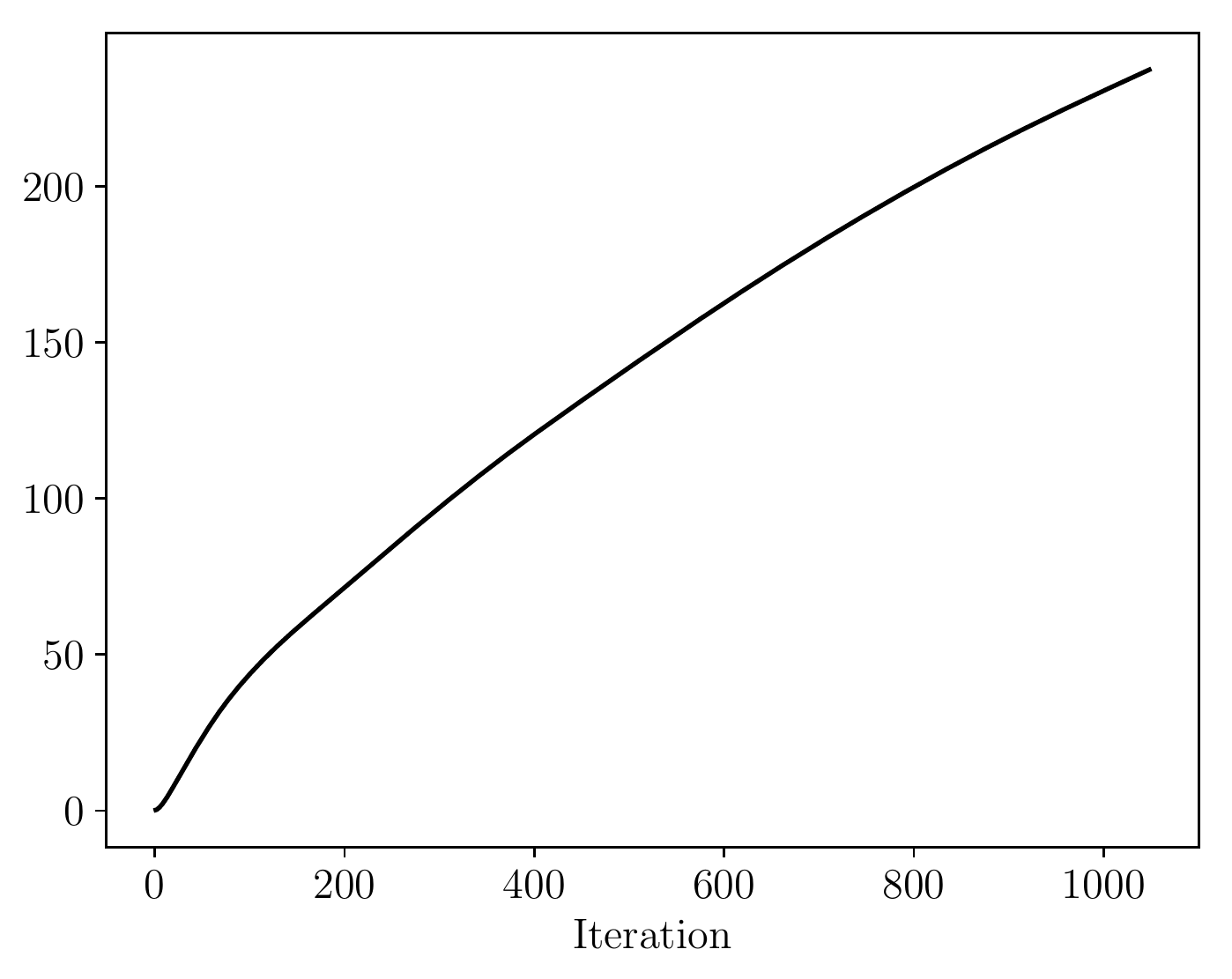}
            \subcaption{Free energy, \cref{eq:energy_conservation_cahn_hilliard}.}
        \end{minipage}%
        \caption{%
            Evolution of $M(\phi)$ and $E_m(\phi)$ as a function of time in the case of a 3D flow through a porous medium (Figure \ref{porousIso}).
            In this case, the mass increases linearly because we impose a constant mass inflow at the pore.
            The free energy increases as well, because the size of the interface grows,
            in agreement with both the mass and energy laws \eqref{Mlaw} and \eqref{eq:energy_conservation_cahn_hilliard}.
        }
        \label{porousMacro}
    }
\end{figure}

\subsubsection{Nucleation processes with complex boundaries}%
\label{ssub:nucleation_processes_with_complex_boundaries}

The last problem we study is the process of phase separation in a domain with
complex boundary characterized by different length scales. Specifically, we
consider a domain defined by the coastline of the two islands that form the
United Kingdom and Ireland. Starting from a satellite black and white
picture, we extracted the isolines that define the contour of the different
islands, which we passed to the FreeFem++ mesh generator to obtain a
triangular mesh (for this, we based our code on a FreeFem++ example for the
Leman lake).
At the boundary we consider the contact angles $\theta = \pi/4, \pi/2, 3 \pi/4$,
and we assume that the phase field is initially set to a random
value at each grid point, drawn from a random normal distribution with
variance $0.1$.
A fixed mesh was used for this simulation,
and the parameters used were $b = 1000$, $\varepsilon = 0.02$, $\Delta E_{\min} =
0.02$, $\Delta E_{\max} = 0.04$, $f=\sqrt{2}$, $\Delta t_{\min} = 0$, $\Delta t_{\max} = 1$.

The evolution of the phase field and of the chemical potential in the case $\theta = \pi/4$,
obtained with the adaptive time-stepping scheme~\ref{algorithm:timeStepAlgo},
is presented in \cref{fig:nucleation_iso}.
For each of the contact angles considered,
we also ran a simulation with the fixed time step $\Delta t = 0.01$,
using the method OD1-W instead of OD2-W
to benefit from the stabilizing effect introduced by the philic numerical dissipation of OD1-W.
We note in particular that OD2-W is unstable for the selected value of $\Delta t$,
with oscillations appearing in the energy curves from the first iterations,
and that the time step would have to be reduced significantly to ensure stability.
The final configurations (time 500) are presented in~\cref{fig:nucleation:final_solution} for the three contact angles considered.
We observe that the final configurations
are different depending on whether or not an adaptive time step is used,
which can be attributed to the high sensitivity of the solution to perturbations of the initial condition chosen for this test case;
the areas where separation of the phases first occurs is influenced by numerical errors in the early stages of the simulation.

Simulation data are presented in \cref{fig:nucleationMacro}.
With an adaptive time-stepping scheme,
it appears from \cref{fig:nucleationMacro} (a) that, overall,
the time step increases steadily as the frequency of coalescence events decreases.
At specific times,
the time step decreases slightly in order to accurately capture the evolution.
As expected,
the total free energy has a roughly constant negative slope when plotted against the iteration number.
Here too, we observe a small discrepancy between the fixed and adaptive cases,
which is consistent with differences observed at the final time in~\cref{fig:nucleation:final_solution}.

The CPU times corresponding to the simulations presented in this section are displayed in~\ref{table:cpu_times_uk}.
For the parameters selected,
the adaptive time-stepping scheme leads to a lower computational cost.
\begin{table}[hb]%
    \centering
    \begin{tabular}{|l|l|l|}
        \hline
        Contact angle & Adaptive time step & Fixed  time step \\
        \hline
        $\pi/4$ & 36:13:15 & 65:48:44  \\ \hline
        $\pi/2$ & 32:58:01 & 65:24:10 \\ \hline
        $3\pi/4$ & 36:43:13 & 67:00:50 \\ \hline
    \end{tabular}
    \caption{%
        CPU times (hh:mm:ss) using an Intel i7-3770 processor for the simulations presented in \cref{ssub:nucleation_processes_with_complex_boundaries} (nucleation in a domain with complex boundaries),
        with or without time-step adaptation.
        The method OD2-W was used for the simulations with an adaptive time step,
        and the method OD1-W was used for the simulations with a fixed time step.
        In all cases, we used a fixed mesh with mesh size $h = 0.01$
        (the size of the domain is roughly 5 by 5)
        and $P1$ elements, leading to 181587 unknowns.
        The parameter $\varepsilon$ was set to $0.02$.
    }%
    \label{table:cpu_times_uk}
\end{table}
This test demonstrates the advantage of using a finite element approach,
as it would have been very complicated to solve the CH equation in the
geometries we consider here with e.g.~a spectral method or finite differences.

\begin{figure}
    \centering
    \includegraphics[width=0.19\textwidth]{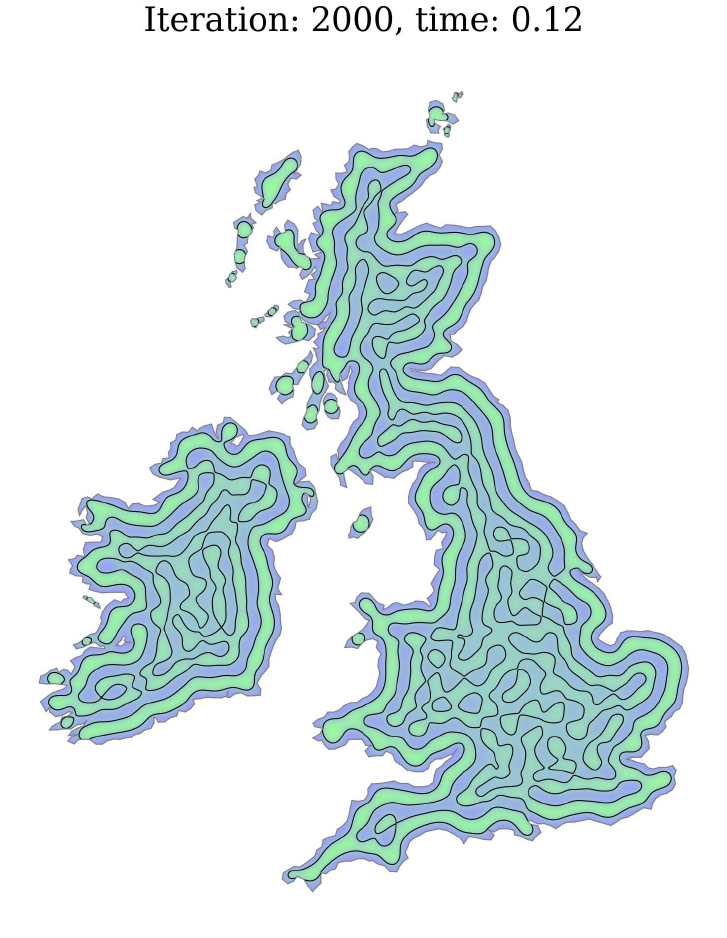}
    \includegraphics[width=0.19\textwidth]{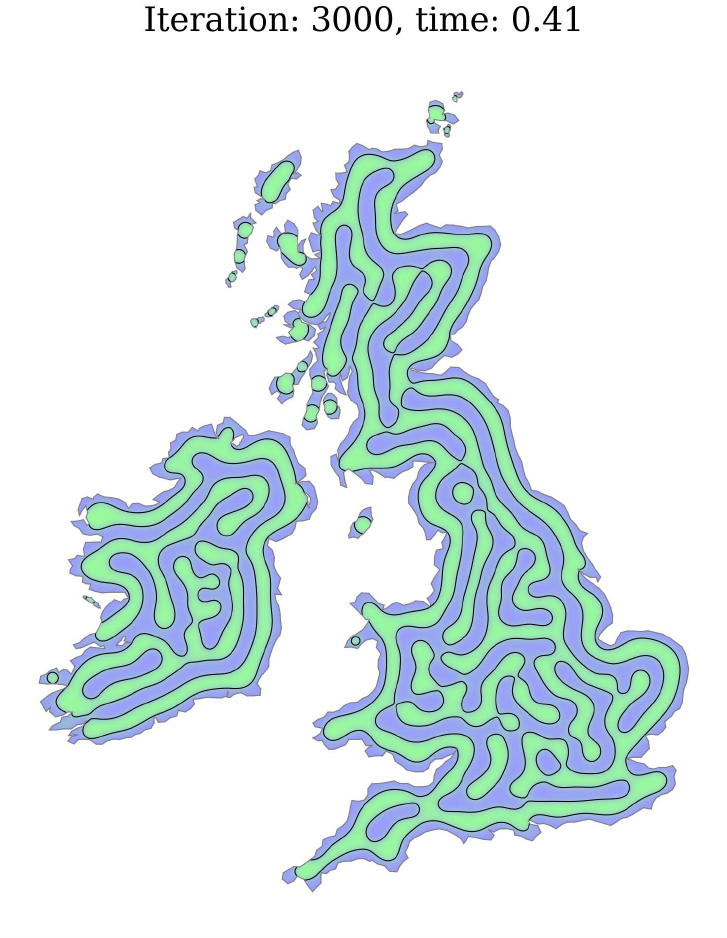}
    \includegraphics[width=0.19\textwidth]{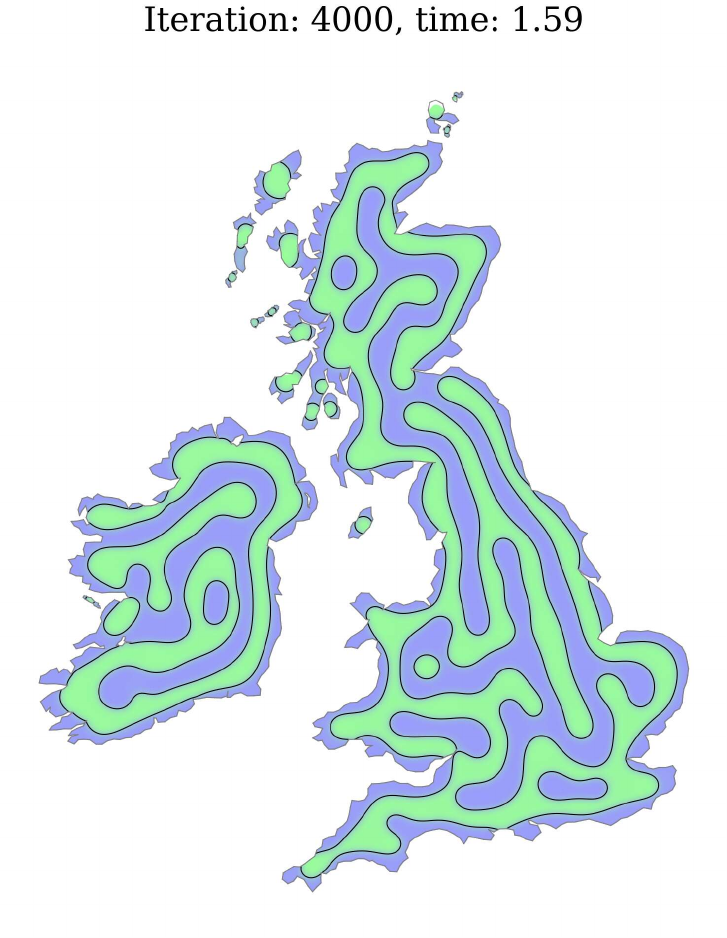}
    \includegraphics[width=0.19\textwidth]{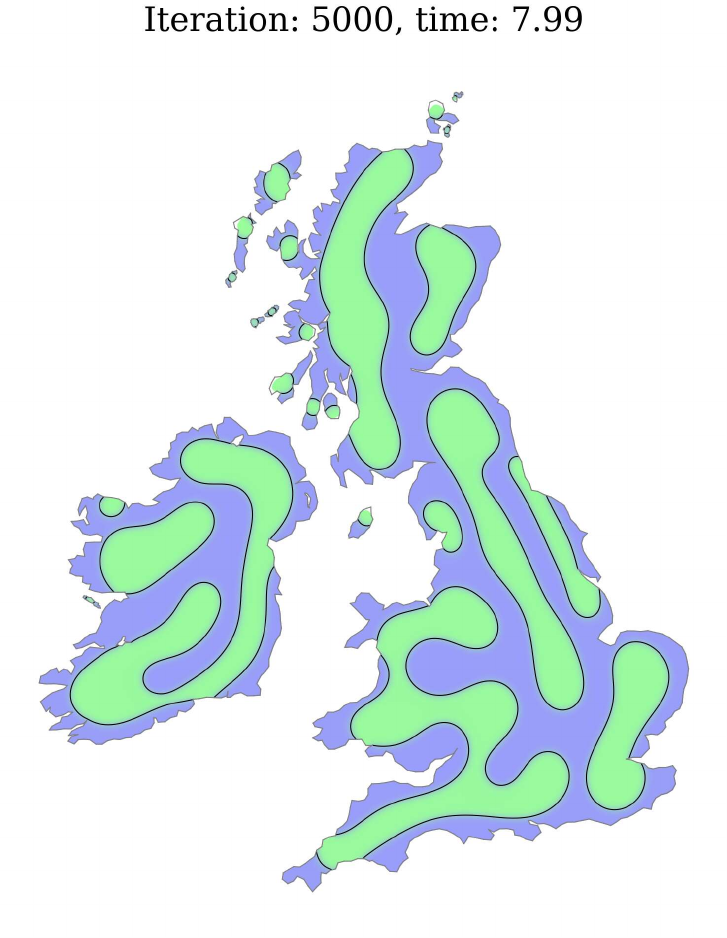}
    \includegraphics[width=0.19\textwidth]{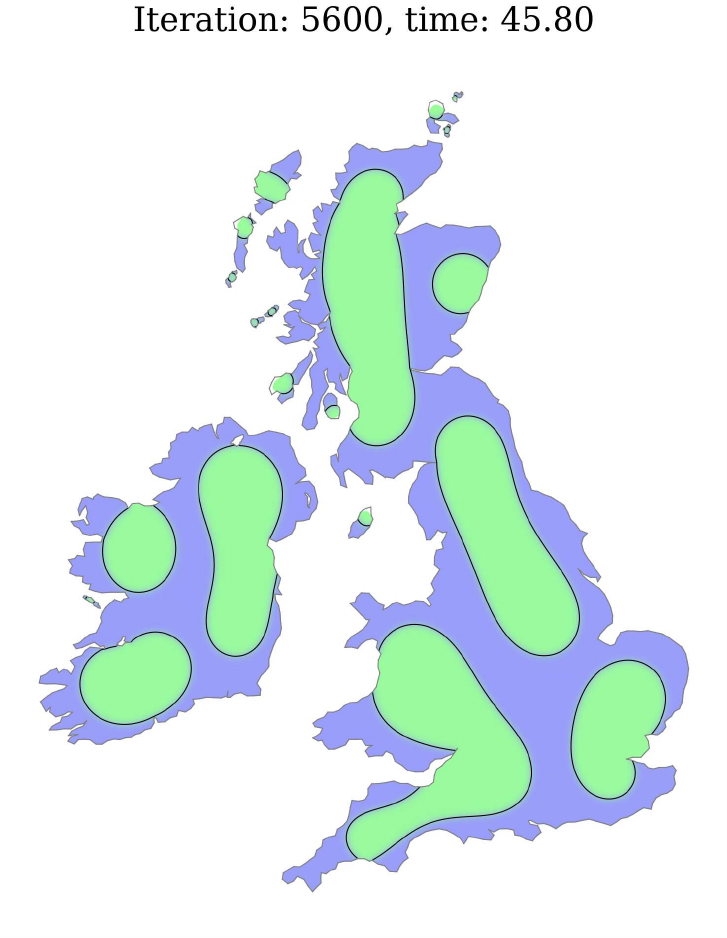}

    \includegraphics[width=0.19\textwidth]{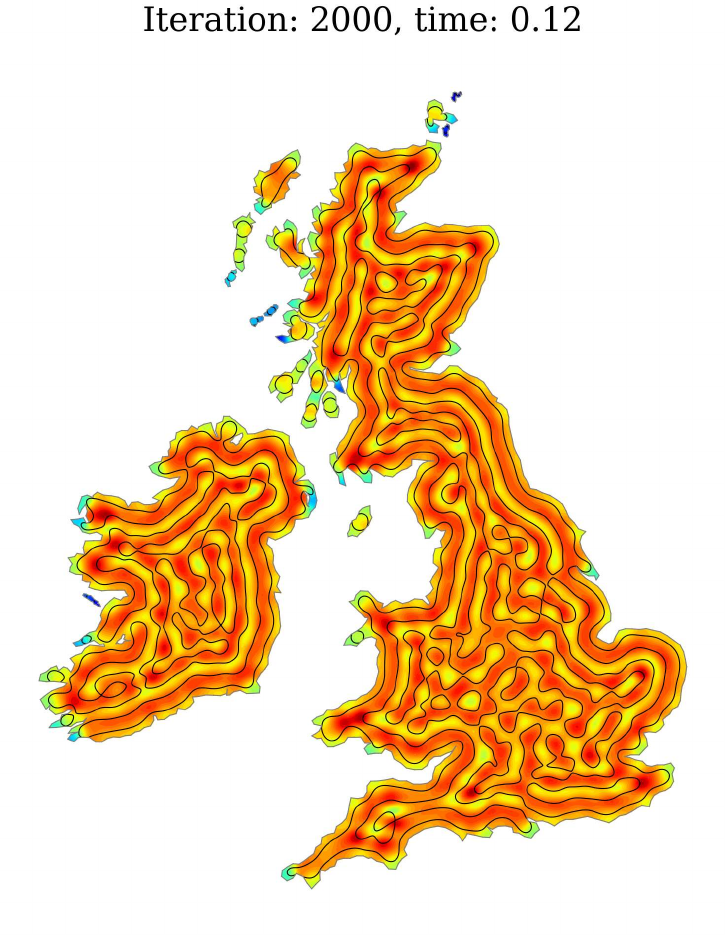}
    \includegraphics[width=0.19\textwidth]{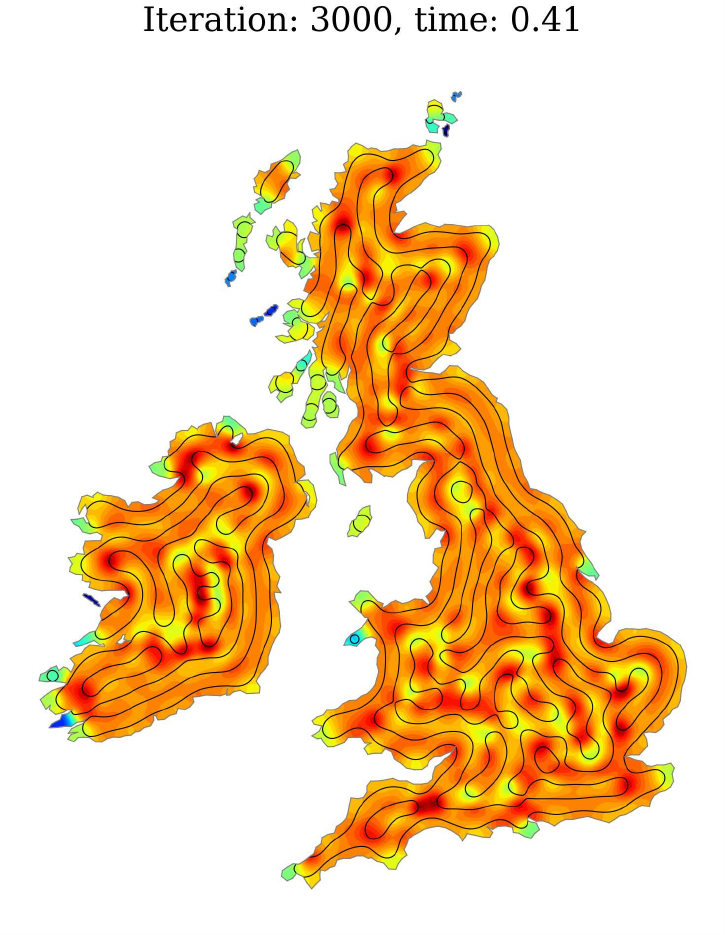}
    \includegraphics[width=0.19\textwidth]{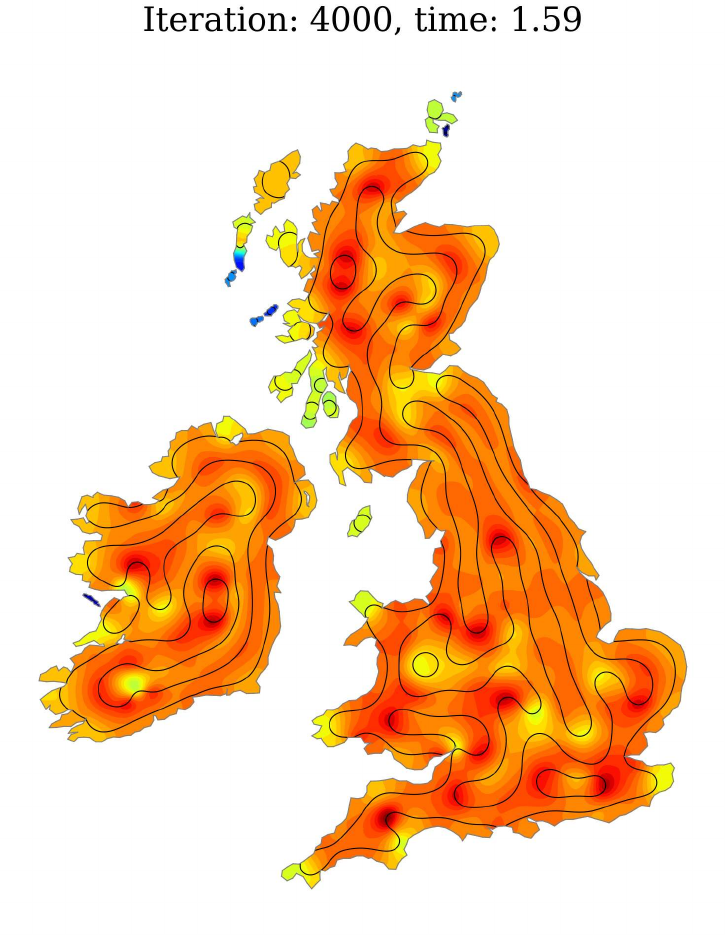}
    \includegraphics[width=0.19\textwidth]{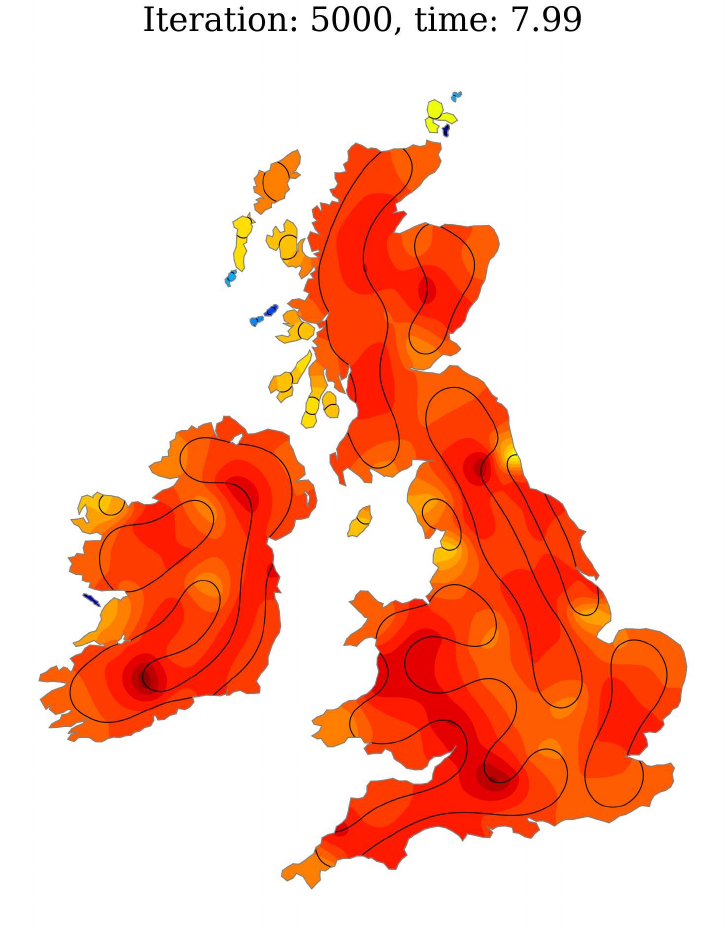}
    \includegraphics[width=0.19\textwidth]{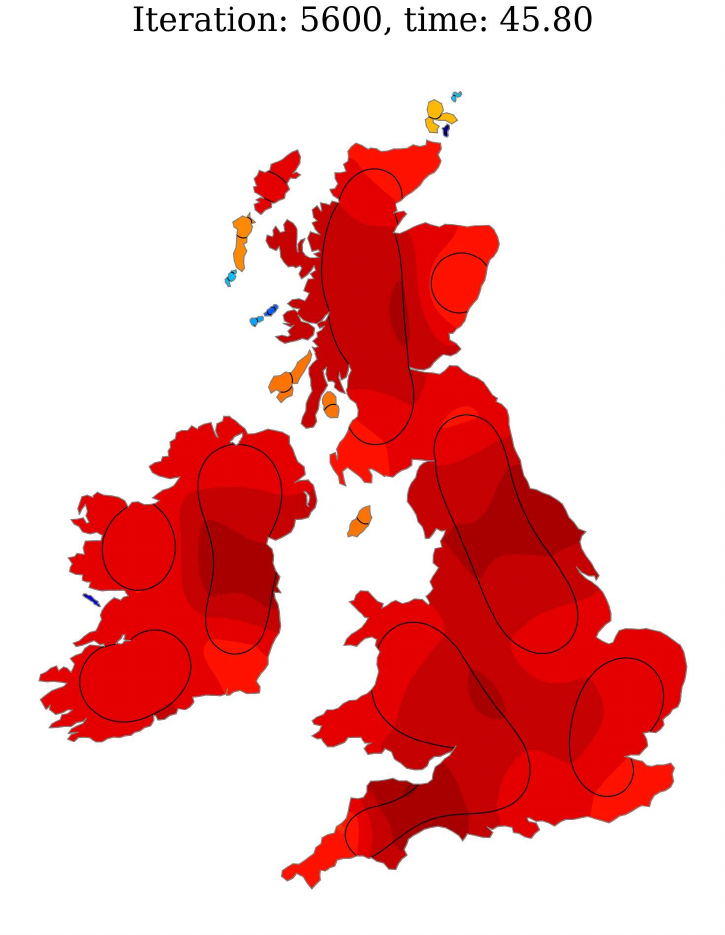}
    \caption{Evolution of the phase field and chemical potential for
    the nucleation in a domain with complex boundaries,
    when starting from a random distribution.
    As before, blue corresponds to $\phi=1$ and green to $\phi=-1$.
    The contact angle imposed at the boundaries is $\theta = \frac{\pi}{4}$.
  }
  \label{fig:nucleation_iso}
\end{figure}
\begin{figure}
  \centering {%
      \fcolorbox{black}{white}{\begin{minipage}[b]{.31\linewidth}
            \centering
            \includegraphics[width=.45\linewidth]{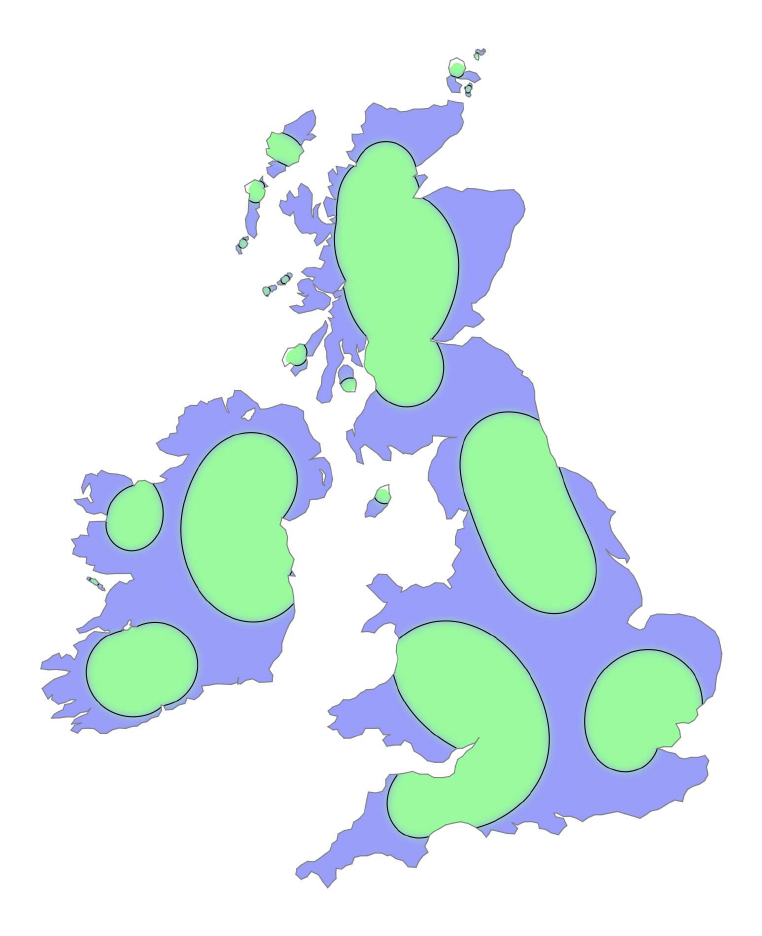}
            \includegraphics[width=.45\linewidth]{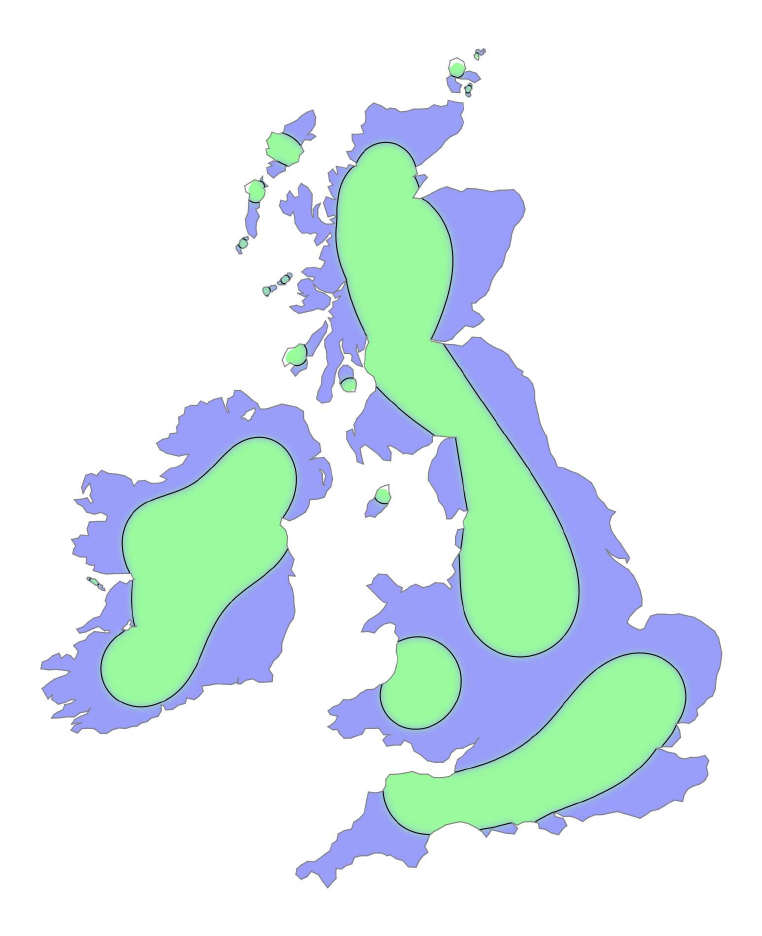}
            \subcaption{$\theta = \pi/4$}
        \end{minipage}}\hspace{.1cm}%
        \fcolorbox{black}{white}{\begin{minipage}[b]{.31\linewidth}
            \centering
            \includegraphics[width=.45\linewidth]{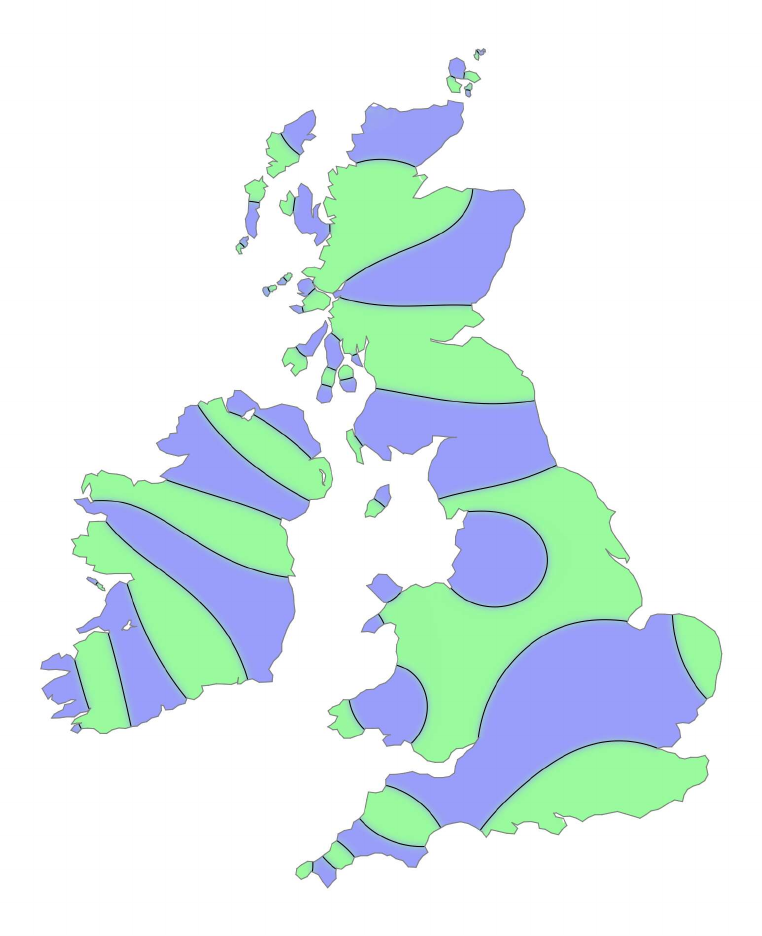}
            \includegraphics[width=.45\linewidth]{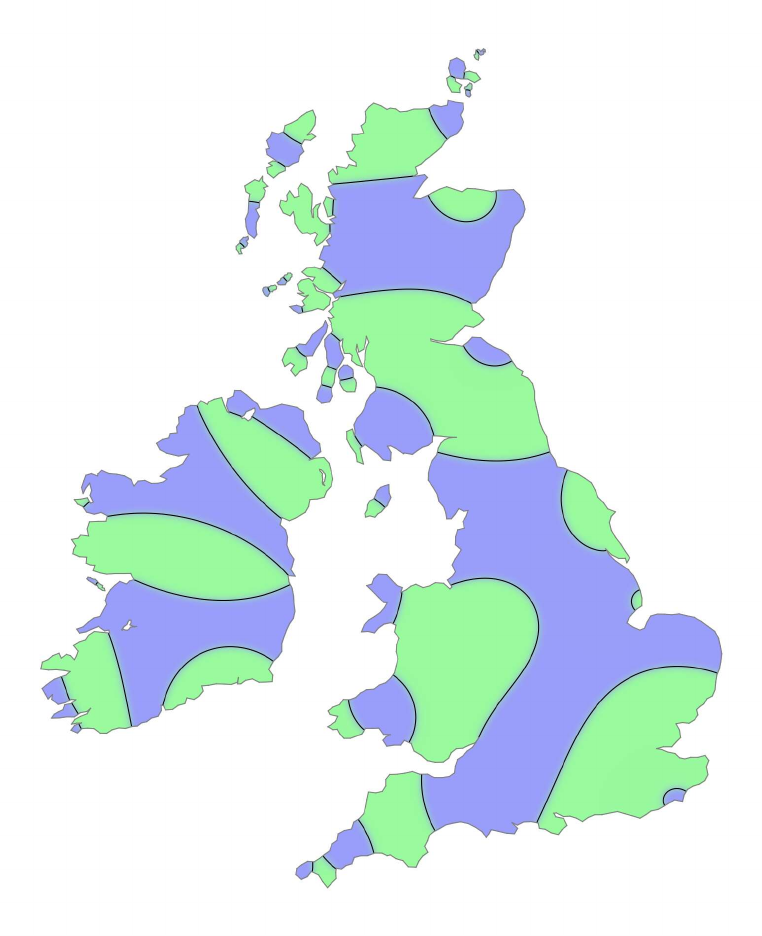}
            \subcaption{$\theta = \pi/2$}
        \end{minipage}}\hspace{.1cm}%
        \fcolorbox{black}{white}{\begin{minipage}[b]{.31\linewidth}
            \centering
            \includegraphics[width=.45\linewidth]{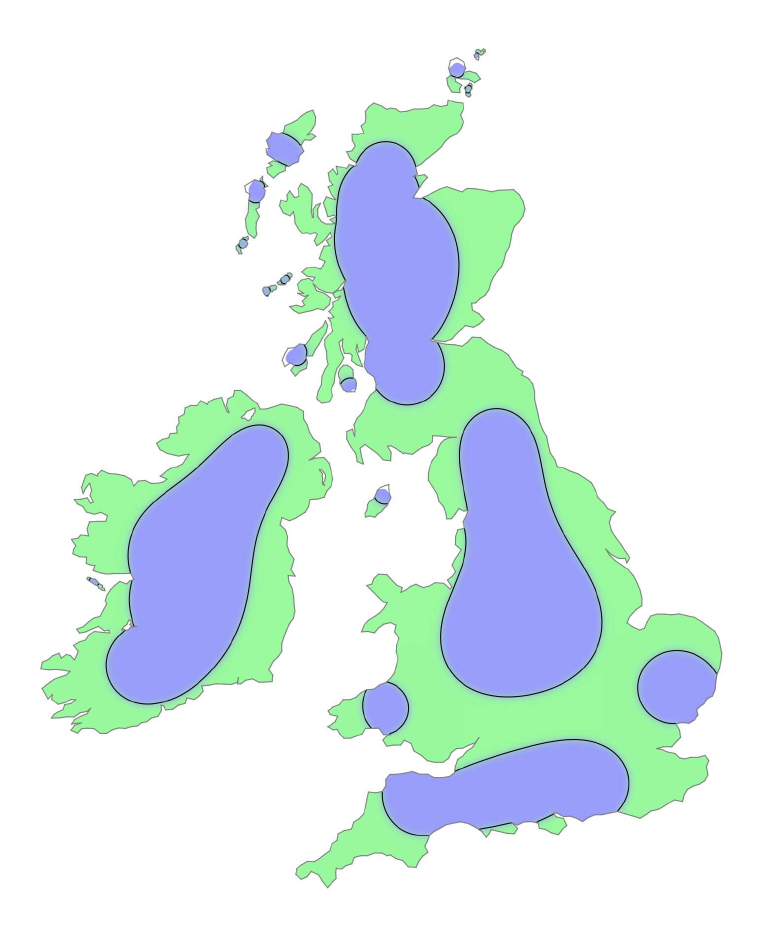}
            \includegraphics[width=.45\linewidth]{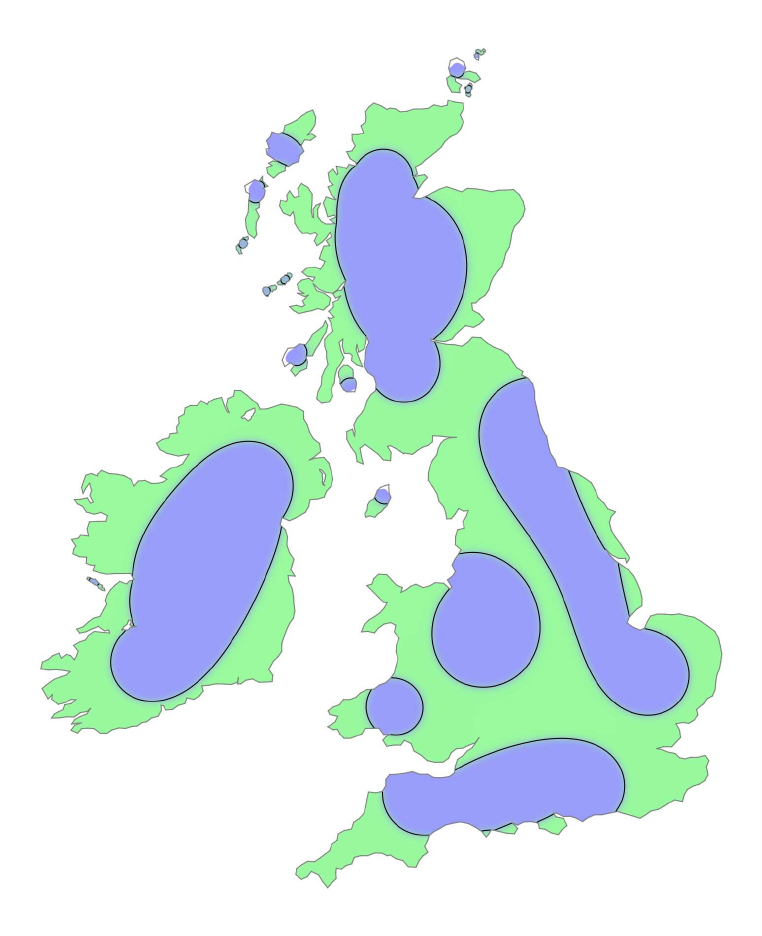}
            \subcaption{$\theta = 3\pi/4$}
        \end{minipage}}%
        \caption{%
            Comparison of the solutions at the final time,
            with (left) and without (right) time-step adaptation.
            At the initial time, the phase field is set to a random value at each grid point,
            drawn from a random normal distribution with variance $0.1$.
            Although one could expect the solutions for $\theta = \pi/4$ and $\theta = 3\pi/4$ to differ only by a sign,
            this is not the case.
            There is also a significant difference between the solutions obtained with and without time-step adaptation.
            These differences can be explained by the sensitivity of the evolution to perturbations of the initial condition
            and to numerical errors in the early stage of the simulation.
        }
        \label{fig:nucleation:final_solution}
    }
\end{figure}

\begin{figure}
  \centering
        \begin{minipage}[b]{.48\linewidth}
            \centering
            \includegraphics[width=\linewidth]{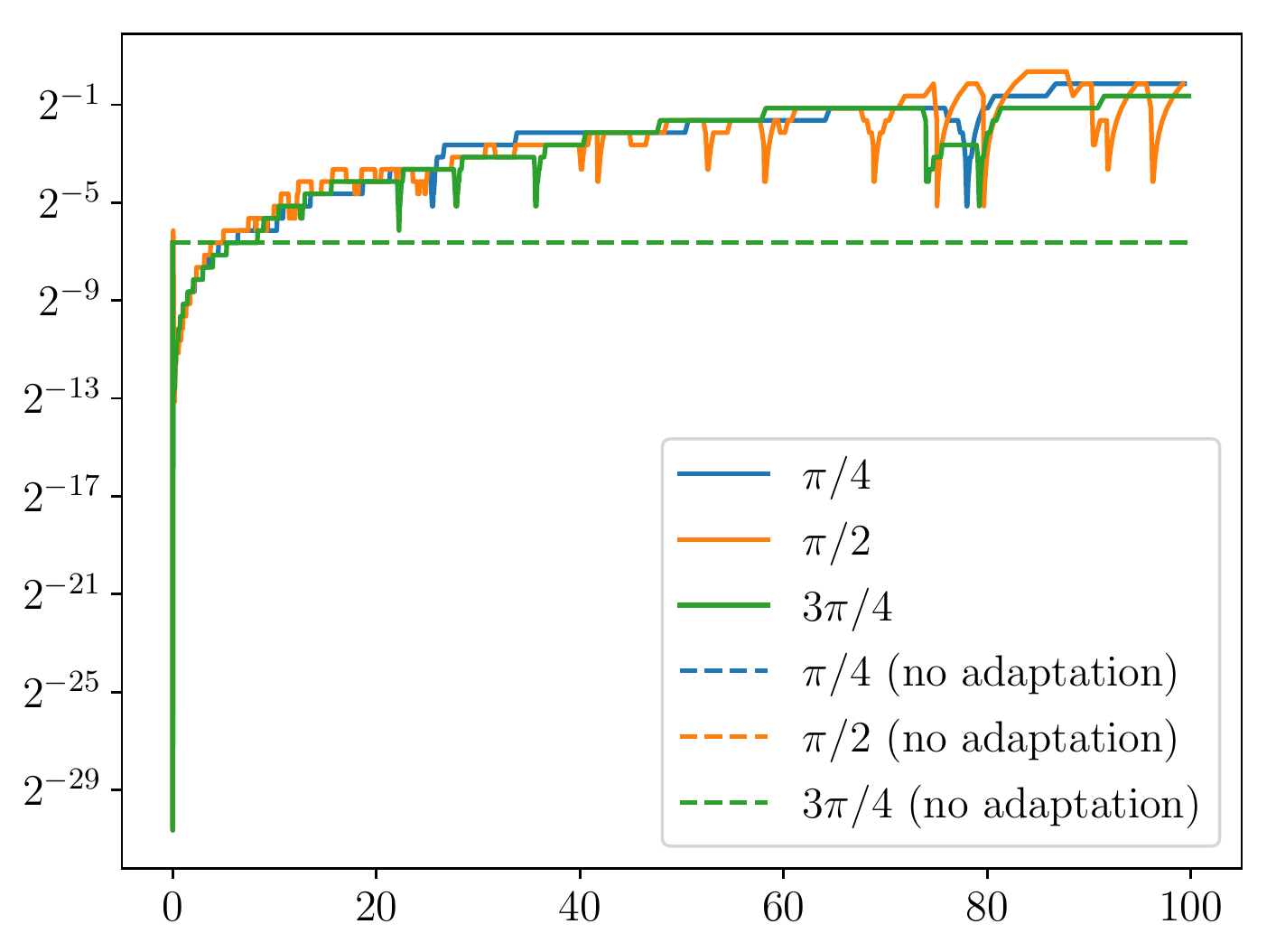}
            \subcaption{Time step (\cref{algorithm:timeStepAlgo}) vs time}
        \end{minipage}%
        \begin{minipage}[b]{.48\linewidth}
            \centering
            \includegraphics[width=\linewidth]{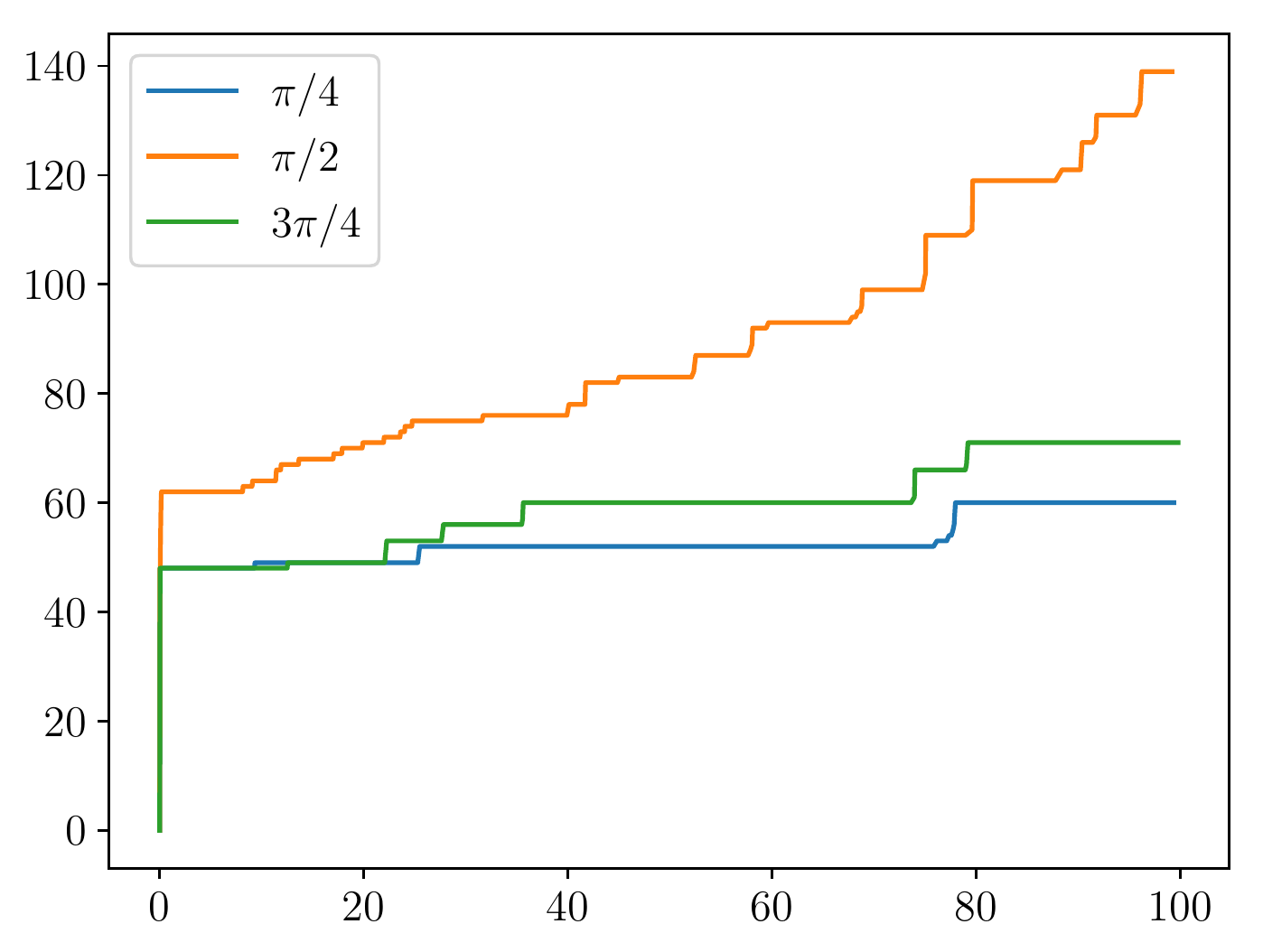}
            \subcaption{\# of recalculations vs time.}
        \end{minipage}

        \begin{minipage}[b]{.48\linewidth}
            \centering
            \includegraphics[width=\linewidth]{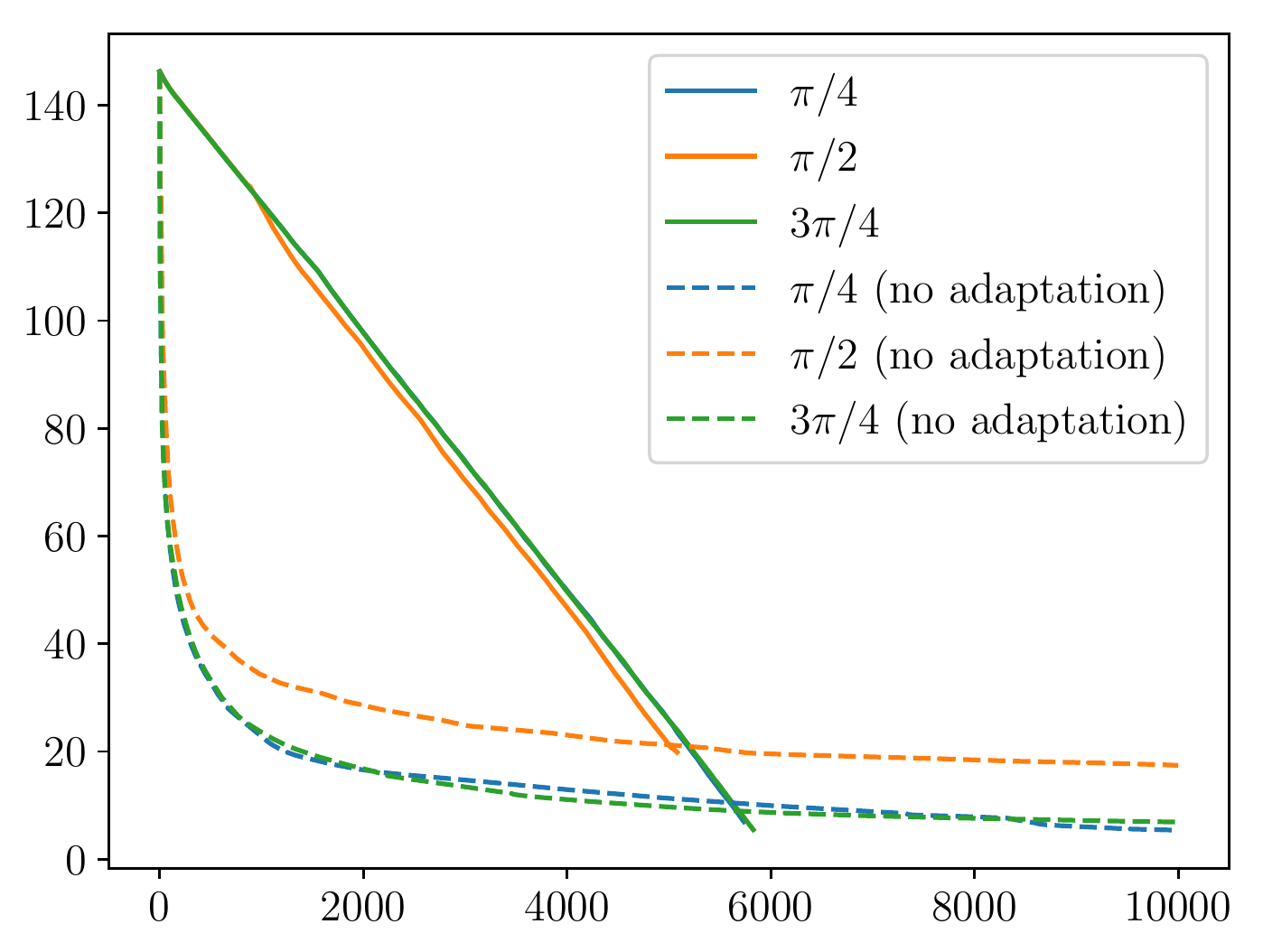}
            \subcaption{Total energy vs iteration.}
        \end{minipage}%
        \begin{minipage}[b]{.48\linewidth}
            \centering
            \includegraphics[width=\linewidth]{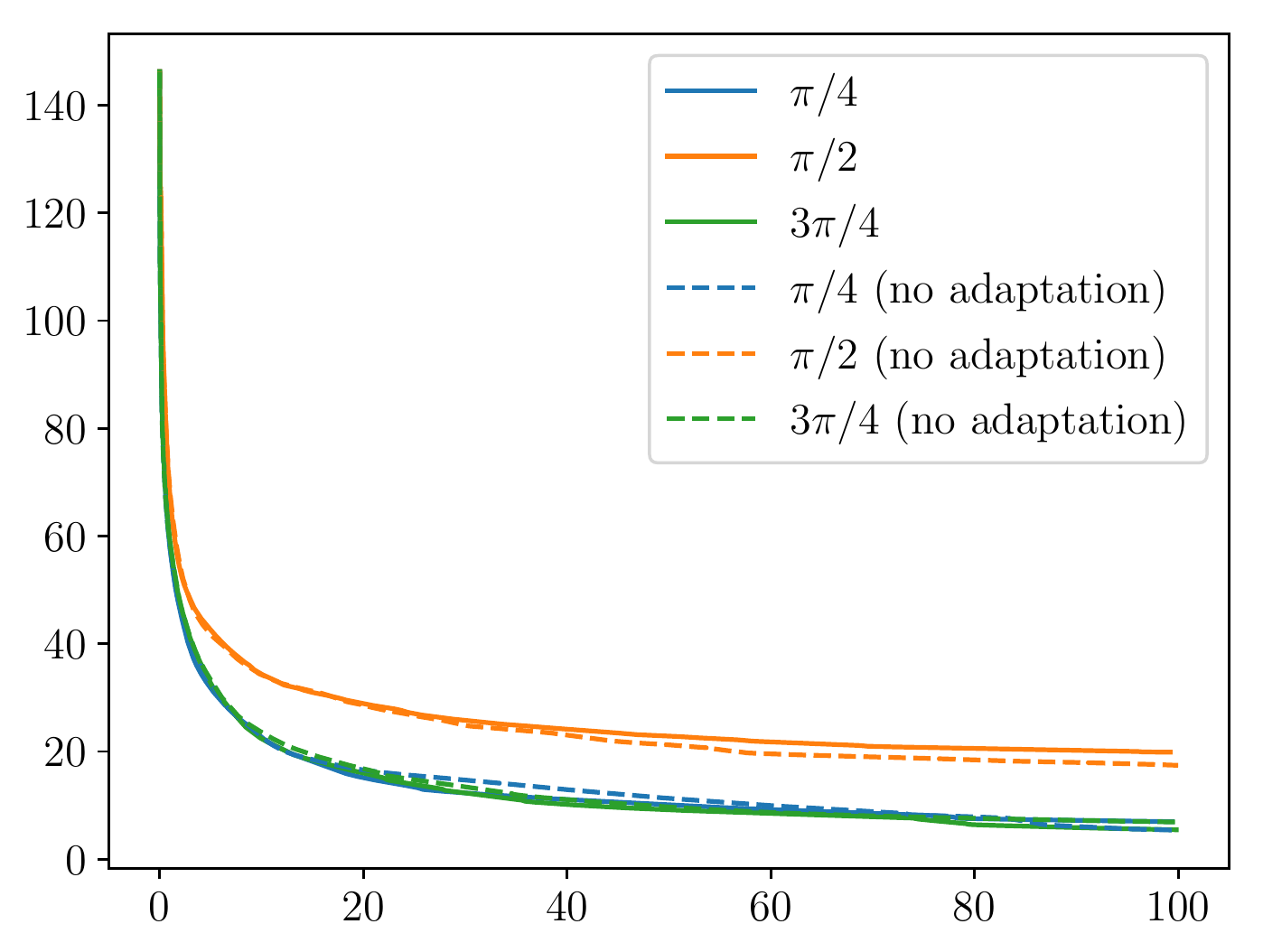}
            \subcaption{Total energy vs time.}
        \end{minipage}

        \begin{minipage}[b]{.48\linewidth}
            \centering
            \includegraphics[width=\linewidth]{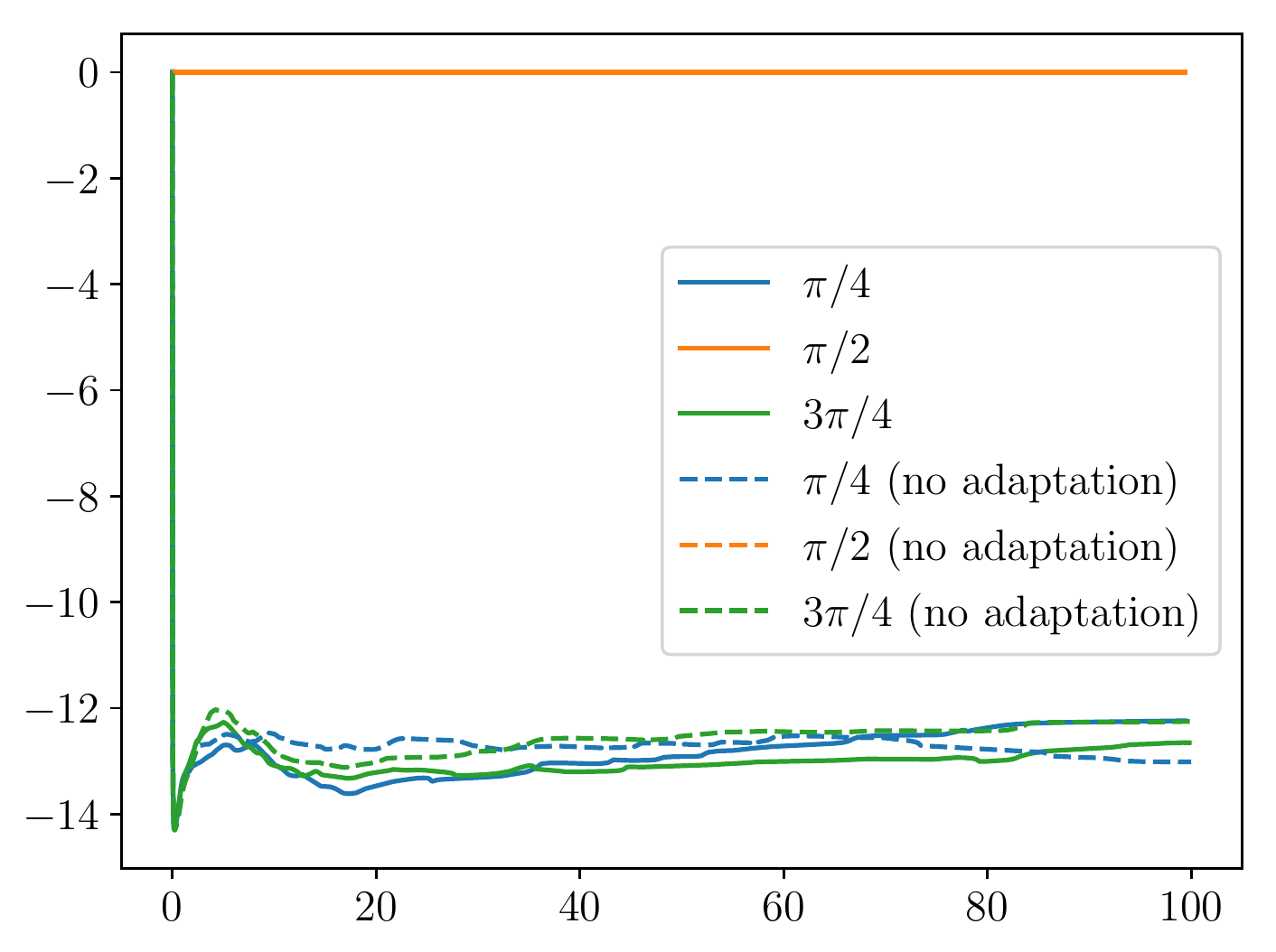}
            \subcaption{Wall energy vs time.}
        \end{minipage}%
        \begin{minipage}[b]{.48\linewidth}
            \centering
            \includegraphics[width=\linewidth]{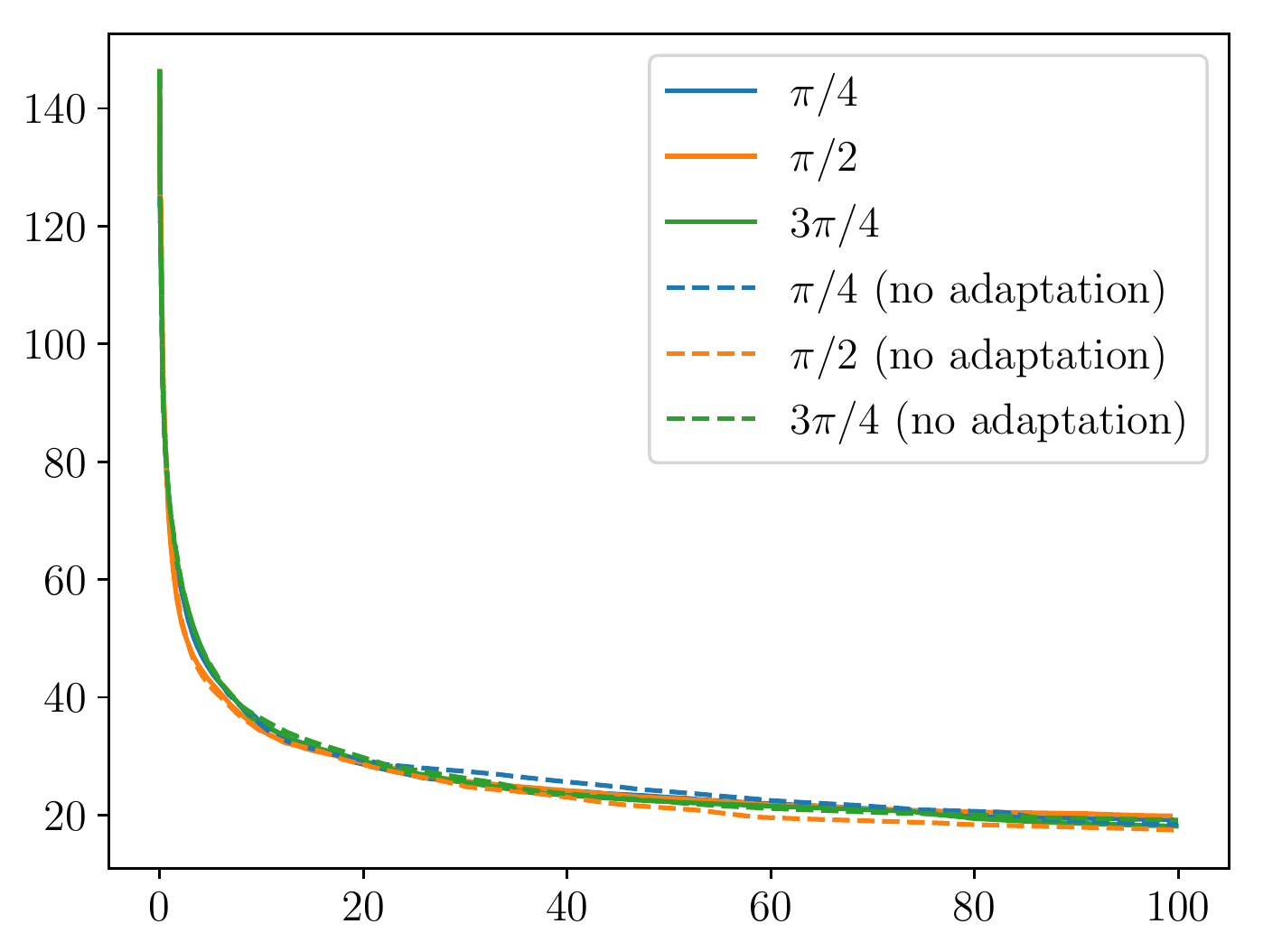}
            \subcaption{Mixing energy vs time.}
        \end{minipage}%
    \caption{%
        Simulation data for the numerical experiments presented in \cref{ssub:nucleation_processes_with_complex_boundaries} (nucleation in a geometry with complex boundaries).
        Overall, the time step increases steadily when the adaptive time-stepping scheme is used,
        which is consistent with the decreasing frequency of coalescence events.
        The time step is refined at times to ensure that the incremental decrease of free energy at each iteration is approximately constant.
    }
    \label{fig:nucleationMacro}
\end{figure}

\section{Conclusions}%
\label{sec:phase_field:conclusions}
We have proposed a new, fast and reliable numerical method to solve the CH equation with a wetting boundary condition.
Our method is a generalization of the OD2 scheme introduced in~\cite{guillen2013linear},
which considered only the homogeneous condition $\grad \phi \cdot \vect n = 0$.
In addition, we have designed a new time-step adaptation algorithm,
leading to a scheme that is adaptive both in space and time,
and we have shown that this scheme is mass-conservative and satisfies a consistent discrete energy law.

We checked the validity of the proposed numerical scheme with several examples.
First we considered the relaxation towards equilibrium of a sessile droplet and the coalescence of two sessile droplets
on flat, chemically homogeneous substrates;
then we considered several multiphase systems in complex geometries or surrounded by chemically heterogeneous substrates.

Compared to finite differences or spectral approaches,
the method introduced here has the advantage that
it can be used without modification with complex geometries.
Furthermore, the numerical scheme we have proposed can easily be extended to include at least two additional features.
First, a linear, energy-stable, second-order scheme could be developed for the
three-component CH model with wetting boundary conditions,
building on the work of~\cite{boyer2006study,boyer2011numerical}.

Second, we remark that in our work,
we considered a regime in which contact line motion is controlled by diffusive interfacial fluxes,
or in other words, we considered a large diffusivity limit,
where any possible advection effects are neglected.
To account for such effects the model must be appropriately modified to include
an advection term coupled to the Navier--Stokes equations~\cite{Andersonetal,Seppecher,jacqmin2000contact,JasnowVinals,Anderson,Wylock2012}.
Such generalizations are indeed possible within the proposed numerical scheme and
we hope to address these and related issues in future studies.

\section{Acknowledgements} We are grateful to Ben Goddard, Demetrios Papageorgiou and Srikanth Ravipati for useful suggestions.
We acknowledge financial support by the UK Engineering and Physical Sciences Research Council (EPSRC) through
Grants No.~EP/L027186 and EP/L020564 and European Research Council (ERC) via Advanced Grant No. 247031.

\appendix
\gdef\thesection{\Alph{section}} 
\makeatletter
\renewcommand\@seccntformat[1]{Appendix \csname the#1\endcsname.\hspace{0.5em}}
\makeatother

\section{Proof of \texorpdfstring{\cref{theorem:well-posedness_modified_boundary_condition}}{existence theorem}}
\label{sec:proof_of_well_posedness_theorem}
Before presenting the proof,
we recall a particular Sobolev embedding for smooth bounded domains;
see e.g.~\cite[Chap. 5]{evans2010partial}.
Let $d \geq 2$, $\emptyset \neq \Omega \subset \real^d$ be open with $C^1$ boundary,
and assume that $q < \infty$ if $d = 2$ or $q < p^*:= 2d/(d-2)$ if $d > 2$.
Then the following embedding is compact
\begin{align*}
    \quad  H^1(\Omega) \hookrightarrow L^{q}(\Omega).
\end{align*}

We also recall two other well-known compactness results;
see e.g.~\cite{lions1969quelques}.
Let $X,Y,Z$ be Banach spaces with a compact embedding $X \subset Y$ and a continuous embedding $Y \subset Z$.
Then the following embeddings are compact:
\begin{subequations}
    \begin{align}
        \label{eq:galerkin_splitting_lyons_lemma_lp}
        & \{u\in \lp{2}{0,T;X} |\, \derivative{1}[u]{t} \in \lp{2}{0,T;Z}\} \hookrightarrow \lp{2}{0,T;Y}, \\
        \label{eq:galerkin_splitting_lyons_lemma_cont}
        & \{u\in \lp{\infty}{0,T;X} |\, \derivative{1}[u]{t} \in \lp{2}{0,T;Z}\} \hookrightarrow C([0,T],Y).
    \end{align}
\end{subequations}
\begin{proof}
Without loss of generality, we assume that the mobility, $b$, is equal to $1$.
In the spirit of~\cite[Theorem 2]{elliott1996cahn}, we apply a Faedo--Galerkin approximation.
Let $\{\varphi_n\}_{n\in \nat}$ and $\{\lambda_n\}_{n\in \nat}$ denote the eigenfunctions and eigenvalues of
the Laplace operator with a homogeneous Neumann boundary condition, i.e.
\begin{equation*}
    \left\{
        \begin{aligned}
            &- \laplacian \varphi_n = \lambda_n  \, \varphi_n \quad &&\text{ in }\Omega, \\
            &\grad \varphi_n \cdot \vect n = 0 \quad &&\text{ in } \partial \Omega,
        \end{aligned}
        \right.
\end{equation*}
normalized such that
\begin{equation*}
    \int_{\Omega} \varphi_n \, \varphi_m \, \d \Omega = \delta_{mn}.
\end{equation*}
We assume without loss of generality that $\lambda_1 = 0$.
To build an approximation of the solution to
\cref{eq:variational_equations_for_weak_formulation_phi,eq:variational_equations_for_weak_formulation_mu} in the finite-dimensional space $S_N := \Span\{\varphi_1, \dots, \varphi_N\}$,
we consider the following ansatz,
\begin{equation*}
    \phi^N(t) = \sum_{n=1}^{N} a_n^N(t) \, \varphi_n, \qquad \mu^N(t) = \sum_{n=1}^{N} b_n^N(t) \, \varphi_n,
\end{equation*}
and the variational formulation
\begin{subequations}
    \begin{align}
        \label{eq:galerkin_splitting_u}
        &\ip{\partial_t \phi^N}{\bar \phi} + \ip{\grad \mu^N}{\grad \bar \phi} = \ip{\dot m}{\bar \phi}_{\partial \Omega}
        \quad &\forall \bar \phi \in S_N, \\
        \label{eq:galerkin_splitting_v}
        &\ip{\mu^N}{\bar \mu} = \energyB \ip{\grad \phi^N}{\grad \bar \mu} + \energyA \ip{f_m(\phi^N)}{\bar \mu} + \ip{f_w(\phi^N)}{\bar \mu}_{\partial \Omega}
        \quad &\forall \bar \mu \in S_N, \\
        \label{eq:galerkin_boundary_condition}
        &\ip{\phi^N(0)}{\bar \phi} = \ip{\phi_0}{\bar \phi}
        \quad &\forall \bar \phi \in S_N.
    \end{align}
\end{subequations}
To this formulation corresponds the following system of ordinary differential equations, with
unknown functions $\seq{a^N}{n}{1}{N}$ and $\seq{b^N}{n}{1}{N}$:
\begin{subequations}
    \begin{align}
        \label{eq:galerkin_splitting_u_differential_equation}
        &\derivative*{1}[(a^N_n)]{t} = \lambda_n \, b^N_n + \ip{\dot m}{\varphi_n}_{\partial\Omega}, \\
        \label{eq:galerkin_splitting_v_differential_equation}
        &b^N_n = \energyB \, \lambda_n \, a^N_n + \energyA \ip{\textstyle
    f_m\left(\sum_{i=1}^{N}a^N_i\,\varphi_i\right)}{\varphi_n} + \ip{\textstyle f_w\left(\textstyle \sum_{i=1}^N a^N_i \, \varphi_i \right)}{\varphi_n}_{\partial \Omega}, \\
        \label{eq:galerkin_boundary_condition_differential_equation}
        &a^N_n(0) = \ip{\phi_0}{\varphi_n},
    \end{align}
\end{subequations}
for $n = 1, \dots, N$.
%
%
Local existence and uniqueness of a solution to this system of equations is guaranteed by the
fact that that the right-hand side of~\eqref{eq:galerkin_splitting_u_differential_equation}
depends continuously on the coefficients $\seq{a^N}{n}{1}{N}$.
To show the existence of a global solution,
we will use the \emph{a priori} estimate presented in the following lemma.
\begin{lemma}
    Assume that $\abs{F_w(\phi)} \leq C(1 + \abs{\phi}^2)$.
    Then the solution $(\phi^N,\mu^N)$ to \cref{eq:galerkin_splitting_u,eq:galerkin_splitting_v,eq:galerkin_boundary_condition} satisfies
    \begin{equation}
        \label{eq:galerkin_splitting_energy_estimate_refined}
        \frac{1}{2} \int_{\Omega} \left(\frac{1}{2} \, \energyB \, \abs{\grad \phi^N}^2 + \energyA \, F_m(\phi^N) \right) \, \d \Omega
        + \frac{1}{2} \int_{\Omega_T} \abs{\grad \mu^N}^2 \leq C,
    \end{equation}
    where $C$ is independent of $N$ and $\Omega_T := \Omega \times (0, T)$.
\end{lemma}
\begin{proof}
Setting $\bar \phi = \mu^N$, $\bar \mu = \partial_t \phi^N$ in
~\cref{eq:galerkin_splitting_u,eq:galerkin_splitting_v} and subtracting leads to the equation
\begin{equation}
    \label{eq:galerkin_splitting_energy_estimate}
    \begin{aligned}
        \derivative*{1}{t} \left[ E_m(\phi^N) + E_w(\phi^N) \right]
        :=& \derivative*{1}{t} \left[ \int_{\Omega} \frac{1}{2} \, \energyB \, \abs{\grad \phi^N}^2 + \energyA \, F_m(\phi^N) \, \d \Omega + \int_{\partial \Omega} F_w(\phi^N) \, \dsurf \right] \\
        =& - \int_{\Omega} \abs{\grad \mu^N}^2 \d \Omega + \int_{\partial\Omega} \dot m \, \mu^N \, \dsurf.
    \end{aligned}
\end{equation}
Using a trace inequality,
H\"older's inequality,
and Young's inequality with a parameter,
we have, for all $u \in H^1(\Omega)$,
\begin{align*}
    \norm{F_w(u)}_{\lp{1}{\partial \Omega}}
    &\leq C \norm{1 + \abs{u}^2}_{\lp{1}{\partial \Omega}} \\
    &\leq C \, \left(1 + \norm{u^2}_{\lp{1}{\Omega}} + \norm{\grad (u^2)}_{\lp{1}{\Omega}}\right) \\
    &= C \, \left(1 + \norm{u^2}_{\lp{1}{\Omega}} + 2\norm{u \, \grad u}_{\lp{1}{\Omega}}\right) \\
    &\leq C \, \left(1 + \norm{u^2}_{\lp{1}{\Omega}} + \frac{1}{\alpha}\norm{u^2}_{\lp{1}{\Omega}} + \alpha\norm{\grad u}_{\lp{2}{\Omega}}^2\right) \qquad \forall \alpha > 0.
\end{align*}
Now we use the simple fact that,
for any $\beta > 0$ and $0 \leq s \leq t$,
the inequality $\abs{x}^s \leq \beta^s + \beta^{s-t} \, \abs{x}^t$ holds true for all $x \in \real$,
to obtain
\begin{align}
    \label{eq:galerkin_splitting_bound_boundary_energy}
    \norm{F_w(u)}_{\lp{1}{\partial \Omega}}
    \leq C + \frac{1}{2} E_m(u)
\end{align}
for a constant $C$ independent of $u$.

In addition, using a trace inequality,
Poincar\'e inequality,
and ~\eqref{eq:galerkin_splitting_v} with $\bar \mu = 1$,
\begin{align}
    \abs{\int_{\partial \Omega} \dot m \, \mu^N \, \dsurf}
    &\leq \abs{\int_{\partial \Omega} \dot m \, \left(\mu^N - \frac{1}{\abs{\Omega}}\int_{\Omega} \mu^N \,\d \Omega \right) \, \dsurf}
    + \abs{\frac{1}{\abs{\Omega}} \, \int_{\partial\Omega} \dot m \, \dsurf \int_{\Omega} \mu^N \, \d \Omega} \nonumber \\
    & \leq C \, \norm{\dot m}_{\lp{2}{\partial \Omega}} \norm{\grad \mu^N}_{\lp{2}{\Omega}}
    +  \frac{1}{\abs{\Omega}} \, \int_{\partial\Omega} \abs{\dot m} \, \dsurf \abs{ \energyA \ip{f_m(\phi^N)}{1} + \ip{f_w(\phi^N)}{1}_{\partial \Omega} } \nonumber \\
    \label{eq:galerkin_splitting_bound_boundary_mass_flux}
    & \leq C \, \norm{\dot m}_{\lp{2}{\partial \Omega}}^2 + \frac{1}{2}  \norm{\grad\mu^N}_{\lp{2}{\Omega}}^2 + C \, \norm{\dot m}_{\lp{2}{\partial \Omega}}\,(E_m(\phi^N) + 1).
\end{align}
Integrating~\eqref{eq:galerkin_splitting_energy_estimate} in time,
and rearranging using \cref{eq:galerkin_splitting_bound_boundary_energy,eq:galerkin_splitting_bound_boundary_mass_flux},
\begin{align*}
    \label{eq:galerkin_splitting_energy_estimate_refined_before_gronwall}
    \frac{1}{2} E_m(\phi^N(t)) + \frac{1}{2} \int_{\Omega_t} \abs{\grad \mu^N}^2
    &\leq C + \frac{3}{2} E_m(\phi^N(0)) + C \, \int_{\partial \Omega_T} \abs{\dot m}^2 + C \int_{0}^{t} \norm{\dot m}_{\lp{2}{\partial \Omega}} \, E_m(\phi^N) \, \d s, \\
    &\leq C + C\int_{0}^{t} \norm{\dot m}_{\lp{2}{\partial \Omega}} \, E_m(\phi^N(s)) \, \d s
\end{align*}
where we used the notations $\Omega_t$ and $\partial \Omega_t$, $t > 0$, to denote
$\Omega \times (0,t)$ and $\partial\Omega \times (0,t)$, respectively.
The last inequality holds by the assumptions that $\phi_0 \in H^1(\Omega)$ and $\dot m \in C([0,T];\lp{2}{\partial\Omega})$.
Using a Gr\"onwall inequality, we have \cref{eq:galerkin_splitting_energy_estimate_refined}.
\end{proof}

By integration by parts of the first term in
\cref{eq:galerkin_splitting_energy_estimate_refined}, we obtain
$\sum_{n=1}^{N} \lambda_n (a^N_n)^2 < C$. This result, together with the
inequality
\begin{equation*}
    \label{eq:galerkin_splitting_bound_on_average_phi}
    a_1^N(t) = a_1^N(0) + \int_0^T \int_{\partial\Omega} \dot m \, \dsurf \leq a_1^N(0) + C \, \norm{\dot m}_{C([0, T];\lp{2}{\partial\Omega}}
\end{equation*}
implied by \Cref{eq:galerkin_splitting_u_differential_equation} and the fact that $\lambda_1 = 0$,
show that the coefficients $\seq{a^N}{n}{1}{N}$ do not blow up,
and by \cref{eq:galerkin_splitting_v_differential_equation} neither do the coefficients $\seq{b^N}{n}{1}{N}$,
implying global existence.

In addition to~\eqref{eq:galerkin_splitting_energy_estimate_refined},
we have the usual estimate on $\derivative{1}[\phi^N]{t}$:
denoting by $\Pi^N$ the $L^2(\Omega)$ projection on $S_N$,
for all $\psi \in \lp{2}{0,T;H^1(\Omega)}$ the following holds:
\begin{align}
    \abs{\int_{\Omega_T} \derivative{1}[\phi^N]{t} \, \psi}
    &= \abs{\int_{\Omega_T} \derivative{1}[\phi^N]{t}\, (\Pi^N\psi)} \nonumber \\
    &= \abs{\int_{\Omega_T} \grad \mu^N \cdot \grad (\Pi^N\psi)}
    + \abs{\int_{\partial \Omega_T} \dot m \,(\Pi^N \psi)} \nonumber  \\
    &\leq \left(\int_{\Omega_T}  \abs{\grad \mu^N}^2\right)^{\frac{1}{2}}  \left(\int_{\Omega_T} \abs{\grad (\Pi^N \psi)}^2\right)^{\frac{1}{2}}
    + \left(\int_{\partial\Omega_T}  \abs{\dot m}^2\right)^{\frac{1}{2}}  \left(\int_{\partial\Omega_T} \abs{\Pi^N \psi}^2\right)^{\frac{1}{2}} \nonumber \\
    \label{eq:galerkin_splitting_energy_estimate_time_derivative}%
    &\leq \, C \norm{\psi}_{\lp{2}{0,T;H^1(\Omega)}}.
\end{align}
This shows that $\norm{\derivative{1}[\phi^N]{t}}_{\lp{2}{0,T;(H^1(\Omega))'}} \leq C$.

Let $p$ be such that the embedding
$H^1(\Omega) \subset \lp{p}{\Omega}$ is compact,
i.e., by Rellich--Kondrachov theorem,
$p < \infty$ if $d=1$ or $d=2$,
and $p < \frac{2d}{d-2}$ if $d>2$.
Using~\cref{eq:galerkin_splitting_energy_estimate_refined,eq:galerkin_splitting_energy_estimate_time_derivative},
we can apply results~\eqref{eq:galerkin_splitting_lyons_lemma_lp} and~\eqref{eq:galerkin_splitting_lyons_lemma_cont} to our case,
with $X = H^1(\Omega)$, $Y = \lp{p}{\Omega}$ and $Z=(H^1(\Omega))'$,
to conclude that there exists a subsequence such that
\begin{subequations}
\begin{align}
    \label{eq:galerkin_splitting_convergence_weakstar_phi_linf}
    \phi^N \to \phi \quad & \text{weak-* in } \lp{\infty}{0,T;H^1(\Omega)}, \\
    \label{eq:galerkin_splitting_convergence_weak_derphi_l2}
    \derivative{1}[\phi^N]{t} \to \derivative{1}[\phi]{t}  \quad & \text{weakly in } \lp{2}{0,T;(H^1(\Omega))'}, \\
    \label{eq:galerkin_splitting_convergence_strong_phi_cont}
    \phi^N \to \phi \quad & \text{strongly in } C([0,T],L^p(\Omega)), \\
    \label{eq:galerkin_splitting_convergence_strong_phi_lp}
    \phi^N \to \phi \quad & \text{strongly in } \lp{2}{0,T;\lp{p}{\Omega}}.
\end{align}
\end{subequations}
when $N \to \infty$.
In addition,
note that since $\phi^N$ is bounded in $\lp{\infty}{0,T;\lp{2}{\partial\Omega}}$,
there is a subsequence such that
$\phi^N \to v$ weak-* in $\lp{\infty}{0,T;\lp{2}{\partial\Omega}}$
for some function $v$ in that space,
and thus also $\phi^N \to v$ weakly in the coarser $\lp{2}{0,T;\lp{2}{\partial\Omega}}$.
But also $\phi^N \to \phi$ weakly in $\lp{2}{0,T;\lp{2}{\partial\Omega}}$,
because $\phi^N \to \phi$ weakly in $\lp{2}{0,T;H^1(\Omega)}$ and
by continuity of the trace operator
(indeed, an operator between Hilbert spaces that is continuous in the strong topologies,
is continuous in the weak ones too), so $v = \phi$.
The same reasoning can be applied to $f_w(\phi^N)$,
taking into account that $f_w: u \rightarrow f_w(u)$ is continuous on $\lp{2}{0,T;\lp{2}{\partial\Omega}}$,
to conclude
\begin{align}
    \phi^N \to \phi \quad & \text{weak-* in } \lp{\infty}{0,T;\lp{2}{\partial \Omega}}, \nonumber \\
    \label{eq:galerkin_splitting_convergence_at_the_boundary_fw}
    f_w(\phi^N) \to f_w(\phi) \quad & \text{weak-* in } \lp{\infty}{0,T;\lp{2}{\partial \Omega}}.
\end{align}

Regarding the chemical potential,
testing~\eqref{eq:galerkin_splitting_v} with $\bar \mu = 1$ implies that:
\begin{equation*}
    \int_{\Omega} \mu \, \d \Omega = \int_{\Omega} \energyA \, f_m(\phi^N) \, \d \Omega + \int_{\partial \Omega} f_w(\phi^N) \, \dsurf,
\end{equation*}
which,
together with the energy estimate~\eqref{eq:galerkin_splitting_energy_estimate_refined},
implies that
$\mu^N$ is bounded in $\lp{2}{0,T;H^1(\Omega)}$,
leading to the existence of a further subsequence such that
\begin{equation}
    \label{eq:galerkin_splitting_convergence_weak_mu_l2}
    \mu^N \rightarrow \mu \quad \text{weakly in } \lp{2}{0,T;H^1(\Omega)}.
\end{equation}
Proceeding in a standard fashion,
we consider an integer $M$ and arbitrary functions $ \phi^M,  \mu^M \in C([0,T],H^1(\Omega))$
such that
\begin{equation*}
    \phi^M = \sum_{n=1}^{M} \bar a_n(t) \, \varphi_n, \qquad  \mu^M = \sum_{n=1}^{M} \bar b_n(t) \, \varphi_n,
\end{equation*}
with $\{\bar a_n\}_{n=1}^M$, $\{\bar b_n\}_{n=1}^M$ smooth functions.
Using $\phi^M$ and $\mu^M$ as test functions in \cref{eq:galerkin_splitting_u,eq:galerkin_splitting_v},
integrating in time,
taking the limit $N \to \infty$,
and using the convergence results given in \cref{eq:galerkin_splitting_convergence_weak_derphi_l2,eq:galerkin_splitting_convergence_strong_phi_cont,eq:galerkin_splitting_convergence_strong_phi_lp,eq:galerkin_splitting_convergence_at_the_boundary_fw,eq:galerkin_splitting_convergence_weak_mu_l2,eq:galerkin_splitting_convergence_weakstar_phi_linf}, we obtain
\begin{subequations}
    \begin{align*}
        &\int_0^T \dup{\partial_t \phi}{\phi^M} \, \d t + \int_0^T \ip{\grad \mu}{\grad \phi^M} \,\d t = \int_0^T \ip{\dot m}{\phi^M}_{\partial \Omega} \, \d t, \\
        &\int_0^T \ip{\mu}{\mu^M} \, \d t = \int_0^T \energyB \ip{\grad \phi}{\grad  \mu^M} \, \d t  + \int_0^T \energyA \ip{f_m(\phi)}{\mu^M} \, \d t + \int_0^T \ip{f_w(\phi)}{\mu^M}_{\partial \Omega} \, \d t,
    \end{align*}
\end{subequations}
from which we conclude using a standard density argument.
\end{proof}

\bibliographystyle{abbrvnat}
\bibliography{main}

\end{document}